\documentclass[11pt]{amsart}
\usepackage{amsmath, amssymb, mathabx, enumitem}
\usepackage{color,mathdots}
\usepackage[colorlinks=true,linkcolor=blue,urlcolor=blue,citecolor=blue, pagebackref]{hyperref}
\usepackage[alphabetic]{amsrefs}

\usepackage{comment}
\usepackage{tikz-cd}
\usepackage[margin=0.9in]{geometry}

\newcommand{\A}{\mathbb{A}}
\newcommand{\C}{\mathbb{C}}
\newcommand{\Q}{\mathbb{Q}}
\newcommand{\Z}{\mathbb{Z}}
\newcommand{\wh}{\widehat}
\newcommand{\codim}{\text{codim}}
\newcommand{\op}{\operatorname}
\newcommand{\pic}{\op{Pic}}
\newcommand{\im}{\op{im}}
\newcommand{\wrt}{w.r.t.\ }

\newtheorem{theorem}{Theorem}[section]

\newtheorem{remark}[theorem]{ Remark}

\newtheorem{corollary}[theorem]{Corollary}

\newtheorem{obs}[theorem]{Observation}
\newtheorem{proposition}[theorem]{Proposition}

\newtheorem{lemma}[theorem]{Lemma}

\newtheorem{definition/lemma}[theorem]{Definition/Lemma}
\newtheorem{defi}[theorem]{Definition}

\title[Embedded subgroup extremal rays]{Extremal rays of the embedded subgroup saturation cone}
\author{Joshua Kiers}

\begin{document}

\begin{abstract}
We examine the extremal rays of the cone of dominant weights $(\mu, \wh\mu)$ for groups $G\subseteq \wh G$ for which there exists $N \gg0$ such that
$$
\left(V(N\mu)\otimes V(N\wh \mu)\right)^G\ne (0).
$$
We exhibit formulas for a class of rays (``type I'') on any regular face of the cone. These rays are identified thanks to a generalization of Fulton's conjecture, which we prove along the way. We verify that the remaining rays (``type II'') on the face are the images of extremal rays for a smaller cone under a certain map, whose formula is given. A procedure is given for finding the rays of the cone not on any regular face. This is a generalization of the work of Belkale and Kiers on extremal rays for the saturated tensor cone; the specialization is given by $\wh G = G\times G$ with the diagonal embedding of $G$. We include several examples to illustrate the formulas. 
\end{abstract}

\maketitle

\section{Introduction}

In this paper we extend the main results of \cite{BKiers} on the extremal rays of the saturated tensor cone. For a connected semisimple complex algebraic group $G$ and fixed maximal toral and Borel subgroups $T\subset B$, the saturated tensor cone $\mathcal{C}(G)$ consists of triples of dominant 
weights $\lambda, \mu,\nu:T\to \C^*$ such that $\lambda+\mu+\nu$ is in the root lattice for $G$ and the tensor product of irreducible representations 
$$
V(N\lambda)\otimes V(N\mu)\otimes V(N\nu)
$$
has a nontrivial subspace of $G$-invariants for some $N>0$. See \cite{Kumar} for a survey of the study of this cone, which has been studied extensively, since an original conjecture of Horn on eigenvalues of a sum of Hermitian matrices, with contributions from \cite{Kly, Bel, KTW, BeS, KLM, BK, Ress, RessII}. 

A more general setup is the following: let $\wh G$ be a connected semisimple complex algebraic group, and let $G\subseteq \wh G$ be a connected reductive subgroup.  Let $T\subseteq B$, $\wh T\subseteq \wh B$ be fixed maximal tori and Borel subgroups for $G$ and $\wh G$ satisfying $T\subseteq \wh T$ and $B\subseteq \wh B$.  The saturated tensor cone $\mathcal{C}(G\hookrightarrow \wh G)$ is the semigroup consisting of pairs of dominant (w.r.t. $B$, $\wh B$) weights $\mu, \wh \mu$ s.t. 
$$
\dim \left(V(N\mu)\otimes V(N\wh \mu)\right)^G > 0
$$
for some $N>0$. This cone was analyzed in \cite{BeS} and \cite{Ress,RessII}, and it is the natural extension of $\mathcal{C}(G)$ to a much broader range of examples and applications (see for example \cite{DYNKIN} for a comprehensive study of the possible embeddings $G\subseteq \wh G$).
When $G$ is diagonally embedded in $G\times G$, one recovers $\mathcal{C}(G) = \mathcal{C}(G\xrightarrow{\text{diag}}G\times G)$. 

Our main results are formulas for the extremal rays of the rational cone $\mathcal{C}(G\hookrightarrow \wh G)_\Q:=\mathcal{C}(G\hookrightarrow \wh G)\otimes_{\Z_{\ge 0}}\Q_{\ge 0}$, generalizing the formulas given in \cite{BKiers} for $\mathcal{C}(G)_\Q$ by adapting them to the complexities of the Lie combinatorics in the $G\hookrightarrow \wh G$ context. There are a few differences:
\begin{enumerate}[label=(\Roman*)]
\item Unlike in \cite{BKiers}, extremal rays of $\mathcal{C}(G\hookrightarrow \wh G)$ need not lie on a \emph{regular face} - that is, the locus where one of the Schubert calculus inequalities holds with equality. We only present formulas for rays on regular faces; however, the other rays are easy to check for: see Observation \ref{obviate} and the discussion preceding.

\item The formulas for extremal rays on a regular face $\mathcal{F}$ are most conveniently expressed and used when $\wh B$ is in good position relative to part of the data defining $\mathcal{F}$. This may not be the case a priori, but we can conjugate $\wh B$ suitably (depending on $\mathcal{F}$) to account for this; see Section \ref{changingtime}. In \cite{BKiers}, the choice $\wh B = B\times B$ is already in good position for every face, so this issue did not arise. 

\item  Underpinning the main results of \cite{BKiers} was the main theorem of \cite{BKR}:  a generalization of a conjecture of Fulton. We need a new (more general) case of this conjecture, so we prove it here. 

\item In the case $G\xrightarrow{\op{diag}} G\times G$, we could just as well have assumed $G$ is reductive. If $G' = [G,G]$ is the semisimple part of $G'$, then there is a natural map $\mathcal{C}(G)\twoheadrightarrow \mathcal{C}(G')$ given by the restriction of dominant weights of $T$ to $T\cap [G,G]$, and moreover the fibres of this map are readily identified. 

In the case $G\subseteq \wh G$ with both $G,\wh G$ reductive, there is still a natural map 
$$
\mathcal{C}(G\hookrightarrow \wh G) \twoheadrightarrow \mathcal{C}(G\cap [\wh G,\wh G] \hookrightarrow [\wh G,\wh G])
$$
with identifiable fibres. However, $G\cap [\wh G,\wh G]$ need not be semisimple. Therefore, while we can reduce to $\wh G$ semisimple, for full generality we only assume $G$ to be reductive. In practice this means that our type I ray formulas will include an additional parameter, as compared to \cite{BKiers}, describing the action of $Z^0(G)$. Thanks to a suggestion from P. Belkale, we can calculate this parameter using $T$-equivariant cohomology. 

\item In order for one to use the aforementioned $T$-equivariant cohomology formula, one needs to have a means of calculating not only cup products but also pullbacks. We outline a trick for ``approximating'' double Schubert polynomials in all types that turns out to be sufficient for such calculations. 
\end{enumerate}

\subsection{Facets of $\mathcal{C}(G\hookrightarrow \wh G)$} 

We make one more simplifying assumption on the embedding $G\hookrightarrow \wh G$: 
\begin{align}\label{assumB}
&\text{assume there is no nontrivial connected normal subgroup}\\\nonumber&\text{$N\trianglelefteq G$ such that $N\trianglelefteq \wh G$ as well}
\end{align} 
(equivalently, no nontrivial ideal of the Lie algebra $\mathfrak{g}$ is also an ideal of $\wh{\mathfrak{g}}$). Indeed, if such a subgroup $N$ exists, then one finds that $N,G/N,$ and $\wh G/N$ are reductive (in fact $\wh G/N$ is semisimple if $\wh G$ is) by examining the Lie algebras. Furthermore there is a natural isomorphism of cones 
$$
\mathcal{C}(G\hookrightarrow \wh G) \simeq \mathcal{C}(G/N\hookrightarrow \wh G/N)\times \mathcal{C}(N\hookrightarrow N),
$$
the latter factor being trivial to describe: $\mathcal{C}(N\hookrightarrow N) = \{(\lambda,\mu): \lambda=-\mu\}$.  As shown in \cite{Ress}, assumption (\ref{assumB}) is equivalent to the condition that $\mathcal{C}(G\hookrightarrow\wh G)$ have nonempty interior inside the ambient vector space of all rational weights $(\mu, \wh \mu)$. 

Now let $\delta:\C^*\to T$ be a one-parameter subgroup such that $\alpha(\delta)\ge 0$ for each positive root $\alpha$ of $G$; that is, $\delta$ is $G$-dominant. One defines a parabolic subgroup $P(\delta)\subseteq G$ by 
$$
P(\delta):=\{g\in G: \lim_{t\to 0} \delta(t) g\delta(t)^{-1}\text{~exists in~}G\}.
$$
The dominance assumption on $\delta$ ensures $B\subseteq P(\delta)$. Viewing $\delta$ naturally as a cocharacter of $\wh T$, one also defines the parabolic subgroup $\wh P(\delta)$ of $\wh G$ in the same way, although notably $\wh P(\delta)$ need not contain $\wh B$ as a subgroup. By definition, $P(\delta) = \wh P(\delta)\cap G$. 

There are associated Levi subgroups $L(\delta)\subseteq P(\delta)$ and $\wh L(\delta)\subseteq \wh P(\delta)$ defined by 
$$
L(\delta):=\{g\in G: \lim_{t\to 0} \delta(t) g\delta(t)^{-1}=g\},
$$
and similarly for $\wh L(\delta)$. Again $L(\delta) = \wh L(\delta)\cap G$. When context makes it clear, we may omit the reference to $\delta$ and simply write $P,\wh P, L, \wh L$. 

The cohomology rings $H^*(G/P)$ and $H^*(\wh G/\wh P)$ have distinguished bases given by the Schubert varieties: for $w\in W/W_\delta$, define $X_w^P:=\overline{BwP}\subseteq G/P$; similarly define $\wh X_{\wh w}^{\wh P}:= \overline{\wh B \wh w\wh P}\subseteq \wh G/\wh P$ for $\wh w\in \wh W/\wh W_{\delta}$ (here $W_\delta$ is the stabilizer subgroup of $\delta$ in $W$, similarly $\wh W_{\delta}$ in $\wh W$.) We write $X_w$ and $\wh X_{\wh w}$ when the reference to $P$, $\wh P$ is clear. The Schubert basis consists of the Poincar\'e duals, $[X_w]$ (resp., $[\wh X_{\wh w}]$), of the homology fundamental classes of the Schubert varieties. 
Moreover, the equivariant cohomology rings $H^*_T(G/P)$ and $H^*_{\wh T}(\wh G/\wh P)$ also have distinguished Schubert bases as modules over $H^*_T(pt)$ (respectively, $H^*_{\wh T}(pt)$). We denote the equivariant class of a Schubert variety by $[X_w]^T$ or $[\wh X_{\wh w}]^{\wh T}$.

Say a $G$-dominant one-parameter subgroup $\delta$ is \emph{indivisible} if it cannot be written $\delta = \bar{\delta}^n$ as the power of another such OPS. 
Following \cite{Ress}, a (nonzero) indivisible $G$-dominant one-parameter subgroup $\delta$ is called \emph{admissible} (or \emph{special} in \cite{Kumar}) w.r.t. $(G,\wh G)$ if the span $\C \dot\delta\subset \mathfrak{h} = \op{Lie}(T)$ is orthogonal to a hyperplane of $\mathfrak{h}^*$ spanned by a subset of $\op{Wt}_{\mathfrak{h}}(\wh{\mathfrak{g}}/\mathfrak{g})$, the set of $\mathfrak{h}$-weights of $\wh{\mathfrak{g}}/\mathfrak{g}$. Equivalently, $\C \dot\delta$ equals
 the common kernel of the $\mathfrak{h}$-weights of $\wh {\mathfrak{l}}(\delta)/\mathfrak{l}(\delta)$. Let $\mathfrak{S}$ denote the set of all admissible indivisible $G$-dominant one-parameter subgroups. It's easy to see $\mathfrak{S}$ is a finite set; moreover it is nonempty as $\op{Wt}_{\mathfrak{h}}(\wh{\mathfrak{g}}/\mathfrak{g})$ spans $\mathfrak{h}^*$ 
(this follows from our assumption (\ref{assumB}): by the proof of \cite[Proposition 12]{Ress},
$\mathfrak{h}\to \op{End}(\wh{\mathfrak{g}}/\mathfrak{g})$ is injective and induces a surjection from the abstract span of $\op{Wt}_{\mathfrak{h}}(\wh{\mathfrak{g}}/\mathfrak{g})$ to $\mathfrak{h}^*$.)

Let $\phi_\delta$ denote the induced map $G/P\to \wh G/\wh P$, and $\phi_\delta^*$ the corresponding pullback in (equivariant) cohomology. In \cite{RessRich}, Ressayre and Richmond define a deformed pullback 
$$
\phi_\delta^\odot: H^*(\wh G/\wh P; \odot_0)\to H^*(G/P; \odot_0)
$$
which is a ring homomorphism for the Belkale-Kumar deformed product in cohomology of flag varieties \cite{BK}. 

We recall now the theorem of Ressayre \cite{Ress,RessRich} describing the cone $\mathcal{C}(G\hookrightarrow \wh G)$ with a minimal set of inequalities:  
\begin{theorem}\label{christus}
A pair of dominant weights $\mu,\wh \mu$ is in $\mathcal{C}(G\hookrightarrow \wh G)$ if and only if for every $\delta\in \mathfrak{S}$ and every pair $w,\wh w\in W/W_\delta\times \wh W/\wh W_\delta$ such that 
\begin{align}\label{prod}
\phi_{\delta}^\odot\left([\wh X_{\wh w}]\right)\odot_0 [X_w] = [X_e]
\end{align}
in $H^*(G/P;\odot_0)$, the inequality
$$
\mu(w\dot \delta)+\wh \mu(\wh w\dot \delta)\le 0 
$$
holds. Furthermore, no inequalities may be removed from this list. 
\end{theorem}

Thus for $\delta\in \mathfrak{S}$ and $w,\wh w$ satisfying (\ref{prod}), we may define the \emph{regular facet} $\mathcal{F}(w,\wh w,\delta)$ of $\mathcal{C}(G\hookrightarrow \wh G)$ to be 
$$
\mathcal{F}(w,\wh w,\delta) = \{\mu,\wh \mu\in \mathfrak{h}^*_{\Z,+}\times \wh{\mathfrak{h}}^*_{\Z,+}: \mu(w\dot \delta)+\wh \mu(\wh w\dot \delta)=0\} \cap \mathcal{C}(G\hookrightarrow \wh G),
$$
where $\mathfrak{h}^*_{\Z,+}$ denotes the set of dominant weights for $G$ w.r.t. $B$, and $\wh{\mathfrak{h}}^*_{\Z,+}$ sim. for $\wh G$ w.r.t. $\wh B$. It is a face of codimension one, not equal to one of the facets determining the dominant chamber.

In the sequel, we fix a regular facet 
of the cone and study its extremal rays. Of course there could be (a priori) other extremal rays of $\mathcal{C}(G\hookrightarrow \wh G)$. (In the case of $G\xrightarrow{\op{diag}}G\times G$, this was not so, see \cite[Lemma 37]{BKiers}.) However, these extraneous extremal rays are only of a certain type:

\begin{obs}\label{obviate}
If $(\mu,\wh \mu)$ gives an extremal ray of $\mathcal{C}(G\hookrightarrow \wh G)$ and does not belong to any regular facet, then $\mu=0$ and, up to scaling, $\wh \mu$ is a fundamental dominant weight.
\end{obs}

We discuss these extraneous rays in Section \ref{extraz}, culminating in the following theorem, which decreases the required inequalities for verifying whether a candidate $(0,\wh \omega_j)$ is an extremal ray. Here $\mathfrak{T}$ is a finite set (defined precisely in \S12) of indivisible one-parameter subgroups containing $\mathfrak{S}$; moreover $\mathfrak{T} = \mathfrak{S}$ if $\op{Wt}_{\mathfrak{h}}(\wh{\mathfrak{g}}/\mathfrak{g}) = \op{Wt}_{\mathfrak{h}}(\wh{\mathfrak{g}})$. 
\begin{theorem}
The following are equivalent:
\begin{enumerate}[label=(\alph*)]
\item  the ray given by $(0,\wh \omega_j)$ is extremal;
\item $(0,\wh\omega_j)\in \mathcal{C}(G\hookrightarrow \wh G)$;
\item the inequality
$$
\wh \omega_j(\wh w \dot \delta)\le 0
$$
holds for every $\delta\in \mathfrak{T}$ and $\wh w\in \wh W$ such that $\phi_\delta^\odot[\wh X_{\wh w}] = [X_e]$. 
\end{enumerate}
\end{theorem}

\subsection{Change of basis on a regular facet}\label{changingtime}

Suppose $\delta, w,\wh w$ are given as above satisfying (\ref{prod}); that is, $\delta, w, \wh w$ are the data of a regular facet $\mathcal{F}$. The theorems and formulas in the remainder of the paper are easier to describe if $P(\delta)$, $\wh P(\delta)$ are \emph{both} standard parabolics (we are only guaranteed $P(\delta)$ is). To accommodate this, we introduce a specific change of basis on $\wh{\mathfrak{h}}^*$ induced by an element of $\wh W$. Namely, let $\wh v\in \wh W$ satisfy:
\begin{enumerate}[label = (\alph*)]
\item[(H1)] $\wh v\delta$ is dominant \wrt $\wh B$;
\item[(H2)] $\wh v$ has minimal length (\wrt $\wh B$) among all elements satisfying (H1).
\end{enumerate}
Note that $\wh v\delta$ is uniquely determined by $\delta$.

\begin{proposition}\label{moveme}
Set $\wh B':=\wh v^{-1}\wh B\wh v$. Then
\begin{enumerate}[label=(\alph*)]
\item $\delta$ is dominant \wrt $\wh B'$; therefore $\wh B'\subseteq \wh P(\delta)$;
\item $\wh \mu$ is a dominant weight \wrt $\wh B$ $\iff$ $\wh v^{-1} \wh \mu$ is dominant \wrt $\wh B'$; therefore the set $\{\wh \omega_j':=\wh v^{-1} \wh \omega_j\}$ consists of the fundamental weights \wrt $\wh B'$;
\item $B\subseteq \wh B'$;
\item $\phi_\delta^\odot([\wh X'_{\wh v^{-1}\wh w}])\odot_0([X_w]) = [X_e]$ in $H^*(G/P;\odot_0)$, where $\wh X'_{\wh u}$ denotes the subvariety $\overline{\wh B' \wh u\wh P}\subset \wh G/\wh P$ for any $\wh u\in \wh W$.
\end{enumerate}

\end{proposition}

See Section \ref{prelims} for a short proof. 

{\bf Therefore we will always assume, in the remainder of this paper, that $P$ and $\wh P$ are both standard parabolics relative to the given pair of Borels $B\subseteq \wh B$.} 
For an example of changing bases, see Section \ref{illustrious}.

\subsection{Type I rays}

Suppose $\delta, w,\wh w$ satisfy (\ref{prod}), $\delta$ not necessarily in $\mathfrak{S}$. Define an associated universal intersection scheme 
$$
\mathcal{X} = \{(g,\wh g,z)\in G/B\times \wh G/\wh B\times \wh G/\wh P : z\in \phi_\delta(gX_w)\cap \wh g\wh X_{\wh w}\}.
$$
By the cup product assumption, $X_w$ and $\phi_{\delta}^{-1}(\wh X_{\wh w})$ generically meet in a single point. Indeed, the natural map $\pi:\mathcal{X}\to G/B\times \wh G/\wh B$ is birational \cite[Corollary 5.3]{BKR}. It may be possible, then, to construct $G$-invariant divisors on $G/B\times \wh G/\wh B$ (which may, via the Borel-Weil correspondence, give rise to extremal rays of $\mathcal{C}(G\hookrightarrow \wh G)_\Q$) by first constructing $G$-invariant divisors on $\mathcal{X}$. We now make this precise. 

Suppose either $v\xrightarrow{\alpha_\ell} w$ or $v\xrightarrow{\wh\alpha_\ell}\wh w$ for some $\ell$, where in either Weyl group we take $u\xrightarrow{\gamma}u'$ to mean $u' = s_\gamma u$ and $\ell(u')=\ell(u)+1$. Then define 
$$
\tilde D(v) = \{(g,\wh g,z)\in G/B\times \wh G/\wh B\times \wh G/\wh P : z\in \phi_\delta(gX_u)\cap \wh g\wh X_{\wh u}\},
$$
where $u=v, \wh u= \wh w$ or $u=w, \wh u = v$, depending on the case above. Let $D(v)$ be the image of $\tilde D(v)$ in $G/B\times \wh G/\wh B$. Our first main theorem concerns the properties of $D(v)$:

\begin{theorem}\label{divisors}
Set $D=D(v)$.
\begin{enumerate}[label = (\alph*)]
\item $D$ is a closed, codimension $1$ subvariety of $G/B\times \wh G/\wh B$.
\item $H^0(G/B\times \wh G/\wh B,\mathcal{O}(mD))^G$ is $1$-dimensional for all $m\ge 0$.
\item Writing $\mathcal{O}(D) = \mathcal{L}_\mu\boxtimes \mathcal{L}_{\wh \mu}$, $\Q_{\ge0}(\mu, \wh \mu)$ gives an extremal ray of $\mathcal{C}(G\hookrightarrow \wh G)_\Q$.
\item $(\mu,\wh \mu)$ lies on $\mathcal{F}(w,\wh w,\delta)$. 
\end{enumerate}
\end{theorem}

Let $\vec\mu(D(v))$ denote the pair $\mu,\wh \mu$ induced by $D$. We also give an explicit formula for $\vec\mu(D(v))$ in the following basis. A basis for $\pic^{G\times \wh G}(G/B\times \wh G/\wh B)$ is given by the set
\begin{align}\label{aset}
\{\mathcal{L}_{\omega_i}\boxtimes \mathcal{O}\}\cup \{\mathcal{O}\boxtimes \mathcal{L}_{\wh \omega_j}\}\cup\{\mathcal{L}_{\chi_k}\boxtimes \mathcal{O}\},
\end{align}
where $\{\chi_k\}$ is any basis for the character group of $Z^0(G)$ (the identity component of the center).  Moreover, the restriction map
$$
\pic^{G\times \wh G}(G/B\times \wh G/\wh B) \to \pic^G(G/B\times \wh G/\wh B)
$$
induced by the diagonal homomorphism $G\to G\times \wh G$ is an isomorphism since every line bundle on $\wh G/\wh B$ comes with a canonical $\wh G$-linearization. Therefore the set (\ref{aset}) gives a basis for $\pic^G(G/B\times \wh G/\wh B)$.

\begin{theorem}\label{formulaONE}
Write $\mu = \chi + \sum_{k=1}^r c_k \omega_k$ and $\wh \mu = \sum_{k=1}^{\wh r} \wh c_k \wh \omega_k$ in the respective bases of fundamental weights, where $\chi$ is a character of $Z^0(G)$. 
\begin{enumerate}[label = (\alph*)]
\item Then $c_k$ is the intersection number $c$ in 
$$
\displaystyle\phi_{\delta}^*\left([\wh X_{\wh u}]\right)\cdot [X_{s_{\alpha_k}u}] = c[X_e]
$$
if $s_{\alpha_k}u\in W^P$ and is of length $\ell(u)+1$, and $0$ otherwise. Likewise, $\wh c_k$ is the intersection number $c$ in 
$$
\displaystyle\phi_{\delta}^*\left([\wh X_{s_{\wh \alpha_k}\wh u}]\right)\cdot [X_{u}] = c[X_e]
$$
if $s_{\alpha_k}\wh u\in \wh W^{\wh P}$ and is of length $\ell(\wh u)+1$, and $0$ otherwise. 

\item Furthermore, 
$$
\mu+\wh \mu|_T = \int_{G/P}[X_u]^T\cdot \phi_\delta^* [\wh X_{\wh u}]^{\wh T},
$$
which allows one to calculate $\chi$ once all the $c_k,\wh c_k$ are known. 
\end{enumerate}
\end{theorem}

An extremal ray $\Q_{\ge0}(\mu,\wh \mu)$ of $\mathcal{F}_\Q$ is to be called ``type I'' if, for some simple root $\beta$ satisfying $v\xrightarrow{\beta} w$ (resp., $v\xrightarrow{\beta} \wh w$), $\mu(\beta^\vee)>0$ (resp., $\wh \mu(\beta^\vee)>0$). Thus the rays induced by $D(v)$ as above are type I (cf. Lemma \ref{ones}). 

\subsection{Type II rays}

Unsurprisingly, we call an extremal ray $\Q_{\ge0}(\mu,\wh \mu)$ of $\mathcal{F}_\Q$ ``type II'' if for every such $\beta$, $\mu(\beta^\vee)=0$ (resp., $\wh \mu(\beta^\vee)=0$). These vanishing equalities determine a sub-semigroup $\mathcal{F}_2$ inside $\mathcal{F}$ and a subcone $\mathcal{F}_{2,\Q}$ inside $\mathcal{F}_\Q$; the type II rays of $\mathcal{F}_\Q$ are by definition the extremal rays of $\mathcal{F}_{2,\Q}$. One of our theorems is that the rays $D(v)$, together with the type II rays, do indeed generate all of $\mathcal{F}$:

\begin{theorem}\label{1+2}
Let $\{\delta_1,\hdots,\delta_q\}$ be the collection of type I rays $\vec\mu(D(v))$. Then the addition map
$$
\prod_{b=1}^q \Z_{\ge 0} \delta_b \times \mathcal{F}_2 \to \mathcal{F}
$$
is an isomorphism of semigroups. Over $\Q$, it is an isomorphism of rational cones. 
\end{theorem}

We also give a formula for finding extremal rays of $\mathcal{F}_{2,\Q}$. Define a map $\text{Ind}:\mathfrak{h}_{L/\langle \delta\rangle}^*\times \wh {\mathfrak{h}}_{\wh L/\langle \delta\rangle}^*\to \mathfrak{h}^*\times \wh{\mathfrak{h}}^*$ as follows. For a pair $(\nu,\wh \nu)\in \mathfrak{h}_{L/\langle \delta\rangle}^*\times \wh {\mathfrak{h}}_{\wh L/\langle \delta\rangle}^*$, first pull back each of $\nu, \wh \nu$ to elements of $\mathfrak{h}^*$, $\wh{\mathfrak{h}}^*$, respectively, as characters vanishing on $\delta$. Denoting these elements again by $\nu, \wh \nu$, define
$$
\text{Ind}: (\nu, \wh \nu)\mapsto (w\nu, \wh w \wh \nu) - \sum_{v\xrightarrow{\alpha_\ell}w} w\nu(\alpha_\ell^\vee) \vec\mu(D(v)) - \sum_{v\xrightarrow{\wh\alpha_\ell}\wh w} \wh w\wh \nu(\wh \alpha_\ell^\vee) \vec\mu(D(v)). 
$$
We then prove 

\begin{theorem}\label{yadhtrib}
$\operatorname{Ind}$ restricts to a surjection of cones 
$$
\operatorname{Ind}:\mathcal{C}(L/\langle \delta\rangle\to \wh L/\langle \delta\rangle)_\Q \to \mathcal{F}_{2,\Q}.
$$
\end{theorem}

In particular, every extremal ray of $\mathcal{F}_{2,\Q}$ is the image of an extremal ray of the lower-dimensional cone $\mathcal{C}(L/\langle \delta\rangle\to \wh L/\langle \delta\rangle)_\Q$. However, $\operatorname{Ind}$ may not be injective and also may not take all extremal rays to extremal rays. 

Lastly, we derive an identity relating $c = \dim (\ker \op{Ind})$ and $q$, the number of type I rays (see also \cite[Proposition 63]{BKiers}):
\begin{proposition}\label{63}
$c = q- |\wh \Delta| + | \Delta(\wh P)|. $
\end{proposition}

\subsection{Generalized Fulton's conjecture}

In fact, Theorem \ref{divisors}(b) follows almost immediately from the following result. For an arbitrary Schubert variety $X_w$, there is a maximal subgroup (a standard parabolic) $Q_w\subseteq G$ which stabilizes it; set $Y_w = Q_wwP\subseteq X_w$. Similarly define $\wh Q_{\wh w}, \wh Y_{\wh w}$. Analogous to $\mathcal{X}$, define $\mathcal{Y}$ by replacing $X_w$ with $Y_w$, $\wh X_{\wh w}$ with $\wh Y_{\wh w}$. Let $\mathcal{R}$ be the ramification divisor of the birational map $\pi:\mathcal{Y}\to G/B\times \wh G/\wh B$ (note that $\mathcal{Y}$ is smooth). 

\begin{theorem}\label{BFCi}
For every $n\ge 1$, 
$\dim H^0(\mathcal{Y},\mathcal{O}(n\mathcal{R}))^G = 1$.
\end{theorem}

This has a representation-theoretic interpretation, thanks to the following isomorphism. Define weights $\chi_w = \rho - 2\rho_L+w^{-1}\rho$, $\chi_{\wh w} = \wh \rho - 2 \rho_{\wh L} + \wh w^{-1}\wh\rho$. Then 

\begin{theorem}\label{BFCii} For every $n\ge 1$, 
$$
H^0(\mathcal{Y},\mathcal{O}(n\mathcal{R}))^G\simeq \left[V(n(\chi_w-\chi_1))^*\otimes V(n\chi_{\wh w})^*\right]^{L}.
$$
\end{theorem}

Combined, Theorems \ref{BFCi} and \ref{BFCii} generalize Fulton's conjecture for Littlewood-Richardson coefficients, whose history we recall briefly: let $G = GL(r)$ and $\lambda, \mu$ be dominant weights for a maximal torus w.r.t. a chosen Borel subgroup. The Littlewood-Richardson coefficients $c_{\lambda, \mu}^\nu$ are defined by the decomposition of $G$-representations
$$
V(\lambda)\otimes V(\mu) = \bigoplus_{\nu} V(\nu)^{c_{\lambda,\mu}^\nu}.
$$
The original conjecture is
\begin{theorem}\label{origami}
If $c_{\lambda, \mu}^\nu = 1$, then $c_{n\lambda, n\mu}^{n\nu}=1$ for all $n\ge 1$.
\end{theorem}
It was first proven by Knutson, Tao, and Woodward in \cite{KTW}. 

The obvious extension to other groups fails, but the following generalization of Belkale, Kumar, and Ressayre \cite{BKR} holds, where the ``$c_{\lambda,\mu}^\nu=1$'' of Theorem \ref{origami} is reinterpreted as an intersection number:
\begin{theorem}\label{babble}
Let $G$ be any connected reductive group and $P$ any standard parabolic subgroup. For any $w_1, \hdots, w_s\in W^P$ such that 
$$
[X_{w_1}]\odot_0 \cdots \odot_0 [X_{w_s}] = 1[X_e]
$$
in $H^*(G/P; \odot_0)$, we have, for every $n\ge 1$,
$$
\dim[V(n(\chi_{w_1}-\chi_1))\otimes \cdots \otimes V(n\chi_{w_s})]^{L} = 1.
$$
\end{theorem}

Theorems \ref{BFCi}, \ref{BFCii} imply that (upon taking duals) for all $n\ge 1$, $\dim \left[V(n(\chi_w-\chi_1))\otimes V(n\chi_{\wh w})\right]^{L}=1$.  Thus we generalize Theorem \ref{babble} further to the setting of $G\subseteq \wh G$, and one
recovers it by considering the diagonal embedding $G\to \underbrace{G\times \cdots \times G}_{s-1}$. Many of the proofs are similar, but we highlight that the $x_P$-filtration on tangent spaces in \cite[\S 7]{BKR} is replaced by the more natural $\delta$-filtration in our setting; see Section \ref{highlighter}. The stabilizing parabolics $Q_w$ associated to Schubert varieties $X_w$ and the subvarieties $Y_w$ continue to play a crucial role. 

\subsection{Layout of the paper}

Because of its importance to the main results of this paper (the rays formulas), we will first establish the generalized Fulton conjecture (Theorems \ref{BFCi} and \ref{BFCii}) in Section \ref{nextgen}. We will then prove Theorem \ref{divisors} on the existence of the divisors giving rise to type I rays (Sections \ref{divI} and \ref{nextup}) and Theorem \ref{formulaONE} for the type I ray formulas (Section \ref{chemistry}) in succession. Next, we prove the decomposition Theorem \ref{1+2} (Section \ref{redox}) and the induction Theorem \ref{yadhtrib} (Sections \ref{zoom}, \ref{wellordering}, \ref{chemII}). Finally, we prove Proposition \ref{63} in Section \ref{help} and discuss the extraneous extremal rays in Section \ref{extraz}.

We end with a few examples in Section \ref{ampleXample}, some of which were first considered in \cite{BeS} or \cite{PR}; in general there is a wealth of branching situations $G\hookrightarrow \wh G$ one could consider. 

\subsection{Acknowledgements} The author is deeply indebted to the guidance of Prakash Belkale, who suggested the undertaking of this project and participated in many helpful discussions with the author. Furthermore the author thanks Shrawan Kumar for clarifying some aspects of the cone $\mathcal{C}(G\hookrightarrow \widehat G)$, David Anderson for lending his expertise in equivariant cohomology, and a referee for several insightful comments.

\section{Some preliminary comments on the cone $\mathcal{C}(G\hookrightarrow \wh G)$}\label{prelims}

Here we justify that changing basis on a regular facet $\mathcal{F}(w, \wh w,\delta)$ of $\mathcal{C}(G\hookrightarrow \wh G)$ is allowable by proving Proposition \ref{moveme}. 

\begin{proof}
Let $\wh \Delta$ denote the base for $\wh B$. Since $\wh v^{-1}\wh \Delta$ is the base for $\wh B'$, (a) and (b) follow immediately by definitions. 

As for (c), examine the embedding on the level of Lie algebras: $\mathfrak{b}\subseteq \wh{\mathfrak{b}}$ is an $\mathfrak{h}$-equivariant inclusion, so if $\gamma$ is a positive root for $B$, then $\displaystyle \mathfrak{g}_\gamma\subseteq \bigoplus_{\wh\gamma|_\mathfrak{h} = \gamma} \wh{\mathfrak{g}}_{\wh \gamma}$. Furthermore, the sum on the right is actually just over the roots $\wh \gamma$ which are positive for $\wh B$. We wish to show that any such $\wh\gamma$ on the RHS is actually positive \wrt $\wh B'$; equivalently, that $\wh v\wh \gamma$ is positive \wrt $\wh B$. 

To that end, consider the two possible cases: if $\langle \gamma, \dot\delta\rangle =\langle \wh \gamma, \dot \delta\rangle >0$, then $\langle \wh v\wh\gamma, \wh v \dot \delta\rangle >0$. Since $\wh v \delta$ is $\wh B$-dominant, we must have $\wh v \wh\gamma\succ 0$. On the other hand, if $\langle \gamma, \dot\delta \rangle = \langle \wh \gamma, \dot \delta \rangle = 0$, then $s_{\wh \gamma}\delta = \delta$. If $\wh v\wh \gamma \prec 0$, then $\wh v s_{\wh \gamma}$ has strictly smaller length than $\wh v$. 
But $\wh v s_{\wh \gamma}$ satisfies (H1) since $\wh v \delta = \wh v s_{\wh \gamma} \delta$, so this contradicts (H2). 

As for (d), there are two statements to prove (see \cite{RessRich} 
for more on the deformed pullback). We must show that $\phi_\delta^*([\wh X'_{\wh v^{-1}\wh w}])\cdot([X_w]) = [X_e]$ under the usual cup product, and secondly that 
$$
\langle \rho+w^{-1}\rho, \dot\delta\rangle - \langle 2\rho, \dot\delta\rangle + \langle \wh\rho'+\left(\wh v^{-1}\wh w\right)^{-1} \wh \rho',\dot\delta\rangle = 0,
$$
where $\rho$ is the half-sum of positive roots of $B$ and $\wh\rho'$ the same for $\wh B'$. 

The first follows immediately from the given product $\phi_\delta^*([\wh X_{\wh w}])\cdot([X_w]) = [X_e]$ and the
observation that $[\wh X_{\wh w}] = [\wh v \wh X'_{\wh v^{-1}\wh w}] = [\wh X'_{\wh v^{-1}\wh w}]$.
The second follows from the given identity
$$
\langle \rho+w^{-1}\rho, \dot\delta\rangle - \langle 2\rho, \dot\delta\rangle + \langle \wh v^{-1} \wh \rho+\wh w^{-1} \wh \rho,\dot\delta\rangle = 0
$$
and the observation that $\wh\rho = \wh v\wh \rho'$ (here $\wh \rho$ is the half-sum of positive roots of $\wh B$.)
\end{proof}

\section{Generalization of Fulton's conjecture for $G\subseteq \wh G$}\label{nextgen}

With all notation as in the introduction, in this section we prove Theorems \ref{BFCi} and \ref{BFCii}.  As an immediate corollary, we obtain a generalization of Fulton's conjecture for a pair of reductive groups $G\hookrightarrow \wh G$, one embedded in the other. We recall the following deformed pullback in cohomology from \cite{RessRich}. Let $\rho$ be half the sum of positive roots for $G$, and let $\wh \rho$ denote the same for $\wh G$. 

\begin{defi}
Let $\phi_\delta^*$ be the induced pullback in cohomology for an embedding $G/P(\delta)\to \wh G/\wh P(\delta)$. Then in the Schubert basis for $H^*(G/P)$, we may write 
$$
\phi_\delta^*
\left([\wh X_{\wh w}]\right)
=\sum_{w\in W^P} d_{\wh w}^w [X_w]
$$
for suitable integers $d_{\wh w}^w$. Define
$$
\phi_\delta^\odot
\left([\wh X_{\wh w}]\right)
=\sum_{w\in W^P} c_{\wh w}^w [X_w],
$$
where $c_{\wh w}^w = d_{\wh w}^w$ if $\langle \rho+w^{-1}\rho,\dot\delta\rangle - \langle \wh \rho +\wh w^{-1}\wh \rho,\dot \delta \rangle =0$ and $c_{\wh w}^w = 0$ otherwise. 
\end{defi}

There is another, equivalent, definition of this product which replaces the numerical requirement for $c_{\wh w}^w = d_{\wh w}^w$ with a geometric one, which is called \emph{Levi-movability} (L-movability for short):

\begin{proposition}
Suppose $w,\wh w$ satisfy $d_{\wh w}^w\ne 0$. Then $c_{\wh w}^w\ne 0$ if and only if for generic $(l,\wh l)\in L\times \wh L$, the vector space map 
$$
T_{\dot e}(G/P)\to \frac{T_{\dot e}(G/P)}{T_{\dot e}(l\bar w^{-1} X_{\bar w})}\oplus \frac{T_{\dot e}(\wh G/\wh P)}{T_{\dot e}(\wh l \wh w^{-1} \wh X_{\wh w})} 
$$
is an isomorphism, where $\bar w = w_0ww_0^P$ is the dual of $w\in W^P$. The latter condition is equivalent to the statement: generic $L\times \wh L$-translates of $\bar w^{-1}X_{\bar w}$ and $\wh w^{-1}\wh X_{\wh w}$ intersect transversally at $\dot e$.
\end{proposition}

\begin{proof}
This is \cite[Proposition 2.3]{RessRich}.
\end{proof}

Suppose $w,\wh w, \delta$ satisfy 
\begin{align}\label{dgen}
\phi_\delta^\odot \left([\wh X_{\wh w}]\right)\odot_0 [X_w] = d[X_e]
\end{align}
for some $d>0$; we do not necessarily require in this section that $\delta\in \mathfrak{S}$. As always, we assume $\wh P(\delta)$ is standard. We also assume that $w, \wh w$ are minimal length coset representatives in $W/W_\delta$, $\wh W/\wh W_\delta$.

\begin{theorem}[Generalization of Fulton's conjecture]
If $d=1$ in (\ref{dgen}), then for any $n\ge 1$, 
$$
\dim \left(V_L(n(\chi_w-\chi_1))\otimes V_{\wh L}(n\wh \chi_{\wh w})\right)^{L} = 1.$$
\end{theorem}

\subsection{Geometric setup}

Define the universal intersection scheme
$$
\mathcal{X} = \{(g,\wh g,z)\in G/B\times \wh G/\wh B\times \wh G/\wh P : z\in \phi_\delta(gX_w)\cap \wh g\wh X_{\wh w}\};
$$
the scheme structure is given as in \cite[\S 5]{BKR}. For a Schubert variety $X_w$, let $Q_w\subset G$ be its stabilizer. Let $Z_w$ denote its smooth locus, $Y_w$ the orbit $Q_wwP$, and $C_w$ the Schubert cell $BwP$. Observe that 
$$
X_w\supseteq Z_w\supseteq Y_w\supseteq C_w.
$$
Define analogous spaces $\wh Z_{\wh w}, \wh Y_{\wh w}, \wh C_{\wh w}$ for the $\wh G$-context. Then by replacing $X_w, \wh X_{\wh w}$ in the definition of $\mathcal{X}$ with the corresponding pairs of subvarieties, we define open subvarieties
$$
\mathcal{X}\supseteq \mathcal{Z}\supseteq\mathcal{Y}\supseteq \mathcal{C}.
$$
We record various properties of these spaces in the following lemma:
\begin{lemma}\label{five2}
\begin{enumerate}[label=(\alph*)]
\item Each of $\mathcal{X}, \mathcal{Z}, \mathcal{Y}, \mathcal{C}$ is irreducible.
\item $\mathcal{Z}, \mathcal{Y}, \mathcal{C}$ are all smooth.
\item $\mathcal{X}\setminus\mathcal{Z}$ is codimension $\ge2$ inside $\mathcal{X}$.
\end{enumerate}
\end{lemma}
The proofs of these statements are identical to those of \cite[Lemma 5.2]{BKR}, so we omit them here.

Assume $d=1$ in (\ref{dgen}). Then $\pi: \mathcal{Z}\to G/B\times \wh G/\wh B$ is a birational morphism of smooth varieties, $\pi$ fails to be injective exactly where the map on tangent planes is not an isomorphism. We use $\mathcal{R}$ to denote the associated \emph{ramification divisor}, and may use the symbol $\mathcal{R}$ to mean analogous divisors $\mathcal{R}\cap \mathcal{Y}$ and $\mathcal{R}\cap \mathcal{C}$, depending on the context. 

The proof of Theorem \ref{BFCi} relies on the following crucial geometric result of \cite[Proposition 3.1]{BKR}, which we recall without proof:

\begin{proposition}
Suppose $\pi:X\to Y$ is a regular birational morphism of smooth irreducible varieties with $Y$ projective, and suppose $\bar X$ is an irreducible projective scheme containing $X$ as an open subscheme such that
\begin{enumerate}[label=(\alph*)]
\item the codimension of $\bar X\setminus X$ in $\bar X$ is at least $2$, and
\item $\pi$ extends to a regular map $\bar \pi: \bar X\to Y$.
\end{enumerate}
Set $R$ to be the ramification divisor of $\pi$. Then 
$$
\dim H^0(X,\mathcal{O}(nR))=1
$$
for every $n\ge 1$.
\end{proposition}

When applied to our context, we obtain the following result:
\begin{corollary}Suppose equation (\ref{dgen}) holds with $d=1$. Then 
for every integer $n\ge 1$,
$\dim H^0(\mathcal{Z}, \mathcal{O}(n\mathcal{R}))=1$.
\end{corollary}
\begin{proof}
In the setting of the proposition, take $X = \mathcal{Z}$, $Y = G/B\times \wh G/\wh B$, and $\pi:\mathcal{Z}\to G/B\times \wh G/\wh B$ the projection map. Here $\mathcal{X}$ plays the role of $\bar X$. By Lemma \ref{five2}, $\mathcal{Z}\subseteq \mathcal{X}$ is an open subscheme whose complement has codimension $\ge2$. 
\end{proof}

\subsection{Comparison of $\mathcal{Y}$ and $\mathcal{Z}$ and proof of Theorem \ref{BFCi}}
Theorem \ref{BFCi} is a statement about sections on $\mathcal{Y}$, and our previous corollary pertains to $\mathcal{Z}$, so we connect the two here, thereby proving the theorem.

\begin{proposition}\label{Ayayay}
There exists a subvariety $A\subset \mathcal{Z}$ such that 
$\codim(A,\mathcal{Z})\ge2$ and $\mathcal{Z}\setminus \mathcal{Y}\subseteq A\cup \mathcal{R}$.
\end{proposition}

\begin{proof}
A point $(g,\wh g,z)\in \mathcal{Z}\setminus \mathcal{Y}$ if and only if $z\in \phi_\delta(gZ_w)\cap \wh g \wh Z_{\wh w}$ but $z\not\in \phi_{\delta}(gY_w)\cap \wh g\wh Y_{\wh w}$. That is, $z\in \phi_\delta(gC_v)\cap \wh g\wh C_{\wh v}$ for some $v,\wh v\in W^P\times \wh W^{\wh P}$ such that $C_v\not\subseteq Y_w$ or $\wh C_{\wh v}\not\subseteq \wh Y_{\wh w}$, but $C_v\subseteq Z_w$ and $\wh C_{\wh v}\subseteq \wh Z_{\wh w}$. In other words,
$$
\mathcal{Z}\setminus\mathcal{Y} = \bigsqcup_{\begin{array}{c}(vP,\wh v\wh P)\in Z_w\times \wh Z_{\wh w} \\ (vP,\wh v\wh P)\not\in Y_{w}\times \wh Y_{\wh w}\end{array}} 
\underbrace{\left((G\times_BC_v)\times (\wh G\times_{\wh B}\wh C_{\wh v})\right)\times_{G/P\times \wh G/\wh P} \wh G/\wh P.}_{\text{\normalsize $=:\mathcal{C}_{v,\wh v}$}}
$$

By inspection, the codimension of $\mathcal{C}_{v, \wh v}$ inside $\mathcal{Z}$ is equal to $\codim(C_v,Z_w)+\codim(\wh C_{\wh v},\wh Z_{\wh w})$. Therefore, if we show that the codimension $1$ cells $\mathcal{C}_{v, \wh v}$ that are disjoint from $\mathcal{Y}$ are contained in $\mathcal{R}$, we may take $A$ to be the disjoint union of the remaining cells in the above expression and the result will follow. 

To that end, we observe that (given $(vP,\wh v\wh P)\in Z_w\times \wh Z_{\wh w}$) $\codim(C_v,Z_w)+\codim(\wh C_{\wh v},\wh Z_{\wh w}) = 1$ if and only if 
\begin{enumerate}[label=(C\arabic*)]
\item $v\xrightarrow{\beta}w$ and $\wh v = \wh w$ for some root $\beta\in \Phi^+$ or 
\item $v=w$ and $\wh v\xrightarrow{\beta} \wh w$ for some root $\beta\in \wh \Phi^+$
\end{enumerate}
(these are obviously mutually exclusive). Furthermore, if $\beta$ is a simple root in either (C1) or (C2), then $\mathcal{C}_{v,\wh v}\subset \mathcal{Y}$ by \cite[Proposition 7.2]{BKR} (since then $C_v\subset Y_w$ in case (C1) or $\wh C_{v}\subset \wh Y_{\wh w}$ in case (C2)). So the result follows from 
\begin{proposition}\label{proselyte}
If $v,\wh v$ satisfy either (C1) or (C2) with $\beta$ not simple, then $\mathcal{C}_{v,\wh v}$ is contained in $\mathcal{R}$. 
\end{proposition}
The proof is the content of the next subsection.
\end{proof}

\begin{proof}[Proof of Theorem \ref{BFCi}]
The key here is that, for each $n$, $H^0(\mathcal{Y}, \mathcal{O}(n\mathcal{R}))$ includes into $H^0(\mathcal{Z}, \mathcal{O}(m(n)\mathcal{R}))$, where $m(n)\ge n$ is an integer depending on $n$. This is because functions on $\mathcal{Y}$ with poles to prescribed orders along $\mathcal{R}$ may be uniquely extended across the subvariety $A$ from Proposition \ref{Ayayay} to functions on $\mathcal{Z}$, possibly with greater order poles along $\mathcal{R}$. 

Therefore we have the inclusions 
$$
\C \hookrightarrow H^0(\mathcal{Y}, \mathcal{O}(n\mathcal{R}))^G\hookrightarrow H^0(\mathcal{Y}, \mathcal{O}(n\mathcal{R}))\hookrightarrow H^0(\mathcal{Z},\mathcal{O}(m(n)\mathcal{R})) \simeq \C
$$
for each $n$, and the result follows. 
\end{proof}

\subsection{Tangent space analysis}\label{highlighter}

This section is devoted to the proof of Proposition \ref{proselyte}; it may be read independently of the rest of the paper. 

The following lemma is proved in \cite[Lemma 7.3]{BKR}:
\begin{lemma}
Suppose $v\xrightarrow{\beta}w\in W^P$. As $T$-modules, 
$$
T_{\dot v}(X_w) \simeq \left(\bigoplus_{\gamma\in \Phi^+\cap v\Phi^-} \mathfrak{g}_\gamma\right)\oplus \mathfrak{g}_{-\beta}.
$$
Equivalently, as $T$-modules, 
$$
T_{\dot e}(v^{-1}X_w) \simeq \left(\bigoplus_{\gamma\in v^{-1}\Phi^+\cap \Phi^-} \mathfrak{g}_\gamma\right)\oplus \mathfrak{g}_{-v^{-1}\beta}.
$$
\end{lemma}

As a direct sum of $T$-eigenspaces, 
$$
T_{\dot e} (G/P)= \bigoplus_{\beta\in \Phi^+\setminus \Phi_\mathfrak{l}^+} T_{\dot e}(G/P)_{-\beta}.
$$
Define, for any $j\in \Z$, 
$$
V_j:=\bigoplus_{\begin{array}{c}\beta\in \Phi^+\setminus \Phi_\mathfrak{l}^+\\ \beta(\dot\delta)=j\end{array}} T_{\dot e}(G/P)_{-\beta}.
$$

Note that $V_j = (0)$ if $j\le 0$ or $j>m_0 := \max_\beta\{\beta(\dot\delta)\}$. Define $V_j(Z):=V_j\cap T_{\dot e}(Z)$ for any $T$-stable subvariety of $G/P$ containing $\dot e$.  Then 
$$
T_{\dot e}(Z)= \bigoplus_{j} V_j(Z)
$$
as $T$-modules. Let $\langle \delta\rangle = \im \delta$. If $Z$ is only $\langle \delta\rangle$-stable, the above decomposition is a valid $\langle \delta\rangle$-module decomposition. 

Recall the following important theorem from \cite[Theorem 7.4]{BKR} (see also \cite[Proposition 3]{RDist}). Although the original statement uses a different filtration $V_j$ than that given by $\delta$, the same proof goes through unchanged (just replace $x_P$ with $\dot \delta$ everywhere). 
\begin{theorem}\label{reprod}
Given that $u\xrightarrow{\beta}w\in W^P$ and $\beta$ is not simple, there exists $j$ such that $\dim V_j(u^{-1}Z_w)\ne \dim V_j(w^{-1}Z_w)$. 
\end{theorem}

In exact parallel, 
$$
T_{\dot e}(\wh G/\wh P) = \bigoplus_{\wh \beta\in \wh \Phi^+\setminus \wh \Phi_{\wh {\mathfrak{l}}}^+} T_{\dot e}(\wh G/\wh P)_{-\wh \beta},
$$
and one may define 
$$
\wh V_j:=\bigoplus_{\begin{array}{c}\wh \beta\in \wh \Phi^+\setminus \wh \Phi_{\wh{\mathfrak{l}}}^+\\ \wh{\beta}(\dot\delta)=j\end{array}} T_{\dot e}(\wh G/\wh P)_{-\wh \beta}.
$$
Analogously, if $\wh u\xrightarrow{\wh \beta} \wh w\in \wh W^{\wh P}$ and $\wh \beta$ is not simple, there exists a $j$ such that $\dim V_j(\wh u^{-1}\wh Z_{\wh w})\ne \dim V_j(\wh w^{-1}\wh Z_{\wh w})$. 

Because $d\phi_\delta: T_{\dot e}(G/P) \hookrightarrow T_{\dot e}(\wh G/\wh P)$ is a $T$-equivariant inclusion, it follows that for any $\beta\in \Phi$, the restriction of $d\phi_\delta$ satisfies
$$
d\phi_\delta: T_{\dot e}(G/P)_{\beta} \hookrightarrow \bigoplus_{\wh{\beta}\big|_\mathfrak{h} \equiv \beta} T_{\dot e}(\wh G/\wh P)_{\wh \beta}.
$$
In particular, then, $d\phi_\delta: V_j\hookrightarrow \wh V_j$ for each $j\in \Z$. 

The gradings $V_j$, $\wh V_j$ give rise to filtrations $\mathcal{F}_j$, $\wh{\mathcal{F}}_j$ of $T_{\dot e}(G/P)$, $T_{\dot e}(\wh G/\wh P)$, respectively. With respect to the adjoint $P$-action on $T_{\dot e}(G/P)$ (resp., $\wh P$ on $\wh T_{\dot e}(\wh G/\wh P)$), each $\mathcal{F}_j$ is $P$-stable (resp., each $\wh{\mathcal{F}}_j$ is $\wh P$-stable). 
Let $\mathcal{F}_j(Z)$, $\wh{\mathcal{F}}_j(Z)$ mean the induced filtrations of any $T_{\dot e}(Z)$. 

Now we introduce a lemma similar in spirit to \cite[Lemma 4.2]{BKR}. The following setup is essentially the same. Let $Y\subset X$ be irreducible smooth varieties, $Y$ locally closed in $X$. Suppose $X$ has a transitive action by a connected linear algebraic group $G$, and suppose $H$ is an algebraic subgroup fixing $Y$. For any $y\in Y$, define $\phi_y:G\to X$ by $g\mapsto gy$. Then for any $g\in G$, there is an induced tangent space map 
$$
d\phi_{(g,y)}: T_gG \to T_{gy}X.
$$
Because $Y$ is $H$-stable, there is an induced map 
$$
\Phi_{(g,y)}:T_{\bar g}(G/H)\to T_{gy}X/T_{gy}(gY).
$$
One easily checks that $\Phi_{(g,y)} = \Phi_{(gh,h^{-1}y)}$ if $h\in H$, so for each equivalence class $[g,y]\in G\times_H Y$ the map $\Phi_{[g,y]}$ is well-defined. The transitivity of the $G$-action implies that the maps $\Phi_{[g,y]}$ are surjective. 

 Suppose $\mathfrak a = [g,z],[\wh g, \wh z]\in \mathcal{Z}$. Define $x=gz, \wh x = \wh g\wh z$. In particular, $\wh x = \phi_\delta(x)$. Consider the following diagram of maps of tangent spaces

\begin{center}
\begin{equation}\label{TAN}
\begin{tikzcd}
T_{\mathfrak a}\mathcal{Z}  \arrow[r, "d\pi"]  \arrow[dd, "d \hat m"']& T_g(G/B) \oplus T_{\wh g}(\wh G/\wh B) \arrow[dd, "\Psi_{[g,z]} \times \Psi_{[\wh g,\wh z]}"] \\\\
T_x(G/P) \arrow[r, ] & \dfrac{T_x(G/P)}{T_x(gZ_w)}\oplus \dfrac{T_{\wh x}(\wh G/\wh P)}{T_{\wh x}(\wh g\wh Z_{\wh w})},
\end{tikzcd}
\end{equation}
\end{center}
where the bottom horizontal map is the canonical projection in the first factor and $d\phi_{\delta}$ followed by the canonical projection in the second factor. 

\begin{lemma}\label{pipes}
Diagram (\ref{TAN}) commutes. In fact, it is a fibre-product diagram. 
\end{lemma}

\begin{proof}
An arbitrary curve through $\mathfrak{a}$ in $\mathcal{Z}$ may be expressed as $\left([g(t),z(t)],[\wh g(t), \wh z(t)]\right)$, where $g(0)=g$, etc. The image under $d\pi$ of this curve's initial velocity is the initial velocity of $\left(\overline {g(t)},\overline{ \wh{ g(t)}}\right)$. Its further image under $\Psi_{[g,z]} \times \Psi_{[\wh g,\wh z]}$ is the pair of projections in the respective quotients of the initial velocities of $g(t)z(t)$ and $\wh g(t) \wh z(t)$. Note that $\wh g(t) \wh z(t) = \phi_\delta(g(t)z(t))$ for all $t$. Therefore the curve's image via the down and across compositions agree and the diagram commutes. 

That $T_{\mathfrak a}\mathcal{Z}$ is a subspace of (i.e., includes into) the fibre-product is clear since, for a curve $\left(\overline {g(t)},\overline{ \wh{ g(t)}}\right)$ through $(g,\wh g)$ in $T_g(G/B)\oplus T_{\wh g}(\wh G/\wh B)$ and corresponding $x(t)$ through $x$ in $G/P$, the curve \\$([g(t),z(t)],[\wh g(t), \wh z(t)])$ can be uniquely recovered via $z(t):=g(t)^{-1}x(t)$, $\wh z(t):=\wh g(t)^{-1}\phi_\delta(x(t))$. 

Counting dimensions, 
\begin{align*}
\dim \mathcal{Z} &= \dim G/P + \left(\dim (G\times_B Z_w) + \dim(\wh G\times_{\wh B} \wh Z_{\wh w}) \right) - \left(\dim G/P + \dim \wh G/ \wh P \right)\\
&= \dim G/P + \dim G/B + \dim Z_w + \dim \wh G/\wh B +\dim \wh Z_{\wh w} \\&\hspace{0.5in}- \dim G/P -\dim \wh G/\wh P \\
&= \dim G/P + \left(\dim G/B + \dim \wh G/\wh B\right)  - \left( \dim G/P - \dim Z_w\right) \\&\hspace{0.5in}- \left(\dim \wh G/\wh P - \dim \wh Z_{\wh w}\right),
\end{align*}
so $T_{\mathfrak a}\mathcal{Z}$ has the correct dimension and the result follows. 
\end{proof}

Now we come to the desired result. 

\begin{proof}[Proof of Proposition \ref{proselyte}]
Assume, for the sake of contradiction, that there exist $v, \wh v$ satisfying either (C1) or (C2) with $\beta$ not simple, and that there exists $\mathfrak{a} = \left([g,z],[\wh g, \wh z]\right)\in \mathcal{C}_{v, \wh v}\cap \mathcal{Z}\setminus \mathcal{R}$. Set $x = gz, \wh x = \wh g\wh z$; note $\wh x = \phi_\delta(x)$. By left $G$-translation, assume $x = \dot eP$ (this is possible since $\mathcal{C}_{v, \wh v},\mathcal{Z}, \mathcal{R}$ are all $G$-invariant.)

By $\mathfrak{a}\not\in \mathcal{R}$, $d\pi$ is an isomorphism, so 
$$
T_{\dot e}(G/P)\simeq \frac{T_{\dot e}(G/P)}{T_{\dot e}(gZ_w)}\oplus \frac{T_{\dot e}(\wh G/\wh P)}{T_{\dot e}(\wh g\wh Z_{\wh w})}
$$
by Lemma \ref{pipes}. Because $\mathfrak{a}\in \mathcal{C}_{v, \wh v}$, write $eP = gz = gbvP$ for suitable $b\in B$, and $e\wh P = \wh g\wh z = \wh g\wh b\wh v\wh P$ for some $\wh b\in \wh B$. So write $g = pv^{-1}b^{-1}, \wh g = \wh p\wh v^{-1}\wh b^{-1}$ for suitable $p\in P, \wh p\in \wh P$. So $T_{\dot e}(gZ_w) = T_{\dot e}(pv^{-1}Z_w)$ and $T_{\dot e}(\wh g\wh Z_{\wh w}) = T_{\dot e}(\wh p\wh v^{-1}\wh Z_{\wh w})$. 

Observe that 
$$
\mathcal{F}_j\to \frac{\mathcal{F}_j}{\mathcal{F}_j(pv^{-1}Z_w)}\oplus \frac{\wh{\mathcal{F}}_j}{\wh{\mathcal{F}}_j(\wh p\wh v^{-1}\wh Z_{\wh w})}
$$
is therefore injective for each $j$, so 
\begin{align}\label{Little1}
\dim \mathcal{F}_j\le \dim \mathcal{F}_j - \dim \mathcal{F}_j(pv^{-1}Z_w) + \dim\wh{\mathcal{F}}_j - \dim\wh{\mathcal{F}}_j(\wh p\wh v^{-1}\wh Z_{\wh w}).
\end{align}
Furthermore, \begin{align*}\dim \mathcal{F}_j(pv^{-1}Z_w) = \dim T_{\dot e}(pv^{-1}Z_w)\cap \mathcal{F}_j &= \dim \op{Ad}_p\left(T_{\dot e}(v^{-1}Z_w)\cap \mathcal{F}_j\right) \\&= \dim \mathcal{F}_j(v^{-1}Z_w)\end{align*} since $\op{Ad}_p(\mathcal{F}_j) = \mathcal{F}_j$. Likewise, $\dim \wh{\mathcal{F}}_j(\wh p\wh v^{-1}\wh Z_{\wh w}) = \dim \wh{\mathcal{F}}_j(\wh v^{-1}\wh Z_{\wh w})$. 

Now, the argument of \cite[Eq. (38) and paragraph preceding it]{BKR} shows that for each $j$ the inequalities 
\begin{align}\label{Little2}
\dim \mathcal{F}_j(w^{-1}Z_w)\le \dim \mathcal{F}_j(v^{-1}Z_w)~~\text{ and }~~\dim \wh{\mathcal{F}}_j(\wh w^{-1}\wh Z_{\wh w})\le \dim \wh{\mathcal{F}}_j(\wh v^{-1}\wh Z_{\wh w})
\end{align}
hold in general. Furthermore, by Theorem \ref{reprod}, there exists a $j=j_0$ such that 
\begin{align}\label{Little3}
\dim \mathcal{F}_j(w^{-1}Z_w)\ne \dim \mathcal{F}_j(v^{-1}Z_w)~~\text{ or }~~\dim \wh{\mathcal{F}}_j(\wh w^{-1}\wh Z_{\wh w})\ne \dim \wh{\mathcal{F}}_j(\wh v^{-1}\wh Z_{\wh w}),
\end{align}
depending on whether (C1) or (C2) holds. 

On the other hand, by 
$L$-movability, 
$$
\psi: T_{\dot e}(G/P)\to \frac{T_{\dot e}(G/P)}{T_{\dot e}(l w^{-1}X_w)}\oplus \frac{T_{\dot e}(\wh G/\wh P)}{T_{\dot e}(\wh l \wh w^{-1}\wh X_{\wh w})}
$$
is an isomorphism for generic $l,\wh l\in L\times \wh L$.

The latter decomposes (since $lw^{-1}X_w$, $\wh l\wh w^{-1}\wh X_{\wh w}$ are $\langle \delta\rangle$-stable) as 
$$
\left(\bigoplus_{j=1}^{m_0} \frac{V_j(G/P)}{V_j(lw^{-1}X_w)}\right)\oplus 
\left(\bigoplus_{j=1}^{m_0} \frac{\wh V_j(\wh G/\wh P)}{\wh V_j(\wh l\wh w^{-1}\wh X_{\wh w})}\right),
$$
and $\psi$ preserves $T$-weight spaces with the same $\delta$ action, so for each $j$ we must have 
$$
V_j(G/P)\simeq \frac{V_j(G/P)}{V_j(lw^{-1}X_w)}\oplus 
 \frac{\wh V_j(\wh G/\wh P)}{\wh V_j(\wh l\wh w^{-1}\wh X_{\wh w})}.
$$

Therefore 
\begin{align}\label{Little4}
\dim\mathcal{F}_j = \dim\mathcal{F}_j - \dim\mathcal{F}_j(lw^{-1}X_w)+\dim \wh{\mathcal{F}}_j - \dim \wh{\mathcal{F}}_j(\wh l\wh w^{-1}\wh X_{\wh w})
\end{align}
for each $j$, and the same holds without $l, \wh l$ by $P$-stability of $\mathcal{F}_j$ (sim. for $\wh{\mathcal{F}}_j$).

Finally, with $j=j_0$,
\begin{align*}
\dim \mathcal{F}_j&\le \dim \mathcal{F}_j-\dim \mathcal{F}_j(v^{-1}Z_w)+\dim \wh{\mathcal{F}}_j - \dim \wh{\mathcal{F}}_j(\wh v^{-1}\wh Z_{\wh w})~~\text{  by (\ref{Little1})}\\
&< \dim \mathcal{F}_j-\dim \mathcal{F}_j(w^{-1}Z_w)+\dim \wh{\mathcal{F}}_j - \dim \wh{\mathcal{F}}_j(\wh w^{-1}\wh Z_{\wh w})~~\text{  by (\ref{Little2}), (\ref{Little3})}\\
&= \dim \mathcal{F}_j~~\text{  by (\ref{Little4})},
\end{align*}
a contradiction.
\end{proof}

\subsection{Relation to representation theory for $L$}

The scheme $\mathcal{Y}$ is vitally important thanks to Theorem \ref{BFCi}. However, our first step in proving Theorem \ref{BFCii} is to exchange $\mathcal{Y}$ and $\mathcal{R}$ for a related pair of varieties. 

Define 
$$
\mathcal{Y}' := \left( (G\times_{Q_w} Y_w)\times (\wh G\times_{\wh Q_{\wh w}} \wh Y_{\wh w})\right)\times_{G/P\times \wh G/\wh P} \wh G/\wh P; 
$$
set-theoretically, 
$$
\mathcal{Y}' = \{(g,\wh g,z)\in G/Q_w\times \wh G/\wh Q_{\wh w}\times \wh G/\wh P: z\in \phi_\delta(gY_w)\cap \wh g\wh Y_{\wh w}\}.
$$

The surjections $G\times_B Y_w\to G\times_{Q_w} Y_w$ and $\wh G\times_{\wh B} \wh Y_{\wh w} \to \wh G\times_{\wh Q_{\wh w}} \wh Y_{\wh w}$ give rise to the surjective morphism $\mathcal{Y} \to \mathcal{Y'}$. In fact, the following diagram is a fibre diagram:
\begin{center}
\begin{tikzcd}
\mathcal{Y}  \arrow[rr, "\tilde p"] \arrow[dd, "\pi"]& & \mathcal{Y}' \arrow[dd, "\pi'"] \\\\ 
G/B\times \wh G/\wh B \arrow[rr] &  &  G/Q_w \times \wh G/\wh Q_{\wh w}.
\end{tikzcd}
\end{center}

Furthermore, $\pi'$ is a dominant morphism. By \cite[Lemma 4.1]{BKR}, for each $n\ge 1$, 
$$
H^0(\mathcal{Y}, \mathcal{O}(n\mathcal{R}))\simeq H^0(\mathcal{Y}',\mathcal{O}(n\mathcal{R}'))
$$
as $G$-modules, where $\mathcal{R}'$ is the ramification divisor of $\pi'$. 

There is a helpful equivalent description of $\mathcal{Y}'$, thanks to the following lemma (the proof is straightforward). 

\begin{lemma}\label{fibre-rich} Define
$$
\mathcal{P}:=  \left(P/w^{-1}Q_ww\cap P \times \wh P/\wh w^{-1}\wh Q_{\wh w}\wh w\cap \wh P\right).
$$
Then $\psi: G\times_P \mathcal{P}\to \mathcal{Y'}$ given by $[g,\bar p, \overline{\wh p}]\mapsto ([gpw^{-1},wP],[g\wh p\wh w^{-1},\wh w\wh P],g\wh P)$ is an isomorphism.
\end{lemma}

We will now relate $\mathcal{O}(\mathcal{R}')$ to a line bundle on $\mathcal{P}$ and then to the representation theory of $L$. First let us recall some well-known properties of the Borel construction of line bundles:

\begin{proposition}\label{4therecord}
Let $R$ be a reductive algebraic group with $B$ be a Borel subgroup of $R$. Suppose $R'$ is a subgroup of $R$ satisfying $B\subseteq R'$.
\begin{enumerate}[label=(\alph*)]
\item For any character $\chi:R'\to \C^*$, $\mathcal{L}_\chi:= R\times_{R'} \C_{-\chi}$ is a line bundle on $R/R'$.
\item The pullback map induces an isomorphism 
$$H^0(R/R',\mathcal{L}_{\chi})\simeq H^0(R/B,\mathcal{L}_\chi).
$$
\end{enumerate}
\end{proposition}

We need a couple more preparatory lemmas. The following is stated in \cite[\S 6]{BKR}, but a proof is included here for the reader's convenience. 

\begin{lemma}\label{extinction}
The torus weight $\chi_w:T\to \C^*$ extends to a character of $w^{-1}Q_ww\cap P$. Likewise, $\chi_{\wh w}$ extends to a character of $\wh w^{-1}\wh Q_{\wh w}\wh w\cap \wh P$.
\end{lemma}

\begin{proof}
The second statement is simply the application of the first to a different group, so we prove the first statement. We na\"ively define $\chi_w:w^{-1}Q_ww\cap P\to \C^*$ by setting $\chi_w(u)=1$ for all $u\in U_\alpha$, $U_\alpha$ a root subgroup of $w^{-1}Q_ww\cap P$ (we have no choice in this as such $u$ are unipotent). Then $\chi_w$ will be well-defined if, 
\begin{align}\label{extension}
\text{whenever $U_\alpha, U_{-\alpha}$ are both root subgroups, $\chi_w(\alpha^\vee)=0$}\end{align}
(on the algebra level). 

We first make a reduction: $U_{\pm \alpha}\subseteq w^{-1}Q_ww\cap P$ implies $\alpha$ is actually a root for $L$. So we may restrict our attention to root subgroups of $w^{-1}Q_ww\cap L$. Note that $w^{-1}Q_ww\cap L\supseteq B_L$, so $w^{-1}Q_ww\cap L$ is a standard parabolic of $L$. Therefore it suffices to check (\ref{extension}) only for simple roots $\alpha$ of $L$. 

This is fairly straightforward: if $-\alpha$ is a root for $w^{-1}Q_ww$, then $-w\alpha$ is (a) a negative root and (b) a root for $Q_w$. Therefore $-w\alpha$ can be expressed as a negative sum of simple roots for $Q_w$:
$$
-w\alpha = \sum -n_i\beta_i,
$$
where the $n_i\ge 0$ and $\{\beta_i\} = \Delta(Q_w) = \Delta\cap w(\Phi_{\mathfrak{l}}^+\sqcup \Phi^-)$. Rearranging, one obtains
$$
\alpha+\sum_{w^{-1}\beta_i\prec 0} n_i(-w^{-1}\beta_i) = \sum_{w^{-1}\beta_i\in \Phi_{\mathfrak{l}}^+} n_iw^{-1}\beta_i.
$$
Now, each $w^{-1}\beta_i$ on the LHS cannot be an element of $\Phi_{\mathfrak{l}}^-$ by the length-minimality of $w$ in its coset. Therefore if the LHS has any $n_i>0$, we reach a contradiction because the LHS is a sum of positive roots (for $G$), some of which are not roots for $L$, but the RHS is a sum of positive roots for $L$. So 
$$
\alpha = \sum_{w^{-1}\beta_i\in \Phi_{\mathfrak{l}}^+} n_iw^{-1}\beta_i.
$$
Because $\alpha$ is a simple root for $L$, each $n_i=0$ above except for some $n_j=1$ and $\alpha = w^{-1}\beta_j$ is simple. Therefore 
\begin{align*}
\chi_w(\alpha^\vee)&=\rho(\alpha^\vee)+w^{-1}\rho(w^{-1}\beta_j^\vee)-2\rho_L(\alpha^\vee)\\
&=1+\rho(\beta_j^\vee)-2\\
&=0.
\end{align*}
\end{proof}

\begin{lemma}\label{BKunderstood}
Suppose $\mu, \wh \mu$ are dominant weights of $T, \wh T$ such that $(\mu+\wh \mu)(\dot \delta)=0$. Then the pullback map
$$
H^0(P/B_L\times \wh P/\wh B_{\wh L},\mathcal{L}(\mu)\boxtimes \mathcal{L}(\wh \mu))^P \to H^0(L/B_L\times \wh L/\wh B_{\wh L},\mathcal{L}(\mu)\boxtimes \mathcal{L}(\wh \mu))^L
$$
is an isomorphism.
\end{lemma}

\begin{proof}
This is just a restatement of Proposition \ref{opendoor}, which will be proved below. 
\end{proof}

Finally we prove Theorem \ref{BFCii}. 

\begin{proposition}\label{sigh}
Suppose $\phi_{\delta}^\odot \left([\wh X_{\wh w}]\right) \odot_0 [X_w] = d[X_e] \in H^*(G/P)$ for some $d>0$.
Then 
$$
H^0(\mathcal{Y},\mathcal{O}(nR)|_{\mathcal{Y}})^G\simeq \left(V_L (n(\chi_{w}-\chi_1))^*\otimes V_{\wh L} (n\wh \chi_{\wh w})^*\right)^L
$$
\end{proposition}

\begin{proof}
Let $T^P = T_{\dot e}(G/P)$, $T^{\wh P} = T_{\dot e}(\wh G/\wh P)$, $T_w = T_{\dot e}(w^{-1}X_w)$, and $T_{\wh w} = T_{\dot e}(\wh w\wh X_{\wh w})$. 

For a point $(g,p,\wh p)\in G\times_P\mathcal{P}$, set $\mathfrak{a} = \psi([g,p,\wh p])$. We have the diagram 
\begin{center}
\begin{tikzcd}
T_{(g,p,\wh p)}(G\times_P\mathcal{P}) \arrow[r, "\sim", "d\psi"'] & T_{\mathfrak{a}}\mathcal{Y}'  \arrow[r, "d\pi"]  \arrow[dd, "d \hat m"']& T_{gpw^{-1}}(G/Q_w) \oplus T_{g\wh p\wh w^{-1}}(\wh G/\wh Q_{\wh w}) \arrow[dd, ] \\\\
& T_{gP}(G/P) \arrow[r, ] & \dfrac{T_{gP}(G/P)}{T_{gP}(gpw^{-1}Y_w)}\oplus \dfrac{T_{g\wh P}(\wh G/\wh P)}{T_{g\wh P}(g\wh p\wh w^{-1}\wh Y_{\wh w})},
\end{tikzcd}
\end{center}
which is a fibre-product diagram for the same reason as (\ref{TAN}). 

There are $P$-equivariant isomorphisms 
$$
P/w^{-1}Q_ww\cap P \times T^P\simeq P\times_{w^{-1}Q_ww\cap P} T^P
$$
and 
$$
\wh P/\wh w^{-1}\wh Q_{\wh w}\wh w\cap \wh P \times T^{\wh P}\simeq \wh P\times_{\wh w^{-1}\wh Q_{\wh w}\wh w\cap \wh P} T^{\wh P}
$$
given by $(\bar p,v)\mapsto (p,p^{-1}v)$ in both cases, cf. \cite[Definition 5]{BK}. Therefore there exist maps 
$$
\mathcal{P}\times T^P \to P/w^{-1}Q_ww\cap P \times T^P\simeq P\times_{w^{-1}Q_ww\cap P} T^P \to P\times_{w^{-1}Q_ww\cap P} (T^P/T_w)
$$
and 
\begin{align*}
\mathcal{P}\times T^P\to \frac{\wh P}{\wh w^{-1}\wh Q_{\wh w}\wh w\cap \wh P \times T^{P}} &\hookrightarrow \frac{\wh P}{\wh w^{-1}\wh Q_{\wh w}\wh w\cap \wh P \times T^{\wh P}}\simeq \wh P\times_{\wh w^{-1}\wh Q_{\wh w}\wh w\cap \wh P} T^{\wh P}\\&\to \wh P\times_{\wh w^{-1}\wh Q_{\wh w}\wh w\cap \wh P} (T^{\wh P}/T_{\wh w}).
\end{align*}

The map between fibres of the bundle map 
$$
G\times_P (\mathcal{P}\times T^P) \to G\times_P (P\times_{w^{-1}Q_ww\cap P} (T^P/T_w))\oplus G\times_P (\wh P\times_{\wh w^{-1}\wh Q_{\wh w}\wh w\cap \wh P} (T^{\wh P}/T_{\wh w}))
$$
over a point $(g,p,\wh p)\in G\times_P \mathcal{P}$ is readily identified with the map 
$$
T_{gP}(G/P)\to \dfrac{T_{gP}(G/P)}{T_{gP}(gpw^{-1}Y_w)}\oplus \dfrac{T_{g\wh P}(\wh G/\wh P)}{T_{g\wh P}(g\wh p\wh w^{-1}\wh Y_{\wh w})};
$$
therefore 
the ramification divisor $\psi^{-1}(R')$ in $G\times_P\mathcal{P}$ is the same as the ramification divisor of the bundle map 
$$
G\times_P (\mathcal{P}\times T^P) \to G\times_P (P\times_{w^{-1}Q_ww\cap P} (T^P/T_w))\oplus G\times_P (\wh P\times_{\wh w^{-1}\wh Q_{\wh w}\wh w\cap \wh P} (T^{\wh P}/T_{\wh w}))
$$
over $G\times_P \mathcal{P}$. Setting 
$$
\mathcal{M} = \mathcal{L}_P(\chi_w - \chi_1)\boxtimes \mathcal{L}_{\wh P}(\chi_{\wh w}),
$$
a line bundle over $\mathcal{P}$ (by Lemma \ref{extinction}), we conclude (cf. the discussion surrounding \cite[Lemma 6]{BK} and \cite[Proposition 6.2]{BKR}) that $\mathcal{O}(\phi^{-1}(R))$ is $G$-isomorphic to $G\times_P\mathcal{M}$ as line bundles over $G\times_P \mathcal{P}$. 

Therefore for any $n$,
\begin{align*}
H^0(\mathcal{Y},\mathcal{O}(nR))^G&\simeq H^0(\mathcal{Y}',\mathcal{O}(nR'))^G\\
&\simeq H^0(G\times_P\mathcal{P},G\times_P\mathcal{M}^{\otimes n})^G\\
&\simeq H^0(\mathcal{P},\mathcal{M}^{\otimes n})^P.
\end{align*}

Finally, set $\mathfrak{L} = L/(w^{-1}Q_ww\cap L) \times \wh L/(\wh w^{-1}\wh Q_{\wh w}\wh w\cap \wh L)$. Then, by Lemma \ref{BKunderstood} and Proposition \ref{4therecord}(b) (see also \cite[Theorem 15, Remark 31(a)]{BK}), it also holds that 
$$
H^0(\mathcal{P},\mathcal{M}^{\otimes n})^P\simeq H^0(\mathfrak{L},(\mathcal{M}|_\mathfrak{L})^{\otimes n})^L,
$$
from which the result follows. 
\end{proof}

\subsection{Interlude}

We will need the ``$\mathcal{C}$ version of Theorem \ref{BFCi}'' in the next section, so this subsection serves as the bridge between the generalized Fulton's conjecture and the type I rays. The proof of the following lemma is straightforward and ommitted; compare with Lemma \ref{fibre-rich}. 

\begin{lemma}\label{newx}
$
\mathcal{C} \simeq G\times_P \left(P/w^{-1}Bw\cap P \times \wh P/\wh w^{-1}\wh B\wh w\cap \wh P\right).
$
\end{lemma}

\begin{proposition}\label{bridge} For all $n\ge 1$,
$H^0(\mathcal{C},\mathcal{O}(n\mathcal{R}))^G\simeq \C$.
\end{proposition}

\begin{proof}
The idea of the proof is to exchange $\mathcal{Y}'$ (see end of proof of Proposition \ref{sigh}) for $\mathcal{C}$, which we hope is manageable since they both appear as $G\times_P(\text{ a homogeneous $P\times \wh P$-variety })$. 

Consider the maps 
\begin{center}
\begin{tikzcd}
P/B_L\times \wh P/\wh B_{\wh L} \arrow[r,"f_1"] \arrow[dr,  "f"'] & P/w^{-1}Bw\cap P \times \wh P/\wh w^{-1}\wh B\wh w\cap \wh P  \arrow[d,"f_2"]  \\
 & \mathcal{P},
\end{tikzcd}
\end{center}
where $\mathcal{P}$ is as in Lemma \ref{fibre-rich}; all arrows are the natural surjections (we are using that $wB_Lw^{-1}\subseteq B$ and $\wh w\wh B_{\wh L}\wh w^{-1}\subseteq \wh B$).  Take $\mathcal{M}$ as in Proposition \ref{sigh}. Then by Proposition \ref{4therecord}(b), all arrows in 
\begin{center}
\begin{tikzcd}
H^0(P/B_L\times \wh P/\wh B_{\wh L},(f^*\mathcal{M})^{\otimes n})  & H^0(P/w^{-1}Bw\cap P \times \wh P/\wh w^{-1}\wh B\wh w\cap \wh P,(f_2^*\mathcal{M})^{\otimes n})  \arrow[l,"f_1^*"']  \\
& H^0(\mathcal{P},\mathcal{M}^{\otimes n}),  \arrow[u,"f_2^*"'] \arrow[ul, "f*"] 
\end{tikzcd}
\end{center}
are $P$-equivariant isomorphisms. The bottom vector space has $P$-invariants $\simeq \C$ for any $n\ge 1$ by Proposition \ref{sigh}. Finally, by the commutativity of the following diagram: 
\begin{center}
\begin{tikzcd}
 & \mathcal{Y} \arrow[dr,"\bar p"] &  \\
\mathcal{C} \arrow[ur, "\iota"] \arrow[rr,"\text{id}\times f_2"]& & \mathcal{Y}', 
\end{tikzcd}
\end{center}
we ascertain that (for any $n\ge 1$)
$$
\mathcal{O}(n\mathcal{R}|_\mathcal{C}) = \iota^*\mathcal{O}(n\mathcal{R}) \simeq \iota^*\bar p^* \mathcal{O}(n\mathcal{R}') \simeq (\text{id}\times f_2)^*(G\times_P\mathcal{M}^{\otimes n}) = G\times_P (f_2^*\mathcal{M})^{\otimes n}.
$$
Therefore 
\begin{align*}
H^0(\mathcal{C},\mathcal{O}(n\mathcal{R}))^G&\simeq H^0(\mathcal{C},G\times_P(f_2^*\mathcal{M})^{\otimes n})^G\\
&\simeq H^0(P/w^{-1}Bw\cap P \times \wh P/\wh w^{-1}\wh B\wh w\cap \wh P,(f_2^*\mathcal{M})^{\otimes n})^P\\
&\simeq \C  
\end{align*}
for any $n\ge 1$.
\end{proof}

\section{Type I extremal rays}\label{divI}

In this section we introduce the divisors $D(v)\subset G/B\times \wh G/\wh B$ whose associated line bundles, via the Borel-Weil theorem, give generators $(\mu, \wh \mu)$ of certain extremal rays on a given regular facet. 

Suppose $w,\wh w, \delta$ satisfy (\ref{prod}); in fact $\delta\in \mathfrak{S}$ is not necessary. We assume, as always, that $\wh P(\delta)$ is a standard parabolic. We also assume $w,\wh w$ are minimal-length representatives in their cosets inside $W/W_\delta, \wh W/\wh W_\delta$. Let $\mathcal{X}\supset\mathcal{Z}\supset\mathcal{Y}\supset\mathcal{C}$, as well as $\mathcal{R}$, be as in Section \ref{nextgen}.

As in the introduction, suppose either $v\xrightarrow{\beta}w$ or $v\xrightarrow{\beta} \wh w$ for some simple root $\beta$ (for the appropriate root system). In the first case, set $u=v,\wh u = \wh w$. Otherwise in the second, set $u=w,\wh u = v$. Define
$$
\tilde D(v):=\{(g,\wh g,z)\in G/B\times \wh G/\wh B\times \wh G/\wh P:z\in \phi_\delta(gX_u)\cap \wh g\wh X_{\wh u}\}
$$
and set $D(v) = \pi(\tilde D(v))$, the projection onto $G/B\times \wh G/\wh B$. Although it is clear that $\tilde D(v)$ is codimension one inside $\mathcal{X}$, we must argue that $D(v)$ is codimension one inside $G/B\times \wh G/\wh B$, which we prove now:

\subsection{Proof of Theorem \ref{divisors}(a)}
The result will follow from 
\begin{lemma}
$\tilde D(v) \cap \mathcal{Y}$ is not contained in $\mathcal{R}$.
\end{lemma}

Indeed, this prevents $\tilde D(v)$ from being contained in $\mathcal{R}$ and thus being contracted to a codimension $\ge2$ subvariety of $G/B\times \wh G/\wh B$. 

\begin{proof}
Take any point $(g,\wh g, z)\in \mathcal{C}-\mathcal{R}$. Then 
$$
z\in \phi_\delta(g C_w)\cap \wh g \wh C_{\wh w}\subseteq \phi_\delta(g X_w)\cap \wh g \wh X_{\wh w}.
$$
By the tangent space requirement (away from $\mathcal{R}$), the preimage of $(g,\wh g)\in G/B\times \wh G/\wh B$ under $\pi$ is $1$-dimensional, and contains $(g,\wh g,z)$. By Zariski's main theorem, this preimage is also connected. Therefore we conclude 
$$
\phi_\delta(g C_w)\cap \wh g \wh C_{\wh w}= \phi_\delta(g X_w)\cap \wh g \wh X_{\wh w}=\{z\},
$$
a single point. Now, $z = \phi_\delta(xP)$ for some $xP\in gBwP$. 
Given $x\wh P = gbw\wh P = \wh g\wh b\wh w\wh P$ for suitable $b,\wh b$, we may replace $gb$, $\wh g\wh b$ with $g$, $\wh g$ without changing the cosets $gB$, $\wh g\wh B$. Furthermore, we may as well assume $x = gw$. Then for suitable $\wh p\in \wh P$, 
$$
x = gw = \wh g\wh w \wh p.
$$

As both $\mathcal{C}$ and $\mathcal{R}$ are (diagonal) $G$-invariant, we may translate by $(gw)^{-1}$ to obtain $(w^{-1}, \wh p^{-1} \wh w^{-1}, e\wh P)\in \mathcal{C}-\mathcal{R}$. Observe that 
\begin{align*}
\{e\wh P\} = \phi_\delta(w^{-1} C_w)\cap \wh p^{-1} \wh w^{-1}\wh C_{\wh w} &\subseteq  \phi_\delta(w^{-1} Y_w)\cap \wh p^{-1} \wh w^{-1}\wh Y_{\wh w}\\& \subseteq  \phi_\delta(w^{-1} X_w)\cap \wh p^{-1} \wh w^{-1}\wh X_{\wh w} = \{e\wh P\},
\end{align*}
so equalities hold all around. 

In case $v\xrightarrow{\beta} w$, we have $s_\beta \in Q_w$ and thus $s_\beta Y_w = Y_w$. Now $w^{-1} = v^{-1} s_\beta$, so 
$$
\{e\wh P\} = \phi_\delta(v^{-1}Y_w)\cap \wh p^{-1} \wh w^{-1}\wh Y_{\wh w}
$$
and therefore $(v^{-1}, \wh p^{-1} \wh w^{-1}, e\wh P)\in \mathcal{Y}- \mathcal{R}$. This point also lies in $\tilde D(v)$ since $e\wh P$ is included in both $v^{-1}Bv\wh P$ and $\wh p^{-1} \wh w^{-1} \wh B\wh w\wh P$. 

In the other case, $s_\beta\in Q_{\wh w}$ and $s_\beta \wh Y_{\wh w} = \wh Y_{\wh w}$. Again $\wh w^{-1} = v^{-1} s_\beta$, so 
$$
\{e\wh P\} = \phi_\delta(w^{-1}Y_w)\cap \wh p^{-1} v^{-1}\wh Y_{\wh w}
$$
and $(w^{-1}, \wh p^{-1} v^{-1}, e\wh P)\in \mathcal{Y}- \mathcal{R}$. This point also lies in $\tilde D(v)$ since $e\wh P$ is included in both $w^{-1}Bw\wh P$ and $\wh p^{-1} v^{-1} \wh Bv\wh P$. 

We conclude that, in either case, $\tilde D(v) \cap \mathcal{Y}-\mathcal{R}\ne \emptyset$. 
\end{proof}

Like in \cite[Corollary 15]{BKiers}, the above proof lets us also conclude that $\pi_*(\tilde D(v)) = D(v)$ as divisors. 

\subsection{Proof of Theorem \ref{divisors}(b)}

Recall that by Proposition \ref{bridge},
$$
H^0(\mathcal{C}-\mathcal{R}, \mathcal{O})^G \simeq \C.
$$
We relate $G$-invariant functions on $\mathcal{C}-\mathcal{R}$ with those on $G/B\times \wh G/\wh B$ away from $D(v)$ by means of 

\begin{lemma}\label{missedme}
$\pi(\mathcal{C}-\mathcal{R})\subseteq G/B\times \wh G/\wh B - D(v)$.
\end{lemma}

\begin{proof}
Assume $(g,\wh g)\in D(v)$ is in the image of $\mathcal{C}-\mathcal{R}$. Then there exists a unique $z$ such that 
$$
\{z\} = \phi_\delta(gC_w)\cap \wh g \wh C_{\wh w} = \phi_\delta(gX_w)\cap \wh g\wh X_{\wh w}
$$
and there exists a $z'$ such that 
$$
z'\in \phi_\delta(gX_v)\cap \wh g\wh X_{\wh w},
$$
or the analogous statement for $v\xrightarrow{\beta}\wh w$. 
Of course, $gX_v\subset gX_w$, so $z'\in \phi_\delta(gX_w)\cap \wh g\wh X_{\wh w}$ implies $z = z'$. However, $gX_v$ is disjoint from $gC_w$, which shows $z\ne z'$, a contradiction. A similar contradiction arises in the other case.
\end{proof}

We come now to the proof of Theorem \ref{divisors}(b): Any $f\in H^0(G/B\times \wh G/\wh B,\mathcal{O}(mD(v)))^G$, viewed as a $G$-invariant function on $G/B\times \wh G/\wh B-D(v)$, can be pulled back to a $G$-invariant function on $\mathcal{C}-\mathcal{R}$ via $\pi$. Now $H^0(\mathcal{C}-\mathcal{R}, \mathcal{O})^G$ consists only of constant functions by Proposition \ref{bridge}. Therefore $f\circ \pi$ is constant, and $f$ is constant on $\pi(\mathcal{C}-\mathcal{R})$. By the birationality of $\pi$, $\pi(\mathcal{C}-\mathcal{R})$ is a dense open subset of $G/B\times \wh G/\wh B$, hence also of $G/B\times \wh G/\wh B-D(v)$. Therefore $f$ itself is actually constant. We conclude that $H^0(G/B\times \wh G/\wh B,\mathcal{O}(mD(v)))^G$ is $1$-dimensional for all $m$. 

\subsection{Proof of Theorem \ref{divisors}(c)}

This statement follows from part (b) exactly as in \cite[Lemma 2.1]{B}.

\section{Parameter stacks for type I rays}\label{nextup}

In this section we introduce some of the core geometry of the paper, using quotient stacks to describe a Levification procedure and prove Proposition \ref{opendoor}, and we prove Theorem \ref{divisors}(d). 

\subsection{Review of principal $G$-spaces}

\begin{defi}
For us, a \emph{principal $G$-space $E$} is a variety endowed with a simply transitive right $G$-action. 
%with a right $G$ action such that for any $x\in E$, the map $G\to E$ given by $g\mapsto xg$ is an isomorphism. 

If $\phi:G\to H$ is a morphism of linear algebraic groups, then 
$$
E\times_G H = \{(e,h)\in E\times H\}/(e,h)\sim(eg,\phi(g)^{-1}h)
$$
is naturally a principal $H$-space. 
\end{defi}

We also define the notion of \emph{relative position}. 

\begin{lemma}
Let $E$ be a principal $G$-space and $B\subseteq P\subseteq G$ as usual. Let $\bar g\in E/B, \bar z\in E/P$. Then there is a unique $w\in W^P$ such that there exist $b\in B, p\in P$ satisfying
$$
z = gbwp^{-1}.
$$
\end{lemma}

\begin{proof}
There is a unique $y\in G$ so that $gy=z$. Any $y\in G$ is expressible as $bwp^{-1}$ for some $b\in B, p\in P$, $w\in W$; furthermore, the choice of $w$ is unique to $y$. Thus $z = gbwp^{-1}$ as prescribed. Furthermore, the choices of $g,z$ as representatives for $\bar g, \bar z$ do not affect $w$, given that $b,p^{-1}$ are free to change accordingly. 
\end{proof}

We define the \emph{relative position} $[\bar g, \bar z]\in W^P$ to be $w$ as above. 

\subsection{Introduction of universal intersection stacks}

We introduce the following stacks, similar in nature to those of \cite[\S 3.4]{BKiers}. 

\begin{enumerate}[label=$\bullet$]
\item Let $\op{Fl}_G$ parametrize principal $G$-spaces $E$ together with $\bar g\in E/B$, $\overline{\wh g}\in (E\times_G \wh G)/\wh B$ (in families over a scheme $X$, it parametrizes principal $G$-bundles $E$ over $X$ locally trivial in the fppf topology, together with sections $\bar{g}\in E/B$ and $\overline{\wh g}\in (E\times_G \wh G)/\wh B$). 

Fixing $x\in E$, $g = xh$ and $\wh g = (x,\wh h)$ defines elements $\bar h\in G/B$ and $\overline{\wh h} \in \wh G/\wh B$. Changing representatives for $\bar g$ and $\overline{\wh g}$ does not change $\bar h$ and $\overline{\wh h}$. Changing $x$ to $x\tilde g$ for $\tilde g\in G$ changes $\bar h, \overline{\wh h}$ to $\overline{\tilde g^{-1} h}$ and $\overline{\tilde g^{-1}\wh h}$. Thus as stacks, 
$$
\op{Fl}_G = \left[\left(G/B\times \wh G/\wh B\right)/G\right],
$$
where the RHS is the quotient stack with right $G$-action given by left multiplication by $g^{-1}$. 

\item Similarly, set $\op{Fl}_L = \left[\left(L/B_L\times \wh L/\wh B_{\wh L}\right)/L\right]$, which parametrizes principal $L$-spaces $F$ together with $\bar q \in F/B_L$ and $\overline{\wh q}\in (F\times_L \wh L)/\wh B_{\wh L}$. 

\item Let $\wh{\mathcal{C}}$ be the stack parametrizing principal $G$-spaces $E$, elements $\bar g\in E/B$, $\overline{\wh g}\in (E\times_G \wh G)/\wh B$, and an element $\bar z\in E/P$ satisfying 
$$
[\bar g, \bar z] = w ~~\text{  and  }~~ \left[\overline{\wh g}, \overline{(z,e)}\right] = \wh w.
$$
Then, similar to above, $\wh{\mathcal{C}} = \left[\mathcal{C}/G\right]$. 

\end{enumerate}

Observe that there is a natural map $\pi:\wh{\mathcal{C}}\to \op{Fl}_G$ induced by the $G$-equivariant morphism $\pi:\mathcal{C}\to G/B\times \wh G/\wh B$. 

The following lemma will help us identify maps between $\wh{\mathcal{C}}$ and $\op{Fl}_L$. 

\begin{lemma}\label{alt}
The stack $\wh{\mathcal{C}}$ parametrizes principal $P$-spaces $E'$ together with elements $\bar y\in E'/(w^{-1}Bw\cap P)$ and $\overline{\wh y}\in (E'\times_P\wh P)/(\wh w^{-1}\wh B\wh w\cap \wh P)$. 
\end{lemma}

\begin{proof}
This is simply a reformulation of Lemma \ref{newx}.
\end{proof}

The equivalent description of $\wh{\mathcal{C}}$ given by Lemma \ref{alt} allows us to use the inclusion $L\to P$ and projection $P\to P/U = L$ maps to define maps $\op{Fl}_L \to \wh{\mathcal{C}}$ and $\wh{\mathcal{C}}\to \op{Fl}_L$, respectively. We describe these maps now. 

First recall (cf. \cite[Lemma 19]{BKiers}) that $B_L\subset w^{-1}Bw\cap P$ and that, if $\phi:P\to L$ is the quotient map, $\phi(w^{-1}Bw\cap P) = B_L$. Thus if $F$ is a principal $L$-space with $\bar q\in F/B_L$ and $\overline{\wh q}\in (F\times_L \wh L)/\wh B_{\wh L}$, the $P$-space $F\times_L P$ and elements $\overline{(q,e)}\in (F\times_L P)/(w^{-1}Bw \cap P)$ and $\overline{(\wh q,e)}\in (F\times_L \wh P)/(\wh w^{-1}\wh B\wh w\cap \wh P)$ are well-defined. 

Conversely, if $E'$ is a principal $P$-space with $\bar y\in E'/(w^{-1}Bw\cap P)$ and $\overline{\wh y}\in (E'\times_P \wh P)/(\wh w^{-1}\wh B\wh w\cap \wh P)$, the $L$-space $E'\times_P L$ and elements 
$\overline{(y,e)}\in (E'\times_P L)/B_L$
and $\overline{(\wh y, e)} \in (E'\times_P \wh L)/\wh B_{\wh L}$
are well-defined. Thus we have maps $i:\op{Fl}_L\to \wh{\mathcal{C}}$ and $\tau:\wh{\mathcal{C}}\to \op{Fl}_L$, and these evidently satisfy $\tau\circ i = \op{id}_{\op{Fl}_L}$. 

\subsection{The main diagram of stacks}
There are natural maps of stacks $\mathcal{C}\to  \wh{\mathcal{C}}$ and $G/B\times \wh G/\wh B\to \op{Fl}_G$, and these commute with the relevant maps $\pi$. Introducing the map $\tilde i = \pi\circ i$, we present the following useful diagram of stacks:

\begin{center}
\begin{tikzcd}
\mathcal{C} \arrow[d,"\pi"]  \arrow[r] & \wh{\mathcal{C}} \arrow[d, "\pi"] \arrow[dr, "\tau"'] & \\
G/B\times \wh G/\wh B \arrow[r] & \op{Fl}_G & \op{Fl}_L \arrow[l, "\tilde i"] \arrow[ul, bend right=20, "i"']
\end{tikzcd}
\end{center}

\subsection{Line bundles on $\wh{\mathcal{C}}$ and $\op{Fl}_L$ are related (Levification)}

The following definition is adapted from \cite[Definition 24]{BKiers}. Whereas in \cite{BKiers} it was true that $Z(L) = Z(\wh L)\cap L$, this need not be the case at present. Instead we find it best to work with $\langle\delta\rangle = \im(\delta)$, which is of course still contained in both $Z(L)$ and $Z(\wh L)$. 

\begin{defi}
Let $\mathcal{M}$ be a line bundle on $\op{Fl}_L$, viewed as an $L$-equivariant line bundle on $L/B_L\times \wh L/\wh B_{\wh L}$. Then $\C^*$ acts on each fibre of $\mathcal{M}$ via $\delta$. Because the group of characters of $\C^*$ is discrete, the map 
$$
X:=L/B_L \times \wh L/\wh B_{\wh L} \to \op{Hom}(\langle\delta\rangle,\C^*)
$$
is constant ($X$ is connected). Thus $\mathcal{M}$ gives rise to a single $\gamma_\mathcal{M}: \langle\delta\rangle\to \C^*$, and $\gamma_\mathcal{M}$ can be defined even if $\mathcal{M}$ is only defined over a connected subset of $X$ (for example, any Zariski open subset, given irreducibility of $X$). 
\end{defi}

The following proposition generalizes \cite[Proposition 25]{BKiers}:
\begin{proposition}\label{opendoor}
Let $U$ be a non-empty open substack of $\op{Fl}_L$, $\mathcal{L}$ a line bundle on $\tau^{-1}(U)$ and $\mathcal{M} = i^*\mathcal{L}$, a line bundle on $U$. Then 
\begin{enumerate}[label=(\alph*)]
\item $\mathcal{L} = \tau^*\mathcal{M}$. This shows $\tau^*:\pic(U)\to \pic(\tau^{-1}(U))$ is an isomorphism with inverse $i^*$.
\item If $\gamma_\mathcal{M}$ is trivial, then $H^0(\tau^{-1}(U),\mathcal{L})\to H^0(U,\mathcal{M})$ is an isomorphism. 
\end{enumerate}
\end{proposition}

Before embarking on the proof, we set up the generalized setting for Levification (cf. \cite[\S 3.6]{BKiers}); here the role of $t^{x_L}$ will be played by $\delta(t)$. 

\begin{defi}
Define a family of maps $\psi: \wh P\times \C^*\to \wh P$ by $\psi_t(p) = \delta(t)p\delta(t)^{-1}$ for $t\in \C^*$. 
\end{defi}

We record several straightforward facts about the $\psi_t$. 

\begin{lemma}
Each $\psi_t$ is the identity on $\wh L$, and of course $\psi_1$ is the identity on $\wh P$. In the limit, $\psi_0 = \lim_{t\to 0}\psi_t$ exists and equals the quotient map $\wh P\to \wh L$. Similarly, the restriction $\psi_t:P\to P$ is the identity on $L\subset P$ and in the limit $\psi_0:P\to L$ is the standard quotient map again. The diagrams 
$$
\begin{array}{ccc}
\begin{tikzcd}
P \arrow[hookrightarrow, r] \arrow[d,"\psi_t"']&  \wh P \arrow[d,"\psi_t"]\\
P \arrow[hookrightarrow, r]& \wh P
\end{tikzcd}
& \text{ and }& 
\begin{tikzcd}
P \arrow[hookrightarrow, r] \arrow[d,"\psi_0"']&  \wh P \arrow[d,"\psi_0"]\\
L \arrow[hookrightarrow, r]& \wh L
\end{tikzcd}
\end{array}
$$
commute. 
\end{lemma}

\begin{defi}
Now given a principal $P$-space $E'$ and elements $\bar y\in E'/(w^{-1}Bw\cap P)$ and $\overline{\wh y}\in (E'\times_P \wh P)/(\wh w^{-1}\wh B\wh w\cap \wh P)$, define for each $t\in \A^1$ the principal $\psi_t(P)$-space $E_t = E'\times_{\psi_t} P$, together with elements $\bar y_t = \overline{(y,e)}\in E_t/(w^{-1}Bw\cap P)$ and $\overline{\wh y}_t = \overline{(\wh y,e)}\in ((E'\times_P\wh P)\times_{\psi_t}\wh P)/(\wh w^{-1}\wh B\wh w\cap \wh P) = (E'\times_{\psi_t} \wh P)/(\wh w^{-1}\wh B\wh w\cap \wh P)$. 

\end{defi}

\begin{proof}[Proof of Proposition \ref{opendoor}]
We will actually prove (b) first and use it for (a). 

(b) Any section of $\mathcal{L}$ over a point $(E',\bar y,\overline{\wh y})$ in $\tau^{-1}(U)$ extends uniquely to each $(E_t,\bar y_t, \overline{\wh y}_t)$ by the $P$-equivariance. By the triviality of the action of $\delta(t)$ on the fibre above the limit point $(E_0, \bar y_0, \overline{\wh y}_0)$, this section can be extended uniquely without zeros or poles to $(E_0, \bar y_0, \overline{\wh y}_0) = i^{-1}((E',\bar y, \overline{\wh y}))$. This shows the injectivity of the pullback map 
$$
i^*:H^0(\tau^{-1}(U),\mathcal{L}) \to H^0(U,\mathcal{M}).
$$
Surjectivity follows by extending any section at $(E_0, \bar y_0, \overline{\wh y}_0)$ to all $E_t$ as in \cite[Lemma 8.6]{B}. 

(a) Take $\mathcal{L}' = \mathcal{L}\otimes (\tau^*\mathcal{M})^{-1}$, and set $\mathcal{M}' = i^* (\mathcal{L}')$. Observe that 
$$
\mathcal{M}' = i^*(\mathcal{L})\otimes \left((\tau\circ i)^*\mathcal{M}\right)^{-1} = \mathcal{M}\otimes \mathcal{M}^{-1}
$$
is actually just $\mathcal{O}_U$. Consequently, $\gamma_{\mathcal{M}'}$ is trivial and (2) applies: 
$$
H^0(\tau^{-1}(U), \mathcal{L}') \simeq H^0(U, \mathcal{O}_U).
$$
The latter contains a nowhere-vanishing section, the constant function $1$, so say $i^*(\sigma) =1$ by the isomorphism. If $\sigma$ itself vanishes anywhere on $\tau^{-1}(U)$, it must not vanish on $\op{im}(i)$ since $1$ does not vanish on $U$.  But any vanishing of $\sigma$ elsewhere can be propagated to $\op{im}(i)$ by Levification, which cannot be. So $i^*(\sigma)$ is a nowhere-vanishing section of $\mathcal{L}'$. We conclude that $\mathcal{L}'$ is trivial, which gives the result. 
\end{proof}

\subsection{Proof of Theorem \ref{divisors}(d)}

We next introduce the following analogue of \cite[Lemma 29]{BKiers}: 

\begin{lemma}\label{florence}
Suppose $\mathcal{L} = \mathcal{L}_{\mu}\boxtimes \mathcal{L}_{\wh \mu}$ is in $\pic(\operatorname{Fl}_G)$, and let $\mathcal{M}$ denote its pullback to $\operatorname{Fl}_L$. Then the following are equivalent: 
\begin{enumerate}[label=(\alph*)]
\item The equality 
$$
 w^{-1}\mu(\dot\delta)+\wh w^{-1}\wh \mu(\dot\delta)=0
$$
holds. 
\item $\gamma_\mathcal{M}: \langle\delta\rangle\to \C^*$ is trivial. 
\end{enumerate}
\end{lemma}

\begin{proof}
The map $L/B_L \times \wh L/\wh B_{\wh L}\to G/B\times \wh G/\wh B$ given by $(\bar q, \overline{\wh q})\mapsto (\overline{qw^{-1}}, \overline{\wh q\wh w^{-1}})$ is well-defined and gives rise to the map on stacks. The fibre of $\mathcal{M}$, viewed as an $L$-equivariant bundle on $L/B_L\times \wh L/\wh B_{\wh L}$, over an arbitrary $(\bar q, \overline{\wh q})$ is simply the fibre 
$$
\mathcal{L}_{(\overline{qw^{-1}},\overline{\wh q\wh w^{-1}})}= \{(qw^{-1},t),(\wh q\wh w^{-1}, \wh t)\mid t,\wh t\in \C\}.
$$
The action of $\delta(s)$ on this line is by 
\begin{align*}
\delta(s).\left( (qw^{-1},t),(\wh q\wh w^{-1}, \wh t)\right)&=\left((qw^{-1}(w.\delta)(s),t),(\wh q\wh w^{-1}(\wh w.\delta)(s),\wh t)\right)\\
&=\left( (qw^{-1},s^{\mu(w\dot\delta)}t),(\wh q\wh w^{-1}, s^{\wh \mu(\wh w\dot\delta)}\wh t)\right)\\
&=s^{\mu(w \dot\delta)+\wh \mu(\wh w \dot\delta)}\left( (qw^{-1},t),(\wh q\wh w^{-1}, \wh t)\right);
\end{align*}
therefore the exponent on $s$ is $0$ (i.e., (a) holds) if and only if the action of $\langle\delta\rangle$ is trivial on any/each fibre (it is constant). 
\end{proof}

Now if there exists a nonzero section $s\in H^0(\operatorname{Fl}_L, \mathcal{M})$, the equivalent conditions above must hold, since the action of $\delta(t)$ will at least be trivial everywhere that the section does not vanish. In this case, the equality 
\begin{align}\label{egalite}
 w^{-1}\mu(\dot \delta)+\wh w^{-1}\wh \mu(\dot \delta)=0
\end{align}
holds. 

With $D(v)$ as before, we note that the section $1$ of $\mathcal{O}(D(v))$ does not vanish when pulled back to $\op{Fl}_L$ since the image of $\op{Fl}_L\to \op{Fl}_G$ misses $D(v)$ (see Lemma \ref{missedme}). Therefore $\tilde i^*\mathcal{O}(D(v))$ satisfies the conditions of the lemma and (\ref{egalite}) holds for $(\mu, \wh\mu)$ such that $\mathcal{O}(D(v)) = \mathcal{L}_{\mu}\boxtimes \mathcal{L}_{\wh \mu}$. That is, $\vec\mu(D(v))$ lies on $\mathcal{F}$. 

\section{Formula for type I rays}\label{chemistry}
In this section we will identify the class of $\mathcal{O}(D(v))$ inside $\pic^G(G/B\times \wh G/\wh B)$. To do so, we first consider its image in $\pic(G/B\times \wh G/\wh B)$, which we identify with its class $[D(v)]$ in the Chow group $A^1(G/B\times \wh G/\wh B)$. 

Second, P. Belkale has suggested the following insightful equivariant technique to solve for the missing ``piece'' of $[D(v)]^G$. Note that $\pic^G(G/B\times \wh G/\wh B)\simeq A^1_G(G/B\times \wh G/\wh B)\simeq A^1_T(\wh G/\wh B)$ comes equipped with a map $\Delta^*$ to $A^1_G(G/B)\simeq A^1_T(pt)$ induced by the diagonal embedding. All told, the composition 
$$
\pic^G(G/B\times \wh G/\wh B) \simeq A^1_T(\wh G/\wh B) \xrightarrow{\phi_\delta^*} A^1_T(G/B)\xrightarrow{\Delta^*} A^1_T(pt) \simeq \mathfrak{h}^*
$$
sends the class of $\mathcal{L}_{\mu}\boxtimes \mathcal{L}_{\wh\mu}$ to the character $\mu+\wh \mu|_T$. 

Say $\mathcal{O}(D(v)) = \mathcal{L}_{\mu'+\chi}\boxtimes \mathcal{L}_{\wh\mu}\in \pic^G(G/B\times \wh G/\wh B)$, where $\mu'$ is in the span of the fundamental weights $\omega_i$ and $\chi$ is a character of $Z^0(G)$ (i.e., vanishes on all simple coroots). Then its image in $\pic(G/B\times \wh G/\wh B)$ is $\mathcal{L}_{\mu'}\boxtimes \mathcal{L}_{\wh \mu}$. Therefore, having first determined $\mu'$ and $\wh \mu$, we solve for $\chi$ by calculating $\Delta^*\phi_\delta^*([D(v)]^G) = \mu+\wh\mu|_T$ (which is manageable) and subtracting $\mu'+\wh \mu|_T$. 

\subsection{Intersection theory setup}
We first determine $[D(v)]$ inside of 
$$
A^1(G/B\times \wh G/\wh B) = \left[A^1(G/B)\otimes A^0(\wh G/\wh B)\right] \oplus \left[A^0(G/B)\otimes A^1(\wh G/\wh B)\right].
$$ 
Since $[D(v)] = \pi_*([\tilde D(v)])$, it suffices to find the components of $[\tilde D(v)]$ in 
$$
A^1(G/B)\otimes A^0(\wh G/\wh B)\otimes A^m(G/P)~~\text{  and  }~~A^0(G/B)\otimes A^1(\wh G/\wh B)\otimes A^m(G/P),
$$
where $m = \dim (G/P)$.

Now $\tilde D(v)$ is, scheme-theoretically, the transverse intersection of 
$$
S'_u:= \{\bar g, \overline{\wh g}, \bar z\in G/B\times \wh G/\wh B\times G/P \mid \bar z\in gX_u\}
$$
and 
$$
\wh S'_{\wh u}:= \{\bar g, \overline{\wh g}, \bar z\in G/B\times \wh G/\wh B\times G/P \mid \phi_\delta(\bar z)\in \wh g\wh X_{\wh u}\},
$$
which are the inverse images of 
$$
S_u:= \{\bar g, \bar z\in G/B\times G/P \mid \bar z\in gX_u\}
$$
and 
$$
\wh S_{\wh u}:= \{\overline{\wh g}, \bar z\in \wh G/\wh B\times G/P \mid \phi_\delta(\bar z)\in \wh g\wh X_{\wh u}\}
$$
under the standard projections. 

Note that $S_u$ has dimension $\dim (G/B)+\ell(u)$, since $S_u\simeq G\times_B X_u$ via $(\bar g, \bar z)\mapsto (g,\overline{g^{-1}z})$. Likewise $\wh T_{\wh u}$ has dimension $\wh G/\wh B + \ell(\wh u)$, where 
$$
\wh T_{\wh u}:= \{\overline{\wh g}, \bar z\in \wh G/\wh B\times \wh G/\wh P \mid \bar z\in \wh g\wh X_{\wh u}\}.
$$ 

Similarly, 
\begin{lemma}
$\dim (\wh S_{\wh u}) = \ell(\wh u) + \dim(G/P) - \dim(\wh G/\wh P)$.
\end{lemma}

\begin{proof}
There is an inclusion $\iota: \wh S_{\wh u}\to \wh T_{\wh u}$ induced by $\phi_\delta$, and we have a proper
intersection
$$
\wh T_{\wh u} \cap \left(\wh G/\wh B\times \phi_\delta(G/P)\right) = \iota(\wh S_{\wh u}).
$$
From this we deduce 
\begin{align*}
\codim(\iota(\wh S_{\wh u}))  &= \codim(\wh T_{\wh u}) + \codim\left(\wh G/\wh B\times \phi_\delta(G/P)\right)\\
&=\dim (\wh G/\wh P)-\ell(\wh u) + \dim(\wh G/\wh P) - \dim (G/P),
\end{align*}
which implies 
$
\dim(\wh S_{\wh u})=\dim (\iota(\wh S_{\wh u})) = \ell(\wh u) + \dim(G/P) - \dim(\wh G/\wh P).
$
\end{proof}

Let $\wh m = \dim(\wh G/\wh P)$. 
Writing $[S_u] = \sum_{j}s_j$, for 
$s_j \in A^{j}(G/B)\otimes A^{m - \ell(u) - j}(G/P)$, and $[\wh S_{\wh u}] = \sum_{k} \wh s_k$, for $\wh s_k\in A^{k}(\wh G/\wh B)\otimes A^{\wh m-\ell(\wh u) - k}(G/P)$, we see that $[S'_u]\cdot [\wh S'_{\wh u}]=p_{13}^*[S_u]\cdot p_{23}^*[\wh S_{\wh u}]$ is supported in 
$$
\bigoplus_{j,k} A^j(G/B)\otimes A^k(\wh G/\wh B)\otimes A^{m+\wh m-\ell(u)-\ell(\wh u) - j - k} (G/P),
$$
and whereas $\ell(u)+\ell(\wh u) = \wh m-1$, we are only interested in the terms where $j+k=1$. 

Applying \cite[\S 4.2]{BKiers}, we have 
\begin{lemma}
$$
[S_u] = 1\otimes [X_u] + \sum_{\ell} \mathcal{L}_{\omega_\ell} \otimes \beta_\ell+\sum_{j\ge 2}s_j,
$$
where $\beta_\ell = [X_{s_{\alpha_\ell}u}]$ if $u\xrightarrow{\alpha_\ell}s_{\alpha_\ell}u\in W^P$ and $\beta_\ell = 0$ otherwise. 

Likewise, 
$$
[\wh T_{\wh u}] = 1\otimes [\wh X_{\wh u}] + \sum_{\ell} \mathcal{L}_{\wh \omega_\ell} \otimes \wh \beta_\ell+\sum_{k\ge 2}\wh t_k,
$$
where $\wh \beta_\ell = [X_{s_{\wh \alpha_\ell}\wh u}]$ if $\wh u\xrightarrow{\wh \alpha_\ell}s_{\wh \alpha_\ell}\wh u\in W^{\wh P}$ and $\wh \beta_\ell = 0$ otherwise, and where $\wh t_k\in A^k(\wh G/\wh B)\otimes A^{\wh m -\ell(\wh u)-k}(\wh G/\wh P)$ pulls back to $\wh s_k$. 
\end{lemma}

\subsection{Proof of Theorem \ref{formulaONE}(a)}

Finally we may calculate
\begin{align*}
[\tilde D(v)] &= [S']\cdot [\wh S'] 
= p_{13}^*[S_u]\cdot p_{23}^*[\wh S_{\wh u}]
= p_{13}^*[S_u]\cdot p_{23}^*\phi_\delta^*[\wh T_{\wh u}]\\
&= \sum_{\wh u\xrightarrow{\wh \alpha_\ell}s_{\wh \alpha_\ell}\wh u\in W^{\wh P}} 1\otimes \mathcal{L}_{\wh \omega_k} \otimes \left([X_u] \cdot \phi_\delta^*([\wh X_{s_{\wh \alpha_\ell}\wh u}])\right)\\
&+ 
\sum_{u\xrightarrow{\alpha_\ell}s_{\alpha_\ell}u\in W^P} \mathcal{L}_{\omega_\ell}\otimes 1\otimes \left([X_{s_{\alpha_\ell} u}]\cdot \phi_\delta^*([\wh X_{\wh u}])\right)\\
&+ \sum_{j+k>1} p_{13}^*(s_j)\cdot p_{23}^*(\wh s_k).
\end{align*}

The result follows from taking $\pi_*$ of both sides, since each term of the third sum belongs to some $A^j(G/B)\otimes A^k(\wh G/\wh B)\otimes A^{n}(G/P)$ with $n=m+1-j-k<m$ and thus is sent to $0$.

\subsection{Proof of Theorem \ref{formulaONE}(b)}

Consider the following commutative diagram. 

\begin{center}
\begin{tikzcd}
G/B\times G/P \arrow[d,"p_1"] \arrow[rr,"\Delta\times\op{id}"] & & G/B\times \wh G/\wh B\times G/P \arrow[d,"\pi"] \arrow[rr,"\phi_\delta\circ p_{23}"] \arrow[drr,"p_{13}"]& & \wh G/\wh B\times \wh G/\wh P \\
G/B \arrow[rr,"\Delta"]& & G/B\times \wh G/\wh B & & G/B\times G/P
\end{tikzcd}
\end{center}

Since $S_u$ and $\wh T_{\wh u}$ are $G$- and $\wh G$-stable, we have an equivariant version of the previous convolution calculation: 
 $$
 [D(v)]^G = \pi_*(p_{13}^*[S_u]^G \cdot p_{23}^* \phi_\delta^* [\wh T_{\wh u}]^{\wh G}),
$$
where the pullback induced by $\phi_\delta^*$ now includes the restriction of the group from $\wh G$ to $G$. 
In the cartesian square, we have $\Delta^*\pi_* = p_{1,*}(\Delta\times \op{id})^*$. Furthermore, $p_{13}\circ(\Delta\times \op{id}) = \op{id}$ and $\phi_\delta\circ p_{23}\circ (\Delta\times \op{id}) = \phi\times\phi_\delta$, where $\phi$ is the embedding $G/B\to \wh G/\wh B$. 
So we have 
$$
\Delta^* [D(v)]^G = p_{1,*}\left([S_u]^G\cdot (\phi\times \phi_\delta)^*[\wh T_{\wh u}]^{\wh G}\right).
$$
Typically, $p_{1,*}$ is denoted by $\int_{G/B}$. Under the identification $A^1_G(G/B\times G/P)\simeq A^1_T(G/P)$, the class $[S_u]^G = [\overline{G(e,u)}]^G$ corresponds to $[X_u]^T$ (see \cite[\S 6.6]{Brion}, wherein the argument for $G/B$ applies also for $G/P$). Likewise, the class $[\wh T_{\wh u}]^{\wh G} = [\wh X_{\wh u}]^{\wh T}$, and the $G$-equivariant pullback $(\phi\times \phi_\delta)^*$ becomes the $T$-equivariant pullback $\phi_\delta^*$, understood as first restricting $A^*_{\wh T}(\wh G/\wh P)\to A^*_T(\wh G/\wh P)$. Therefore we have 
$$
\mu+\wh \mu_T = \Delta^*[D(v)]^G = \int_{G/B} [X_u]^T\cdot \phi_\delta^*[\wh X_{\wh u}]^{\wh T}
$$
as desired. 

\section{Decomposition of $\mathcal{F}$ into subcones}\label{redox}

Having found all possible type I rays $\vec \mu(D(v))$ on $\mathcal{F}$, there may (and generally will) be more extremal rays of $\mathcal{F}$; these will span some proper subcone, which is easily identified after the following lemmas. 

\begin{lemma}\label{ones}
Let $(\mu, \wh \mu) = \vec\mu(D(v))$ be a type I ray corresponding to the data $v\xrightarrow{\alpha_\ell}w$ (resp., $v\xrightarrow{\wh \alpha_\ell}\wh w$). Then $\mu(\alpha_\ell^\vee)=1$ (resp., $\wh \mu(\wh\alpha_{\ell}^\vee)=1$). 
\end{lemma}

\begin{proof}
Obvious from $[X_u]\odot_0 \phi_\delta^{\odot}([\wh X_{\wh u}]) = [X_e] \implies [X_u]\cdot \phi_\delta^*([\wh X_{\wh u}]) = [X_e]$.
\end{proof}

\begin{lemma}\label{zeros}
Let $(\mu, \wh \mu) = \vec\mu(D(v))$ be a type I ray corresponding to the data $v\xrightarrow{\alpha_\ell}w$ or $v\xrightarrow{\wh \alpha_\ell}\wh w$. If $v'\xrightarrow{\alpha_{\ell'}}w$ (resp., $v'\xrightarrow{\wh \alpha_{\ell'}}\wh w$) is a distinct datum defining another type I ray, then we have 
$\mu(\alpha_{\ell'}^\vee)=0$ (resp., $\wh \mu(\wh \alpha_{\ell'}^\vee) = 0$).
\end{lemma}

\begin{proof}
Suppose $D(v)$ comes from the data $v\xrightarrow{\alpha_\ell}w$; the other case will follow similarly. 

If $v'\xrightarrow{\wh \alpha_{\ell'}}\wh w$, then it is not the case that $\wh u \xrightarrow{\wh \alpha_{\ell'}} s_{\wh \alpha_{\ell'}}\wh u$ since $\wh u = \wh w$; therefore $\wh c_{\ell'}=0$. Otherwise, $v'\xrightarrow{\alpha_{\ell'}}w$ for $\ell'\ne \ell$. Then, as in \cite[Lemma 2.4]{BGG}, 

\begin{center}
\begin{tikzcd}
v' \arrow[dr, "\alpha_{\ell'}"] & & & & v' \\
 & w = s_{\alpha_{\ell'}}v' & \implies & s_{\alpha_{\ell'}}v \arrow[ur, "s_{\alpha_{\ell'}}\alpha_\ell"] \arrow[dr, "\alpha_{\ell'}"'] & \\
 v \arrow[ur, "\alpha_\ell"'] & & & & v
\end{tikzcd}
\end{center}
and in particular $s_{\alpha_{\ell'}}v \xrightarrow{\alpha_{\ell'}}v$. This prevents $u\xrightarrow{\alpha_{\ell'}} s_{\alpha_{\ell'}}u$ since $u=v$, and $c_{\ell'}=0$. 
\end{proof}

Set $\mathcal{F}_2\subseteq\mathcal{F}$ to be the set 
\begin{align}\label{fide}
\mathcal{F}_2 = \{(\mu,\wh \mu)\in \mathcal{F} \mid \mu(\alpha_{\ell}^\vee)=0 ~\forall v\xrightarrow{\alpha_\ell} w \text{ and } \wh\mu(\wh\alpha_{\ell}^\vee)=0 ~\forall v\xrightarrow{\wh\alpha_{\ell}}\wh w\}.
\end{align}

Likewise define $\mathcal{F}_{2,\Q}\subseteq \mathcal{F}_\Q$.  Evidently $\mathcal{F}_2$ is a subsemigroup of $\mathcal{F}$ and contains none of the rays $\vec\mu(D(v))$, by Lemma \ref{ones}. Furthermore, the rays $\vec\mu(D(v))$ are linearly independent by Lemma \ref{zeros}; each has some coordinate equal to $1$ where all others equal $0$. We therefore have a natural injection of semigroups
$$
\prod \Z_{\ge0} \vec\mu(D(v))\times \mathcal{F}_2 \hookrightarrow \mathcal{F}.
$$

We now prove Theorem \ref{1+2}: 
\begin{proposition}
The preceding map is also a surjection. 
\end{proposition}

\begin{proof}
Let $(\nu, \wh \nu)\in \mathcal{F}\setminus \mathcal{F}_2$, and, possibly scaling by $N$ assume $H^0(G/B\times \wh G/\wh B,\mathcal{L})^G\ne (0)$, where $\mathcal{L} = L_{\nu}\boxtimes L_{\wh \nu}$. Being outside of $\mathcal{F}_2$, it holds that $\nu(\alpha_\ell^\vee)> 0$ (or $\wh \nu(\wh\alpha_\ell^\vee)> 0$, the proof will be analogous) for some $\ell$ giving a type I datum. 

Choose a nonzero $G$-invariant section $s\in H^0(G/B\times \wh G/\wh B,\mathcal{L})$. For any point $(\bar g, \overline{\wh g})\in D(v)$, where $v\xrightarrow{\alpha_\ell} w$, we have
$$
\phi_\delta(gX_v)\cap \wh g\wh X_{\wh w} \ne \emptyset.
$$
In an open subset of $D(v)$, then, we actually have 
$$
\phi_\delta(gC_v)\cap \wh g\wh C_{\wh w} \ne \emptyset,
$$
and we now choose $g,\wh g$ to be such. Assume for contradiction that $s$ does not vanish at $(g,\wh g)$. Then by some standard invariant theory, we must have 
\begin{align}\label{funk}
v^{-1}\nu(\dot\delta)+\wh w^{-1}\wh \nu(\dot\delta)\le 0
\end{align}
(see \cite[\S 5.2]{BKiers}). I claim this cannot be. 

Indeed, $w^{-1}\nu(\delta) + \wh w^{-1}\wh \nu(\delta)=0$ by definition of $\mathcal{F}$. Furthermore, 
\begin{align*}
v^{-1}\nu(\dot\delta) - w^{-1}\nu(\dot\delta)&=v^{-1}(\nu-s_{\alpha_\ell}\nu)(\dot\delta)\\
&=v^{-1}(\nu(\alpha_\ell^\vee)\alpha_\ell)(\dot\delta)\\
&=\nu(\alpha_\ell^\vee)\cdot v^{-1}(\alpha_\ell)(\dot\delta),
\end{align*}
and we know $\nu(\alpha_\ell^\vee)$ to be a positive integer by assumption. Since $\ell(v)<\ell(s_{\alpha_\ell}v)$, $v^{-1}\alpha_\ell$ is a positive root (\cite[Corollary 2.3]{BGG}). This gives $(v^{-1}\alpha_\ell)(\dot\delta)\ge 0$; equality would hold only if $v^{-1}\alpha_\ell$ were in the root system for $L$. However, $wv^{-1}\alpha_\ell = s_{\alpha_\ell}\alpha_\ell  = -\alpha_\ell\prec 0$ and $w\in W^P$ means $v^{-1}\alpha_\ell$ can't be in the root system for $L$ (see \cite[\S 2.5]{BL}). 

Therefore $0<v^{-1}\nu(\dot\delta)-w^{-1}\nu(\dot\delta) = v^{-1}\nu(\dot\delta)+\wh w^{-1}\wh \nu(\dot\delta)$, which violates inequality (\ref{funk}). We conclude that $s$ vanishes on an open subset of $D(v)$, in which case $s$ vanishes totally on $D(v)$. This implies that $s$ induces a nonzero invariant section of $\mathcal{L}(-D(v))$. 

If $(\nu', \wh \nu')$ represents $\mathcal{L}(-D(v))$, then $(\nu',\wh \nu')\in \mathcal{F}$ and has $\nu'(\alpha_\ell) = \nu(\alpha_\ell)-1$. Furthermore, $(\nu', \wh \nu')$ agrees with $(\nu, \wh \nu)$ on all other relevant $\alpha_{\ell'}^\vee$ or $\wh \alpha_{\ell'}^\vee$. This style of reduction may be continued, then, to reach an element $\mathcal{L}'\in \mathcal{F}_2$ in finitely many steps, whose difference from $\mathcal{L}$ is in the span of the type I rays. 

Now, if we did indeed need to scale $(\nu, \wh \nu)$ by $N$ at the beginning, each of the subtracted $\vec{\mu}(D(v))$ must have been subtracted a multiple of $N$ times. That is, the resulting element of $\mathcal{F}_2$ has coefficients each divisible by $N$; thus we can scale back down to an element of $\mathcal{F}_2$ as desired. 
\end{proof}

\section{More stacks and the geometry of $\mathcal{F}_2$}\label{zoom}

In this section, we identify the cone $\mathcal{F}_2$ as a rational semigroup of line bundles on a new stack. This will allow us to relate to the cone $\mathcal{C}(L/\langle \delta\rangle\subset \wh L/\langle \delta\rangle)$. We begin by introducing a new pair of stacks. 

\subsection{The stack $\op{Fl}_G'$}

Recall that $\mathcal{F}_2$ is contained in the subspace of $\mathfrak{h}^*\times \wh{\mathfrak{h}}^*$ cut out by the vanishing conditions 
$$
v\xrightarrow{\beta} w \implies \mu(\beta^\vee)=0
$$
and 
$$
v\xrightarrow{\wh \beta} \wh w\implies \wh \mu(\wh \beta^\vee)=0,
$$
where $\beta$ and $\wh \beta$ are simple roots in their respective root systems. Therefore if $(\mu, \wh \mu)\in \mathcal{F}_2$, the line bundle $\mathcal{L}_\mu$ on $G/B$ descends naturally to $G/Q_w'$, where $Q_w'$ is the standard parabolic given by 
$$
\Delta(Q_w') = \Delta\cap w\Phi^-.
$$
Similarly, the line bundle $\mathcal{L}_{\wh \mu}$ on $\wh G/\wh B$ descends to $\wh G/\wh Q_{\wh w}'$, where $\Delta(\wh Q_{\wh w}') = \wh \Delta\cap \wh w\wh \Phi^-$. Conversely, any $G$-linearized line bundle on $G/Q_w' \times \wh G/\wh Q_{\wh w}'$ gives $(\mu, \wh \mu)$ satisfying these vanishing conditions. 

Now, $Q_w'\subseteq Q_w$ and $\wh Q_{\wh w}'\subseteq \wh Q_{\wh w}$; this follows from examining $\Delta(Q_w)$, $\Delta(\wh Q_{\wh w})$ as in \cite[Lemma 7.1]{BKR}.

They also satisfy 
$$
B_L\subseteq w^{-1}Q_w' w~~\text{  and  }~~\wh B_{\wh L}\subseteq \wh w^{-1}\wh Q_{\wh w}'\wh w;
$$
this is clear since $B$ and $\wh B$ satisfy this. Therefore the map 
$$
L/B_L\times \wh L/\wh B_{\wh L} \to G/Q_w'\times \wh G/\wh Q_{\wh w}'
$$
given by $(\bar q,\overline{\wh q})\mapsto (\overline{qw^{-1}}, \overline{\wh q\wh w^{-1}})$ is well-defined and factors through the projection $G/B\times \wh G/\wh B\to G/Q_w'\times \wh G/\wh Q_{\wh w}'$. On the level of stacks, one easily checks that this induces a map (factoring through $\op{Fl}_G$)
$$
\tilde i':\op{Fl}_L\to \op{Fl}_G',
$$
where $\op{Fl}_G'$ is the quotient stack $\left(G/Q_w'\times \wh G/\wh Q_{\wh w}'\right)/G$.

Let us record a few definitions.

\begin{defi}
For any stack $X$, we set 
\begin{align*}
\pic^+(X) &= \textnormal{the semigroup of line bundles with non-zero global sections;}\\
\pic^+_\Q(X) &= \{\mathcal{L}\in \pic(X)\otimes \Q \mid \mathcal{L}^{\otimes N}\in \pic^+(X) \textnormal{ for some $N>0$}\}.
\end{align*}
Furthermore, we set 
\begin{align*}
\pic^{\deg=0}(\op{Fl}_G')
\end{align*}
to be the subgroup consisting of line bundles whose pullback to $\op{Fl}_L$ have trivial $\delta$-action on fibres.
\end{defi}

In light of these definitions, we have the following identification. 
\begin{lemma}
As rational cones, 
$$
\pic^{+,\deg=0}_\Q(\op{Fl}_G')\simeq \mathcal{F}_{2,\Q}.
$$
\end{lemma}
\begin{proof}
The only unmentioned aspect so far is that $\deg =0$ exactly characterizes the facet equality defining $\mathcal{F}$; cf. Lemma \ref{florence}.
\end{proof}

\subsection{The stack $\wh{\mathcal{C}}'$}

We now introduce $\wh{\mathcal{C}}'$ and see how it interacts with $\op{Fl}_G'$. 

\begin{defi}
Let $\mathcal{X}'$ be the universal intersection scheme given set-theoretically by 
$$
\mathcal{X}' = \{(\bar g, \overline{\wh g},\bar z)\in G/Q_w' \times \wh G/\wh Q_{\wh w}' \times G/P \mid\phi_\delta(\bar z) \in \phi_\delta(gX_w)\cap \wh g\wh X_{\wh w}\}.
$$
\end{defi}

Note that this definition is valid since $Q_w'\subseteq Q_w$, which stabilizes $X_w$, and the same for their analogues w.r.t. $\wh G$. Now, set $C_w' = Q_w'wP$ and $\wh C_{\wh w}' = \wh Q_{\wh w}'\wh w\wh P$. By replacing $X_w, \wh X_{\wh w}$ with $Z_w, \wh Z_{\wh w}$ and with $C_w', \wh C_{\wh w}'$, respectively, we similarly define (open) intersection subloci $\mathcal{Z}'\supseteq\mathcal{C}'$. 

Set $\wh{\mathcal{C}}'$ to be the stack $\wh{\mathcal{C}}'/G$, which parametrizes principal $G$-spaces $E$ with elements $\bar g\in E/Q_w'$, $\overline{\wh g}\in (E\times_G \wh G)/\wh Q_{\wh w}'$, and $\bar z \in E/P$ such that $z\in gC_w'$ and $(z,e)\in \wh g\wh C_{\wh w}'$. Equivalently, as before, it parametrizes principal $P$-spaces $E'$ together with elements $\bar y\in E'/(w^{-1}Q_w' w\cap P)$ and $\overline{\wh y} \in (E'\times_P \wh P)/(\wh w^{-1}\wh Q_{\wh w}'\wh w\cap \wh P)$. 

The natural projection $\pi': \mathcal{X}' \to G/Q_w'\times \wh G/\wh Q_{\wh w}'$ is birational, and we use $\mathcal{R}'$ to denote the ramification locus inside $\mathcal{Z}'$ (or $\mathcal{C}'$). Our new diagram of stacks is 

\begin{center}
\begin{tikzcd}
\mathcal{C}' \arrow[d,"\pi'"]  \arrow[r] & \wh{\mathcal{C}}' \arrow[d, "\pi'"] \arrow[dr, "\tau'"'] & \\
G/Q_w'\times \wh G/\wh Q_{\wh w}' \arrow[r] & \op{Fl}_G' & \op{Fl}_L \arrow[l, "\tilde i'"] \arrow[ul, bend right=20, "i'"']
\end{tikzcd}
\end{center}

We prove the analogue of \cite[Lemma 42]{BKiers}. 

\begin{lemma}
The closed subvariety $\overline{\pi'(\mathcal{X}'\setminus \mathcal{C}')}$ is of codimension $\ge2$ inside $G/Q_w'\times \wh G/\wh Q_{\wh w}'$. 
\end{lemma}

\begin{proof}
$\mathcal{X}$ is a fibre-product of other spaces in question:
\begin{center}
\begin{tikzcd}
\mathcal{X} \arrow[r,"\tilde\phi"] \arrow[dd,"\pi"]& \mathcal{X}' \arrow[dd,"\pi'"]\\ \\
G/B\times \wh G/\wh B \arrow[r,"\phi"]& G/Q_w' \times \wh G/\wh Q_{\wh w}'
\end{tikzcd}
\end{center}

One easily checks that $\phi$ is a smooth fibre bundle over a smooth base; the fibres are $Q_w'/B\times \wh Q_{\wh w}'/\wh B$. Thus $\phi^{-1}(\pi'(\mathcal{X}'\setminus\mathcal{C}'))$ has the same codimension as $\pi'(\mathcal{X}'\setminus\mathcal{C}')$ (and as the latter's closure). 

The argument of \cite[Remark 8]{BKiers} is still valid in this case, and we have $\phi^{-1}(\pi'(\mathcal{X}'\setminus \mathcal{C}')) = \pi(\tilde\phi^{-1}(\mathcal{X}'\setminus \mathcal{C}'))$. Let us examine $\tilde \phi^{-1}(\mathcal{X}'\setminus\mathcal{C}')$, or, rather, $\tilde \phi^{-1}(\mathcal{C}')$. 

If $(\bar g,\overline{\wh g}, \bar z)$ maps into $\mathcal{C}'$, then $\phi_\delta(z)\in \phi_\delta(gC_w')\cap \wh g\wh C_{\wh w}'\subseteq \phi_\delta(gY_w)\cap \wh g\wh Y_{\wh w}$. That is, $\tilde\phi^{-1}(\mathcal{C}')\subseteq\mathcal{Y}$; furthermore, the codimension of the complement of $C_w'$ inside $Y_w$ is $\ge 2$ (see \cite[Lemma 41]{BKiers}), and the same holds for the associated $\wh G$ spaces. We conclude that $\tilde \phi^{-1}(\mathcal{C}')$ has complement codimension $\ge 2$ inside $\mathcal{Y}$. Thus if we could show $\pi(\mathcal{X}\setminus \mathcal{Y})$ has codimension $\ge 2$, we would have the desired result. 

For this, we recall that $\mathcal{Z}\setminus\mathcal{Y}\subset \mathcal{R}\cup A$ for some codimension $\ge 2$ $A\subset \mathcal{X}$. We find that
$$
\mathcal{X}\setminus \mathcal{Y} \subseteq \mathcal{X}\setminus \mathcal{Z} \cup \mathcal{Z}\setminus \mathcal{Y} \subseteq \mathcal{X}\setminus \mathcal{Z}\cup \mathcal{R}\cup A,
$$
and everything on the right is mapped to codimension $\ge 2$ in $G/B\times \wh G/\wh B$ under $\pi$. This completes the argument.
\end{proof}

\begin{corollary}
Let $\mathcal{R}'$ be the ramification locus of $\pi':\mathcal{C}'\to G/Q_w'\times \wh G/\wh Q_{\wh w}'$. Then, restricted to $\mathcal{C}'\setminus \mathcal{R}'$, $\pi'$ is an open embedding whose image has complement of codimension $\ge 2$. 
\end{corollary}

This allows us to conclude: 
\begin{proposition}
$\pi'^*$ induces an isomorphism $\pic(\op{Fl}_G')\simeq \pic(\wh{\mathcal{C}}'\setminus \wh{\mathcal{R}}')$. 
\end{proposition}

\begin{proof}
Identical to that of \cite[Corollary 50]{BKiers}. 
\end{proof}

\subsection{Connection with the Levi subgroup}

The family of maps $\psi_t:P\to P$ and Levification procedure carry forward to the present case, and Proposition \ref{opendoor} has the following analogue (the proof is the same): 

\begin{proposition}\label{aGain}
For any non-empty open substack $U$ of $\op{Fl}_L$ and any line bundle $\mathcal{L}$ on $\tau'^{-1}(U)$, setting $\mathcal{M} = i'^*\mathcal{L}$, we have 
\begin{enumerate}[label=(\alph*)]
\item $\mathcal{L} = \tau'^*\mathcal{M}$. Thus as before, $\tau'^*$ and $i'^*$ give inverse isomorphisms $\pic(U)\simeq \pic(\tau'^{-1}(U))$. 
\item In the case that $\gamma_\mathcal{M}$ is trivial, 
$$
i'^*:H^0(\tau'^{-1}(U), \mathcal{L}) \to H^0(U,\mathcal{M})
$$
is also an isomorphism. 
\end{enumerate}
\end{proposition}

Let $\wh{\mathcal{R}}'$ be the locus of $(E',\bar y, \overline{\wh y})\in \wh{\mathcal{C}}'$ whose determinant lines of 
$$
E'\times_PT_{\dot e}(G/P)\to \dfrac{E'\times_PT_{\dot e}(G/P)}{\{y\}\times T_{\dot e}(w^{-1}C_w')} \oplus \dfrac{E'\times_PT_{\dot e}(\wh G/\wh P)}{\{\wh y\} \times T_{\dot e}(\wh w^{-1} \wh C_{\wh w}')}
$$
vanish. Set $\mathcal{R}_L$ to be the inverse image of $\wh{\mathcal{R}}'$ under $i'$; consequently $i'^*\mathcal{O}(\wh{\mathcal{R}}') = \mathcal{O}(\mathcal{R}_L)$. 

\begin{lemma}\label{aoe}
For $\mathcal{M} = \mathcal{O}(\mathcal{R}_L)$, $\gamma_\mathcal{M}$ is trivial. 
\end{lemma}

\begin{proof}
Let $\left(\bar l, ~\widebar{\wh l}~\right)\in L/B_L\times \wh L/\wh B_{\wh L}$ be arbitrary. The fibre of $\mathcal{M}$ over this point is the determinant line of 
$$
T_{\dot e}(G/P)\to \dfrac{T_{\dot e}(G/P)}{T_{\dot e}(lw^{-1}C_w')}\oplus \dfrac{T_{\dot e}(\wh G/\wh P)}{T_{\dot e}(\wh l\wh w^{-1}\wh C_{\wh w}')}.
$$
The nonzero deformed pullback product (\ref{prod}) implies that the pair $(w, \wh w)$ is Levi-movable by \cite[Proposition 2.3]{RessRich}. Therefore the above map is an isomorphism (hence nonzero determinant line) for generic $\left(\bar l, ~\widebar{\wh l}~\right)$. Thus the natural $\theta$-section gives a nonzero global section of $\mathcal{M}$; this forces $\gamma_\mathcal{M}$ to be trivial. 
\end{proof}

We wish to prove the following proposition, which will be needed to define the induction map in the next section. 

\begin{proposition}\label{crick}
$\pic(\op{Fl}_L\setminus \mathcal{R}_L)\simeq \pic(\wh{\mathcal{C}}'\setminus \wh{\mathcal{R}}')$.
\end{proposition}

\begin{proof}
The statement follows from Proposition \ref{aGain}(a), provided we show that $\wh{\mathcal{R}}' = \tau'^{-1}\mathcal{R}_L$. Indeed, choose a section $\sigma \in H^0(\wh{\mathcal{C}}',\mathcal{O}(\wh{\mathcal{R}}'))$ which vanishes exactly on $\wh{\mathcal{R}}'$. Then $i'^*\sigma$ vanishes exactly on $\mathcal{R}_L$. We have 
$$
i'^*(\tau'^*(i'^*(\sigma))) = (\tau'\circ i')^*i'^*(\sigma) = i'^*\sigma,
$$
and $i'^*$ is injective (since by Lemma \ref{aoe} and Proposition \ref{aGain}(b)), so $\sigma = \tau'^*(i'^*(\sigma))$, which vanishes exactly on $\tau^{-1}\mathcal{R}_L$. 
\end{proof}

\subsection{Reduction from $L$ to $L/\langle\delta\rangle$}

Let $L_\delta$ be the quotient group $L/\langle \delta \rangle$ and $B_{\delta}$ denote the image of $B_L$ under the surjective homomorphism $L\to L_\delta$; it is a Borel subgroup for the latter reductive group.  We define $\wh L_\delta$ and $\wh B_\delta$ the same way. The natural $L$-equivariant morphism of flag varieties
$$
L/B_{L}\times \wh L/\wh B_{\wh L}\to L_\delta/B_\delta\times \wh L_\delta/\wh B_{\delta}
$$
is an isomorphism. 

Further, this induces a morphism of stacks
$
\epsilon: \op{Fl}_{L}  \to \op{Fl}_{L_\delta}.
$
Our next lemma records the essential relationship between line bundles on $\op{Fl}_{L}$ and $\op{Fl}_{L_\delta}$, but first we make the following definition. 
\begin{defi}
Let 
$
\pic^{\deg =0}(\op{Fl}_L)
$
denote the subgroup of $\pic(\op{Fl}_L)$ with trivial $\delta$-action on the fibres.
\end{defi}

\begin{lemma}\label{interview}
\begin{enumerate}[label=(\alph*)]
\item $\epsilon^*: \pic(\op{Fl}_{L_\delta}) \to \pic(\op{Fl}_L)$ is injective, with image equal to $\pic^{\op{deg}=0}(\op{Fl}_L)$.
\item The preceding isomorphism restricts to 
$$
\pic_\Q^{+}(\op{Fl}_{L_\delta})\xrightarrow{\sim} \pic_\Q^{+,\op{deg}=0}(\op{Fl}_{L}),
$$
which we shall call $\epsilon_+^*$. 
\end{enumerate}
\end{lemma}

\begin{proof}
\begin{enumerate}[label=(\alph*)]
\item Every line bundle $\mathcal{L}$ on $X_\delta = L_\delta/B_\delta\times \wh L_\delta/\wh B_{\delta}$ which is $L_\delta$-linearized is naturally $L$-linearized via $\psi: L\to L_\delta$. Moreover, the image of $\delta$ lies in the kernel of $\psi$; hence the action of $\delta$ is trivial on fibres. In the other direction, any $L$-linearization of a line bundle $\mathcal{L}$ on $X_\delta$ which has trivial $\delta$-action descends naturally to an $L_\delta$-linearization. 
\item As long as the $\delta$-action is trivial, $L$-equivariant global sections are the same as $L_\delta$-equivariant global sections. 
\end{enumerate}
\end{proof}

\section{Induction and type II rays}\label{wellordering}

Here we give an alternate definition of the map $\op{Ind}$ of Theorem \ref{yadhtrib}; in the next section we will show that they are the same. 

\begin{defi}Define the induction map by the composition 

\begin{center}
\begin{tikzcd}
\op{Ind}: \pic_\Q(\op{Fl}_{L_\delta}) \arrow[r, "\epsilon^*", "\sim"'] & ~\pic_\Q^{\deg =0}(\op{Fl}_L) \arrow[r, twoheadrightarrow, "\iota^*"] & ~\pic_\Q^{\deg=0}(\op{Fl}_L\setminus\mathcal{R}_L) \\
 ~ \arrow[r, "\tau'^*", "\sim"'] & ~\pic_\Q^{\deg =0}(\wh{\mathcal{C}}'\setminus\wh{\mathcal{R}}') \arrow[r, "(\pi^*)^{-1}", "\sim"'] & ~\pic_\Q^{\deg=0}(\op{Fl}_G'). &
\end{tikzcd}
\end{center}
\end{defi}

All maps are isomorphisms or surjections as indicated except possibly that $\iota^*$ is surjective; this follows exactly as in \cite[Lemma 54]{BKiers}. 

Recall that $\mathcal{F}_{2,\Q}$ is $\pic_\Q^{+,\deg=0}(\op{Fl}_G')$ and $\mathcal{C}(L_\delta \subset \wh L_\delta)_\Q$ is $\pic_\Q^+(\op{Fl}_{L_\delta})$.  What we need now is 
\begin{proposition}
The map $\op{Ind}$ restricts to a well-defined surjection 
\begin{center}
\begin{tikzcd}
\op{Ind}: \pic_\Q^+(\op{Fl}_{L_\delta}) \arrow[r, twoheadrightarrow] & ~\pic_\Q^{+,\deg=0}(\op{Fl}_G').
\end{tikzcd}
\end{center}
\end{proposition}

\begin{proof}
First, 
\begin{align}\label{jaz1}
\pic_\Q^+(\op{Fl}_{L_\delta})\simeq \pic_\Q^{+,\deg=0}(\op{Fl}_L)
\end{align}
via $\epsilon_+^*$ by Lemma \ref{interview}(b). 
We will return to $\iota^*$ momentarily. 

Second, for $\mathcal{M}\in \pic_\Q^{+,\deg =0}(\op{Fl}_L\setminus\mathcal{R}_L)$, Proposition \ref{aGain} tells us that $ H^0(\op{Fl}_L\setminus\mathcal{R}_L,\mathcal{M}) \simeq H^0(\wh{\mathcal{C}}'\setminus\wh{\mathcal{R}}',\tau'^*\mathcal{M})$ via $\tau'^*$. So $\tau'^*$ restricts to an isomorphism 
\begin{align}\label{jaz2}
\pic_\Q^{+,\deg=0}(\op{Fl}_L\setminus\mathcal{R}_L)\simeq \pic_{\Q}^{+,\deg=0}(\wh{\mathcal{C}}'\setminus\wh{\mathcal{R}}'),
\end{align}
whose inverse is $i'^*$. 

Third, the isomorphism $\pi^*$ must also induce isomorphisms on the level of global sections, because sections can be extended across codimension $\ge 2$. That is, for $\mathcal{L}$ in $\pic_\Q^{+,\deg=0}(\op{Fl}_G')$, $H^0(\op{Fl}_G',\mathcal{L})\simeq H^0(\wh{\mathcal{C}}'\setminus\wh{\mathcal{R}}',\pi^*\mathcal{L})$. Thus $\pi^*$ restricts to an isomorphism 
\begin{align}\label{jaz3}
\pic_\Q^{+,\deg=0}(\op{Fl}_G')\simeq \pic_\Q^{+,\deg=0}(\wh{\mathcal{C}}'\setminus\wh{\mathcal{R}}').
\end{align}

Finally, take $\mathcal{M}$ in $\pic_\Q^{+,\deg=0}(\op{Fl}_L)$. Then $\tilde{i}'^*\op{Ind}(\mathcal{M})$ lives in $\pic^{\deg=0}_\Q(\op{Fl}_L)$, and the two agree on $\op{Fl}_L\setminus\mathcal{R}_L$ (just check on stalks). The restriction map 
$$
H^0(\op{Fl}_L,\tilde{i}'^*\op{Ind}(\mathcal{M})) \to H^0(\op{Fl}_L\setminus\mathcal{R}_L,\mathcal{M})
$$
is an injection: say a section $\sigma$ vanishes away from $\mathcal{R}_L$; it is supported only on $\mathcal{R}_L$. Sections of $\pi^*\op{Ind}(\mathcal{M})$ are supported on $\wh{\mathcal{C}}\setminus\wh{\mathcal{R}}$, so those of $\tilde{i}'^*\op{Ind}(\mathcal{M})$ are supported away from $\mathcal{R}_L$. We must conclude that $\sigma = 0$; i.e., the map is injective. 

Consider then the diagram:

\begin{center}
\begin{tikzcd}
H^0(\op{Fl}_G',\op{Ind}(\mathcal{M})) \arrow[rd, "{\rotatebox[origin=cc]{-17}{$ \sim $}}"] \arrow[d]&\\
H^0(\op{Fl}_L,\tilde{i}'^*\op{Ind}(\mathcal{M})) \arrow[r] & H^0(\op{Fl}_L\setminus \mathcal{R}_L, \mathcal{M})
\end{tikzcd}
\end{center}
As the horizontal map is an injection, all maps in sight must be isomorphisms. As a consequence, the map $\pic_\Q^{\deg=0}(\op{Fl}_L\setminus \mathcal{R}_L)\to \pic_\Q^{\deg=0}(\op{Fl}_L)$ via $\mathcal{M}\mapsto \tilde{i}'^*\op{Ind}\mathcal{M}$ takes bundles with nonzero global sections to the same, thus providing a section for the surjection 
\begin{center}
\begin{tikzcd}
~\pic_\Q^{+,\deg=0}(\op{Fl}_L) \arrow[r,twoheadrightarrow] &  ~\pic_\Q^{+,\deg=0}(\op{Fl}_L\setminus\mathcal{R}_L).
\end{tikzcd}
\end{center}

Combining this with the isomorphisms (\ref{jaz1}), (\ref{jaz2}), and (\ref{jaz3}) gives the result. 
\end{proof}

\section{Formula for induction}\label{chemII}

As a corollary to the previous section, every extremal ray of $\mathcal{F}_{2,\Q}$ is the image of an extremal ray of $\mathcal{C}(L_\delta\subset \wh L_{\delta})_\Q$. This is because the map $\op{Ind}$ is $\Q$-linear. One might ask whether the embedding $L_\delta \subset \wh L_\delta$ is of the same class as $G\subset \wh G$, in order to decide whether this induction is really a fair burden on the reader. 

\begin{proposition}
Let $\delta\in \mathfrak{S}$. Then (\ref{assumB}) holds for $L/\langle \delta\rangle\hookrightarrow \wh L/\langle\delta\rangle$. 
\end{proposition}
\begin{proof}
Suppose that a nontrivial ideal $I_1$ of $\mathfrak{l}/\C\dot\delta$ is also a nontrivial ideal of $\wh{\mathfrak{l}}/\C\dot\delta$. Decompose
\begin{align*}
\mathfrak{l} &= \C\dot\delta \oplus I_1 \oplus \mathfrak{m}\\
\wh{\mathfrak{l}} &= \C\dot\delta \oplus I_1 \oplus \wh{\mathfrak{m}},
\end{align*}
for suitable reductive Lie algebras $\mathfrak{m}, \wh{\mathfrak{m}}$. 
The $\mathfrak{h}$-weights of $\wh{\mathfrak{l}}/\mathfrak{l}$ coincide with the set of $\mathfrak{h}$-weights of $\wh{\mathfrak{m}}/\mathfrak{m}$. Since $\mathfrak{h}\cap I_1$ has positive dimension and is contained in the common kernel of these weights, we find that $\delta\not\in \mathfrak{S}$. 
\end{proof}

\begin{remark}\label{twenty-four}
The reader may notice that it is possible for $\wh L_\delta$ to have a nontrivial connected center. However, the previous proposition shows such a subgroup must intersect $L_\delta$ with dimension $0$. Strictly speaking, this means $\pic^+_\Q(\op{Fl}_{L_\delta})$ can be identified with $\mathcal{C}(L_\delta \cap [\wh L_\delta, \wh L_\delta]\subset [\wh L_\delta,\wh L_\delta])$. For, setting $\wh L_\delta' =  [\wh L_\delta,\wh L_\delta]$ and $L_\delta' = L_\delta\cap \wh L_\delta'$, we have isomorphisms 
\begin{align*}
L_\delta'/B_\delta'\times \wh L_\delta'/\wh B_\delta' &\xrightarrow{\sim} L_\delta/B_\delta\times \wh L_\delta/\wh B_\delta\\
\pic^{L_\delta'}_\Q (L_\delta'/B_\delta'\times \wh L_\delta'/\wh B_\delta')&\xleftarrow{\sim} 
\pic^{L_\delta}_\Q(L /B_\delta\times \wh L_\delta/\wh B_\delta).
\end{align*}
Nevertheless, the quotient $\mathcal{C}(L_\delta\subset \wh L_\delta) \twoheadrightarrow \mathcal{C}(L_\delta'\subset \wh L_\delta')$ has kernel generated by pairs $(0,\chi)$ where $\chi:\wh L_\delta \to \C^*$, and it is usually preferable to describe the induction map applied to $\mathcal{L} \in \pic^+(\op{Fl}_{L_\delta})$ by its action on (any) lift $\mathcal{L}_\nu\boxtimes \mathcal{L}_{\wh \nu}\in \mathcal{C}(L_\delta\subset \wh L_\delta)$. 
\end{remark}

Therefore $\op{Ind}$ can be used to find extremal rays in theory; in practice it would helpful to have a formula, which we give here. 
To be precise, we complete the proof of Theorem \ref{yadhtrib} by showing the map $\op{Ind}$ of the previous section has the formula stated in the introduction. 

\begin{theorem}
Let $(\nu, \wh \nu) \in \mathfrak{h}_{L_\delta,\Q}^*\times \wh{\mathfrak{h}}_{\wh L_\delta,\Q}^* = (\mathfrak{h}_\Q/(\dot\delta))^*\times (\wh{\mathfrak{h}}_\Q/(\dot\delta))^*$. In other words, $\nu$ and $\wh \nu$ are characters on the original tori but vanish on $\delta$. 
We claim 
$$
\op{Ind}(\mathcal{L}_\nu\boxtimes \mathcal{L}_{\wh \nu}) = \mathcal{L}_\mu\boxtimes \mathcal{L}_{\wh \mu},
$$
where 
\begin{align}\label{formulaTWO}
(\mu, \wh \mu) = (w\nu, \wh w \wh \nu) - \sum_{v\xrightarrow{\alpha_\ell}w} w\nu(\alpha_\ell^\vee) \vec\mu(D(v)) - \sum_{v\xrightarrow{\wh\alpha_\ell}\wh w} \wh w\wh \nu(\wh \alpha_\ell^\vee) \vec\mu(D(v)).
\end{align}
\end{theorem}

\begin{proof}
Set $\mathcal{L} = \mathcal{L}_\mu\boxtimes \mathcal{L}_{\wh \mu}$, where $(\mu, \wh \mu)$ are defined by (\ref{formulaTWO}). Applying Lemmas \ref{ones} and \ref{zeros}, it is first of all clear that $(\mu, \wh \mu)$ satisfy the vanishing conditions of (\ref{fide}) required for membership in $\mathcal{F}_{2,\Q}$. Now letting $p:\op{Fl}_G\to \op{Fl}_G'$ denote the natural projection, consider the diagram
\begin{center}
\begin{tikzcd}
\wh{\mathcal{C}}\setminus \wh{\mathcal{R}} \arrow[d, "\pi"] & \\ 
\op{Fl}_G -\bigcup D(v) \arrow[d, "p"] & \op{Fl}_L \setminus \mathcal{R}_L, \arrow[l, "\tilde i"] \arrow[dl, "\tilde{i}'"] \arrow[ul, "i"'] \\  
\op{Fl}_G' & 
\end{tikzcd}
\end{center}
which commutes thanks to Lemma \ref{missedme}. Evidently $p^*\mathcal{L} = \mathcal{L}_{w\nu}\boxtimes\mathcal{L}_{\wh w\wh \nu}\big|_U$, where $U$ is $\op{Fl}_G- \bigcup D(v)$. Further pulling back via $\tilde i$ yields 
$$
\tilde{i}'^* \mathcal{L} = \mathcal{L}_\nu\boxtimes \mathcal{L}_{\wh \nu} \big|_{\op{Fl}_L\setminus\mathcal{R}_L};
$$
that is, the pullbacks of $\mathcal{L}$ and $\op{Ind}(\mathcal{L}_\nu\boxtimes\mathcal{L}_{\wh \nu})$ to $\op{Fl}_L$ agree on $\op{Fl}_L\setminus\mathcal{R}_L$. Applying Proposition \ref{opendoor}, 
$
(p\circ\pi)^*\mathcal{L}$ and $(p\circ\pi)^*\op{Ind}(\mathcal{L}_\nu\boxtimes\mathcal{L}_{\wh \nu})
$
agree away from $\wh{\mathcal{R}}$, so $p^*\mathcal{L}$ and $p^*\op{Ind}(\mathcal{L}_\nu\boxtimes\mathcal{L}_{\wh \nu})$ agree on $\op{Fl}_G-\bigcup D(v)$. 

Set $\mathcal{M} = (p^*\mathcal{L})^{-1}\otimes p^*\op{Ind}(\mathcal{L}_\nu\boxtimes\mathcal{L}_{\wh \nu})$, considered as a line bundle on $\op{Fl}_G$; then $\mathcal{M} = \mathcal{O}(D)$ for $D$ some sum of divisors $D(v)$. Since $\mathcal{O}(D)$ satisfies the vanishing conditions (\ref{fide}) (it is a tensor product of line bundles that do), $D$ must actually be trivial by Lemmas \ref{ones} and \ref{zeros}. Since $p^*$ is injective, this completes the proof. 
\end{proof}

\section{On the number of components of $\mathcal{R}_L$}\label{help}

Recall that for $(w,\wh w, \delta)$ with $\delta\in \mathfrak{S}$ satisfying 
$$
\phi_\delta^\odot\left([\wh X_{\wh w}]\right)\odot_0 [X_w] = [X_e],
$$
the associated face $\mathcal{F}(w,\wh w,\delta)$ has codimension 1. 

Let $\mathcal{R}_1,\hdots, \mathcal{R}_c$ be the irreducible components of $\mathcal{R}_L$ (really its inverse image in $L/B_L\times \wh L/\wh B_{\wh L}$). Since $L$ is connected, each $\mathcal{R}_i$ is fixed by $L$ and therefore induces a line bundle $\mathcal{O}(\mathcal{R}_i)$ on $\op{Fl}_L$. An important observation is that $\dim H^0(\op{Fl}_L\setminus \mathcal{R}_L,\mathcal{O}) = 1$, so therefore 
$$
\dim H^0(\op{Fl}_L,\mathcal{O}(N_1\mathcal{R}_1)\otimes \cdots \otimes \mathcal{O}(N_c\mathcal{R}_c))=1
$$
for any choices of $N_i\ge 0$. By  \cite[Lemma 62]{BKiers}, we have the following lemma. 

\begin{lemma}
The set $\{\mathcal{O}(\mathcal{R}_1),\hdots,\mathcal{O}(\mathcal{R}_c)\}$ gives a $\Z$-basis for the kernel of the restriction \\ $\pic^{\op{deg}=0}(\op{Fl}_{L})\to \pic^{\op{deg}=0}(\op{Fl}_L\setminus \mathcal{R}_L)$.
\end{lemma}
As before, let $q$ denote the number of type I extremal rays on $\mathcal{F}$. 

\begin{proposition}\label{sloppyeq}
$ $
$$
c = q - |\wh \Delta| + |\Delta(\wh P(\delta))|.
$$
\end{proposition}

\begin{proof}
Recall the isomorphism $\pic_\Q^{\deg=0}(\op{Fl}_L\setminus \mathcal{R}_L)\simeq \pic_\Q^{\deg=0}(\op{Fl}_G')$ and that $\dim \pic_\Q^{\op{deg}=0}(\op{Fl}_{G}') = \dim \mathcal{F}_2 = \dim\mathcal{F}-q$. 
Let $r = \dim X^*(T)$ and $\wh r = \dim X^*(\wh T)$ (here $X^*(M)$ denotes the character lattice for any algebraic group $M$). Counting $\Q$-dimensions, we have 
\begin{align*}
c&=\dim \pic^{\op{deg}=0}_\Q(\op{Fl}_{L}) - \dim \pic_\Q^{\op{deg}=0}(\op{Fl}_{L}\setminus\mathcal{R}_L)\\
&=r+\wh r-\dim X^*(\wh L) - 1 - (\dim \mathcal{F}-q)\\
&=r+\wh r - \dim X^*(\wh L) - 1 -(r+\wh r-1)+q\\
&=q-\dim X^*(\wh L).
\end{align*}
\end{proof}

\section{Inequalities for testing rays $(0,\wh\omega_j)$}\label{extraz}

Recall Observation \ref{obviate} from the introduction: 
\begin{proposition}
If $(\mu,\wh \mu)$ gives an extremal ray of $\mathcal{C}(G\hookrightarrow \wh G)$ and does not belong to any regular face, then $\mu=0$ and, up to scaling, $\wh \mu$ is a fundamental dominant weight.
\end{proposition}

\begin{proof}
If $(\mu,\wh \mu)$ is not on any regular face, then it is an extremal ray for the dominant cone $\mathfrak{h}^*_{\Q,+}\times \wh{\mathfrak{h}}^*_{\Q,+}$ itself. These are (up to scaling) all either of the form $(\omega_i,0)$, where $\omega_i$ is a fundamental weight for $G$ w.r.t. $B$, or $(0,\wh\omega_j)$, where $\wh\omega_j$ is the same for $\wh G$ w.r.t. $\wh B$. Of course, the first of these never occurs, since no non-trivial $G$ representation appears as a subrepresentation of the trivial representation for $\wh G$.
\end{proof}

Testing whether a candidate $(0,\wh\omega_j)$ is indeed a ray of the cone amounts to checking whether it belongs to the cone, which may be done by verifying the inequalities of Theorem \ref{christus}. In this section, we substantially whittle down the number of inequalities needed for this, depending on $j$. 

First define $\mathfrak{T}$ to be the set of all indivisible one-parameter subgroups of $T$ which give an extremal ray of a cone $\mathfrak{h}_{\Q,+}\cap \wh v\wh{\mathfrak{h}}_{\Q,+}$ for some $\wh v\in \wh W$.
Now fix an index $j\in \{1,\hdots, \op{rk}(\wh G)\}$. Define a set 
$$
S_j=\left\{(\wh w,\delta) \mid \delta\in \mathfrak{T}, \phi_\delta^\odot[\wh X_{\wh w}] = [X_e], \wh X_{\wh w}\subseteq \wh X_{s_i\wh w}\implies i=j  \right\}.
$$

\begin{theorem}\label{hopeful}
The ray generated by $(0,\wh \omega_j)$ is an extremal ray of $\mathcal{C}(G\hookrightarrow \wh G)$ if and only if for all $(\wh w,\delta)\in S_j$, the inequality 
$$
\wh \omega_j(\wh w\dot\delta)\le 0
$$
holds. 

Furthermore, if $\op{Wt}_\mathfrak{h}(\wh{\mathfrak{g}}/\mathfrak{g}) = \op{Wt}_\mathfrak{h}(\wh{\mathfrak{g}})$, 
then the smaller set of inequalities associated to $(\wh w , \delta)\in S_j$ with $\delta \in \mathfrak{S}$ will suffice. 
\end{theorem}

Before we come to the proof of the theorem, we recall a few definitions and results from geometric invariant theory which are applicable to our context. We use the notation and formulations of \cite[\S 3]{Kumar}. 

\begin{defi}
Given an algebraic group $S$ acting on a variety $\mathbb{X}$, an $S$-linearized line bundle $\mathbb{L}$ on $\mathbb{X}$, a point $x\in \mathbb{X}$ and a one-parameter subgroup $\delta: \C^* \to S$ such that $\lim_{t\to 0}\delta(t)x$ exists, Mumford 
defines an integer $\mu^\mathbb{L}(x,\delta)$ as follows. The $\C^*$-action on $\mathbb{X}$ induced by $\delta$ has $x_0 = \lim_{t\to 0}\delta(t)x$ as a fixed point, so the fibre of $\mathbb{L}$ above $x_0$ inherits a $\C^*$ action via some character. Characters of $\C^*$ are in bijection with the integers, and we take $\mu^\mathbb{L}(x,\delta)$ to be the integer associated with the character of the fibre action. 
\end{defi}

Let $\mathbb{L} = \mathcal{L}_\mu\boxtimes \mathcal{L}_{\wh \mu}$ be a line bundle over $G/B\times \wh G/\wh B$. Let $\delta$ be a dominant OPS. Let $fP(\delta)$ and $gB, \wh g\wh B$ satisfy
$$
\phi_\delta(fP(\delta))\in \phi_\delta(gBwP(\delta))\cap \wh g \wh B\wh w\wh P(\delta),
$$ 
for some $w\in W/W_P$ and $\wh w\in \wh W/\wh W_{\wh P}$. For $x = (gB, \wh G\wh B)\in G/B\times \wh G/\wh B$, we can calculate $\mu^\mathbb{L}(x,f\delta f^{-1})$ explicitly by \cite[Proposition 3.5, Lemma 3.6]{Kumar}: 
\begin{lemma}
$ $
$$
\mu^{\mathbb{L}}(x,f\delta f^{-1}) = -\mu(w \dot \delta)-\wh \mu(\wh w \dot \delta)
$$
\end{lemma}

Now, given an unstable point $x\in \mathbb{X}$, Kempf defines a \emph{maximally destabilizing} OPS, whose properties we recall here. Let $M(S)$ be the set of fractional one-parameter subgroups (see for example \cite[\S 6]{Kumar}) and $q$ an $S$-invariant norm $M(S)\to \mathbb{R}_{\ge 0}$. Set 
$$
q^*(x) = \inf_{\wh \delta\in M(S)}\{q(\wh \delta) | \mu^{\mathbb{L}}(x,\wh \delta)\le -1\},
$$
and 
$$
\Lambda(x) = \{\wh \delta \in M(S) | \mu^\mathbb{L}(x,\wh \delta)\le -1, q(\wh \delta) = q^*(x)\}.
$$
In \cite{Kempf}, Kempf proves that $\Lambda(x)$ is nonempty and that the associated parabolics $P(\wh \delta)$ for $\wh \delta\in \Lambda(x)$ are identical (they are thus referred to as $P(x)$); in fact $\Lambda(x)$ is a single $P(x)$-orbit under conjugation. 

\begin{proof}[Proof of Theorem \ref{hopeful}]

The direction $(\Rightarrow)$ is clear, since $\phi_\delta^\odot[\wh X_{\wh w}] = [X_e]\iff \phi_\delta^\odot[\wh X_{\wh w}]\odot_0[X_{w_0w_0^P}] = [X_e]$, where $w_0^P$ is the longest element of $W_P$. 

For $(\Leftarrow)$, assume $(0,\wh \omega_j)$ is not an extremal ray. Then $(0,\wh \omega_j)\not \in \mathcal{C}(G\hookrightarrow \wh G)$. Therefore $G/B\times \wh G/\wh B$ has no semistable points for the line bundle $\mathcal{L}_0\boxtimes \mathcal{L}_{\wh\omega_j}$. Pick any $(g,\wh g)\in G\times \wh G$ such that 
\begin{align*}
&\phi_{\delta}(gC_w^P)\cap \wh g\wh C_{\wh w}^{\wh P} \text{ and } \phi_{\delta}(gX_w^P)\cap \wh g\wh X_{\wh w}^{\wh P} \text{ are proper intersections in $\wh G/\wh P$ and }\\
&\phi_{\delta}(gC_w^P)\cap \wh g\wh C_{\wh w}^{\wh P} \text{ is dense inside } \phi_{\delta}(gX_w^P)\cap \wh g\wh X_{\wh w}^{\wh P}
\end{align*}
for any dominant $\delta$ and $(w,\wh w)\in W^P\times \wh W^{\wh P}$, where $P=P(\delta)$ and $\wh P=\wh P(\delta)$. 

Since $x=(\bar g,\overline{\wh g})\in G/B\times \wh G/\wh B$ is unstable, we may find a Kempf's OPS $\wh \delta = [\delta, a]\in \Lambda(x)$ associated to it. Let $\epsilon = f^{-1}\delta f$ be the dominant translate of $\delta$ whose image lives in $T$. Set $P = P(\epsilon), \wh P = \wh P(\epsilon)$. Find the unique $w\in W/W_P$ and $\wh w\in \wh W/\wh W_{\wh P}$ such that 
$$
\phi_{\epsilon}(fP) \in \phi_\epsilon(gBwP)\cap \wh g\wh B\wh w\wh P.
$$

\begin{lemma}\label{a-prox}
$ $
$$
\phi_\epsilon(gBwP)\cap \wh g\wh B\wh w\wh P = \{f\wh P\}.
$$
\end{lemma}

\begin{proof}
Suppose $\phi_\epsilon(hP)$ is also in the intersection. Then 
$$
 \mu^\mathbb{L}(x,h\epsilon h^{-1}) = \mu^\mathbb{L}(x,\delta) = -\mu(w\dot \epsilon)-\wh \mu(\wh w\dot\epsilon),
$$
so for $\lambda = [h\epsilon h^{-1},a]$, $\mu^\mathbb{L}(x,\lambda)\le -1$. Furthermore, $q(\lambda) = q(\wh \delta) = q^*(x)$ since $h\epsilon h^{-1}$ and $\delta$ are conjugate. So $\lambda\in \Lambda(x)$, which means $hP h^{-1} = P(h\epsilon h^{-1}) = P(\delta) = fP f^{-1}$, so $hP= fP$. 
\end{proof}

If $\epsilon$ (after rescaling) already belongs to $\mathfrak{T}$, set $\chi=\epsilon$. Otherwise, we must carefully exchange $\epsilon$ for an extremal OPS as follows. 

Recall from \cite[\S 2]{BeS} the notion of compatible elements of $\wh W$: $\wh v$ is compatible if $\dim\mathfrak{h}_{\Q,+}\cap \wh v\wh{\mathfrak{h}}_{\Q,+} = \dim \mathfrak{h}_{\Q,+}$. If $\wh v$ is compatible and $\chi_0$ is in the interior of $\mathfrak{h}_{\Q,+}\cap \wh v\wh{\mathfrak{h}}_{\Q,+}$, then (cf. \cite[Proposition 2.2.8]{BeS})
\begin{enumerate}[label=(\alph*)]
\item exchanging $\wh v$ for $\wh u\wh v$ where $\wh u\chi_0 = \chi_0$ yields 
$$
\mathfrak{h}_{\Q,+}\cap \wh v\wh{\mathfrak{h}}_{\Q,+} = \mathfrak{h}_{\Q,+}\cap \wh u \wh v\wh{\mathfrak{h}}_{\Q,+};
$$
\item if $\wh v$ is chosen to have minimal length in the right coset $\op{Stab}(\chi_0)\setminus \wh W$, then by Proposition \ref{moveme}(c)
$$
B\subseteq \wh v\wh B\wh v^{-1}.
$$
\end{enumerate}

From now on, we fix $\wh v$ which is compatible, satisfying $B\subseteq \wh v\wh B\wh v^{-1}$, such that $\epsilon \in \mathfrak{h}_{\Q,+}\cap \wh v\wh{\mathfrak{h}}_{\Q,+}$. 
Let $\chi$ be an OPS such that $\dot \chi$ is an extremal ray of the face of $\mathfrak{h}_{\Q,+}\cap \wh v\wh{\mathfrak{h}}_{\Q,+}$ containing $\epsilon$ in its interior.

\begin{lemma}
$ $
$P(\epsilon)\subseteq P(\chi)$ and $\wh P(\epsilon)\subseteq \wh P(\chi)$. 
\end{lemma}

\begin{proof}
It suffices to show $\wh P(\epsilon)\subseteq \wh P(\chi)$. Suppose $\wh \beta$ is a root for $\wh G$ such that $\wh \beta(\dot \epsilon)\ge 0$. 

If $\wh v^{-1}\wh \beta\succ 0$, then $\wh v^{-1}\wh \beta(\zeta)\ge 0$ for any $\zeta \in \wh{\mathfrak{h}}_{\Q,+}$; take $\zeta = \wh v^{-1} \dot\chi$ and we have $\wh \beta(\dot \chi)\ge 0$. 

Otherwise, $\wh v^{-1}\wh \beta\prec 0$, so $\wh \beta(\dot \epsilon)\le 0$; therefore $\wh \beta(\epsilon)=0$. 
Then $\wh v^{-1}(-\wh \beta)$ is a positive root for $\wh G$ and satisfies $\wh v^{-1}(-\wh \beta)(\wh v^{-1} \dot \epsilon) = 0$.
 Note that the faces of $\wh v \wh{\mathfrak{h}}_{\Q,+}$ are defined by the vanishing of roots $\wh \alpha$ such that $\wh v^{-1}\wh \alpha\succ 0$. The faces of $\mathfrak{h}_{\Q,+}$ are defined by the vanishing of simple roots $\alpha_i$. Since $B\subseteq \wh v\wh B\wh v^{-1}$, there exists (for each $i$) a root $\wh \eta_i$ such that $\wh v^{-1}\wh \eta_i\succ 0$ and $\wh \eta_i|_\mathfrak{h}\equiv \alpha_i$. Therefore we have shown 
\begin{align}\label{dorian}
&\text{the faces of $\mathfrak{h}_{\Q,+}\cap \wh v\wh{\mathfrak{h}}_{\Q,+}$ are defined}\\\nonumber&\text{by the vanishing of roots $\wh \alpha$ such that $\wh v^{-1}\wh\alpha\succ 0$}.
\end{align}
In particular, $\dot \epsilon$ belongs to the face defined by $-\wh \beta$. Since $\dot \chi$ is an extremal ray of any face on which $\dot \epsilon$ lies, $-\wh \beta(\dot \chi)=\wh \beta(\dot \chi)=0$. 
\end{proof}

Let $S_\epsilon$ be the set of indices $1\le i\le \op{rk}(G)$ such that $\alpha_i(\dot \epsilon)>0$. Thus $\dot\epsilon = \sum_{i\in S_\epsilon} c_i x_i$. Note that $\dot\chi = \sum_{S'} c_i'x_i$, where $S'\subseteq S_\epsilon$ (otherwise $\dot\epsilon$ would lie in a face of $\mathfrak{h}_{\Q,+}$ that didn't include $\dot\chi$). 

\begin{proposition}\label{27!}
$ $
\begin{enumerate}[label = (\alph*)]
\item 
The only point in $\phi_\chi(gBwP(\chi))\cap \wh g\wh B\wh w\wh P(\chi)$ is $\phi_\chi(fP(\chi))$. 

\item The inequality 
$$
\wh\omega_j(\wh w\dot \chi)>0
$$
is satisfied. Moreover, for any conjugate $\tilde g\chi \tilde g^{-1}$ of $\chi$, 
$$
\mu^{\mathbb{L}}(x,\tilde g\chi \tilde g^{-1}) = \mu^{\mathbb{L}}(x,f\chi f^{-1}) \implies fP(\chi) = \tilde gP(\chi).
$$
\end{enumerate}
\end{proposition}

\begin{proof}
Note that (a) follows from (b) with the same proof as Lemma \ref{a-prox}. So we prove (b), closely mimicking the proof of \cite[Lemma 27]{BK}. 

First, 
we can find a $b\in G$ and some $w\in W$ so that 
$bP(\epsilon) = fP(\epsilon)$ and 
$$
b^{-1}\tilde g\chi \tilde g^{-1}b = w\chi.
$$
We hope to show that $w=e$, so that $\tilde gP(\chi) = bP(\chi) = fP(\chi)$. 

Now, the function $\mathfrak{L}: \mathfrak{h}_{\Q,+}\to \Q$ given by 
$$
r\dot \beta \mapsto -r\mu^{\mathbb{L}}(x,b\beta b^{-1}),
$$
where $r \in \Q$ and $\beta$ is an OPS of $T$, is well-defined. It also satisfies the following (cf. \cite[Lemma 27]{BK}):
\begin{enumerate}[label=(\roman*)]
\item $\mathfrak{L}(h) = \wh \omega_j(\wh w h)$ for $h\in \oplus_{{S_\epsilon}} \Q_{\ge0} x_i$
\item the function $J(h) = \mathfrak{L}(h)/q(h)$ on $\mathfrak{h}_\Q\setminus\{0\}$ is constant on $\Q_{\ge0}$-rays and achieves its maximum uniquely at the ray through $Y:=\dot \epsilon/a$. 
\end{enumerate}

Furthermore, $J$ satisfies 
$$
J(h)\le J(Y)\frac{(Y,h)}{q(Y)q(h)}
$$
for $h$ nonzero and 
$$
J(h)= J(Y)\frac{(Y,h)}{q(Y)q(h)}
$$
if furthermore we assume $h\in \oplus_{{S_\epsilon}} \Q_{\ge0} x_i$; here $(,)$ denotes the Killing form. 

First of all, this already shows that $J(\dot \chi) >0$ since $J(\dot\epsilon)>0$ and $(Y,\dot \chi)>0$ due to the pairings $(x_i,x_j)\ge 0$ in general. This shows $\mu^{\mathbb{L}}(x,b\chi b^{-1})<0$ and $\wh \omega_j(\wh w\dot \chi)>0$. 

Now assume (for the sake of contradiction) that $w\chi \ne \chi$. By induction on length of $w$, one can easily show that $(\dot \epsilon, wx_i)<(\dot \epsilon,x_i)$ if $wx_i\ne x_i$ and $i\in S_\epsilon$. Therefore $(Y,w\dot \chi)<(Y,\dot \chi)$. 

Putting this all together, we have 
$$
J(w\dot \chi)\le J(Y)\frac{(Y,w\dot \chi)}{q(Y)q(w\dot \chi)}=J(Y)\frac{(Y,w\dot \chi)}{q(Y)q(\dot \chi)}<J(Y)\frac{(Y,\dot \chi)}{q(Y)q(\dot \chi)} = J(\dot\chi),
$$
contradicting the hypothesis that $J(\dot\chi) = J(w\dot\chi)$. 
\end{proof}

By genericity of $g,\wh g$, we already know 
\begin{align}\label{eqeq}
\phi_\chi^*[\wh X_{\wh w}]\cdot [X_w] = [X_e] 
\end{align}
in the ring $H^*(G/P(\chi))$. We claim that this product doesn't vanish in the passage to the deformed product. 

\begin{proposition}
The pair $(w,\wh w)$ is Levi-movable. 
\end{proposition}

\begin{proof}
First write $g = fpw^{-1}b$ and $\wh g = f\wh p\wh w^{-1}\wh b$ for suitable $p\in P(\epsilon), \wh p\in \wh P(\epsilon), b\in B, \wh b\in \wh B$. Then 
\begin{align*}
\delta(s)gB = f\epsilon(s) p \epsilon(s)^{-1} w^{-1}B \text{   and    } \delta(s)\wh g\wh B= f \epsilon(s)\wh p\epsilon(s)^{-1}\wh w^{-1}\wh B,
\end{align*}
so in the limit, 
$$
\lim_{s\to 0} \delta(s)(gB, \wh g\wh B) = (flw^{-1}B, f\wh l\wh w^{-1}\wh B),
$$
where $l = \lim_{s\to 0}\epsilon(s) p \epsilon(s)^{-1} \in L(\epsilon)$ and $\wh l = \lim_{s\to 0}\epsilon(s) \wh p \epsilon(s)^{-1} \in L(\epsilon)\in \wh L(\epsilon)$. 

By a result of Ramanan and Ramanathan \cite[Proposition 1.9]{RR}, the limit point 
$$
x_0 = \lim_{s\to 0} \delta(s)(gB, \wh g\wh B)
$$
is unstable and $[\delta, a]\in \Lambda(x_0)$. Obviously $\phi_\epsilon(fP(\epsilon))$ belongs to 
$\phi_\epsilon(flw^{-1}BwP(\epsilon))\cap f\wh l\wh w^{-1}\wh B\wh w\wh P(\epsilon)$, so by Proposition \ref{27!}, this time applied with the unstable point $x_0$ in mind, 
$$
\phi_\chi(flw^{-1}BwP(\chi))\cap f\wh l\wh w^{-1}\wh B\wh w\wh P(\chi)=\{\phi_\chi(fP(\chi))\}.
$$

Now the expected and actual dimensions of this intersection agree; furthermore the multiplicity at $fP(\chi)$ would only increase if it were not transverse, but we already know (\ref{eqeq}) holds. So the intersection
$$
\phi_\chi(lw^{-1}BwP(\chi))\cap \wh l\wh w^{-1}\wh B\wh w\wh P(\chi)=\{\phi_\chi(eP(\chi))\}.
$$
is transverse at $eP(\chi)$, and the pair $w,\wh w$ is Levi-movable. 
\end{proof}

\begin{lemma}
If $\op{Wt}_{\mathfrak{h}}(\wh{\mathfrak{g}}/\mathfrak{g}) = \op{Wt}_{\mathfrak{h}}(\wh{\mathfrak{g}})$, 
$\displaystyle \dim \bigcap_{\beta\in \op{Wt}_\mathfrak{h}(\wh{\mathfrak{l}}(\chi)/\mathfrak{l}(\chi))} \ker \beta=1$.
\end{lemma}

\begin{proof}
Since, by (\ref{dorian}), $\mathfrak{h}_{\Q,+}\cap \wh v\wh{\mathfrak{h}}_{\Q,+}$ is the cone inside $\mathfrak{h}_{\Q}$ dual to the cone $C \subseteq  \mathfrak{h}_\Q^*$ generated by  $S_0=\{\wh \alpha|_\mathfrak{h}\mid \wh v^{-1}\wh \alpha\succ 0\}$, the extremal ray $\Q_{\ge 0}\dot \chi$ is orthogonal to a hyperplane spanned by a proper subset of $S_0$.  In other words, 
$$
\C\dot \chi = \bigcap_{\beta\in \op{Wt}_\mathfrak{h}(\wh{\mathfrak{l}}(\chi))} \ker \beta.
$$

By hypothesis, $\op{Wt}_\mathfrak{h}(\wh{\mathfrak{l}}(\chi)/\mathfrak{l}(\chi))=\op{Wt}_\mathfrak{h}(\wh{\mathfrak{l}}(\chi))$, and the result follows. 
\end{proof}

To summarize so far, we have found a dominant one-parameter subgroup $\chi: \C^* \to T$ (which we may now assume is indivisible) and Weyl group elements $w\in W$, $\wh w\in \wh W$ such that 
\begin{enumerate}[label=(\alph*)]
\item $\chi$ belongs to $\mathfrak{T}$ (in the case $\op{Wt}_{\mathfrak{h}}(\wh{\mathfrak{g}}/\mathfrak{g}) = \op{Wt}_{\mathfrak{h}}(\wh{\mathfrak{g}})$, belongs to $\mathfrak{S}$); 
\item $\phi_\chi^\odot[\wh X_{\wh w}]\odot_0 [X_w]=[X_e]$;
\item $\wh \omega_j(\wh w \dot \chi)>0$;
\item if $\mu^{\mathbb{L}}(x,h \chi h^{-1}) = \mu^{\mathbb{L}}(x,f\chi f^{-1})$, then $hP(\chi) = fP(\chi)$.
\end{enumerate}

For simplicity, now take $P = P(\chi), \wh P = \wh P(\chi)$. 

Assume for the sake of contradiction that 
 $\wh X_{\wh w}\subsetneq \wh X_{s_j\wh w}$ or that $X_{w}\subsetneq X_{v}$ for $v\ne w$. Set $\wh v=s_j\wh w$ and $v=w$  in the first case or $\wh v = \wh w$ in the second. Then since 
$\phi_\chi(gBvP)\cap \wh g\wh B\wh v\wh P$ is dense inside $\phi_\chi(gX_v)\cap \wh g\wh X_{\wh v}$ and the complement is nonempty, there must be some point $hP$ in 
$$
\phi_\chi(gBvP)\cap \wh g\wh B\wh v\wh P.
$$

Then 
$$
\mu^{\mathbb{L}}(x,h\chi h^{-1}) = -\wh\omega_j(\wh v\chi)  = -\wh \omega_j(\wh w\chi) = \mu^{\mathbb{L}}(x,f\chi f^{-1}).
$$
Therefore  $hP = fP$, a contradiction since these live in different Schubert cells of either $G/P$ or $\wh G/\wh P$. 

So we conclude that $w = w_0w_0^P$ and $(\wh w, \chi)\in S_j$, and the failed inequality (c) witnesses the fact that $(0,\wh \omega_j)$ is not in $\mathcal{C}(G\hookrightarrow \wh G)$. 
\end{proof}

\begin{corollary}\label{yippee}
If there are no pairs $( \wh w,\delta)$ such that $\phi_\delta^\odot[\wh X_{\wh w}] = [X_e]$ and $\delta\in \mathfrak{T}$, then every ray of the form $(0,\wh\omega_j)$ is extremal. 
\end{corollary}

\section{Examples}\label{ampleXample}

We begin with a general review of computing pullbacks in (equivariant) cohomology, recalling without proof several standard results (see \cite{BGG}, \cite{Brion}, \cite{Graham}). Let $S=\op{Sym}^*(\mathfrak{h}^*)$. Under the Borel model, there is an isomorphism 
$$
S/J \to H^*(G/B),
$$ 
where $J$ is the ideal generated by the elements of $S^W$ vanishing at $0$. The map is induced by (and uniquely determined by) the Chern class map
\begin{align*}
\mathfrak{h}^*_{\Z} &\to H^2(G/B)\\
\lambda&\mapsto c_1(L_\lambda).
\end{align*}

Furthermore, the cohomology subrings $H^*(G/P)\subseteq H^*(G/B)$, where $P$ is a standard parabolic, are identified with the invariant subrings 
$$
\left[S/J\right]^{W_P} \subseteq S/J.
$$
In similar fashion, we set $\wh S = \op{Sym}^*(\wh{\mathfrak{h}}^*)$ and let $\wh J$ be the corresponding invariant ideal for $\wh W$. 
\begin{proposition} 
The diagram 
\begin{center}
\begin{tikzcd}
\wh S/\wh J \arrow[r] \arrow[d]& H^*(\wh G/\wh B) \arrow[d]\\
S/J \arrow[r] & H^*(G/B) 
\end{tikzcd}
\end{center}
commutes, where the horizontal arrows are the Borel isomorphisms and the vertical arrows the natural pullbacks. More generally, for standard parabolics $P\subseteq \wh P$, 
\begin{center}
\begin{tikzcd}
\left[\wh S/\wh J\right]^{\wh W_{\wh P}} \arrow[r] \arrow[d]& H^*(\wh G/\wh P) \arrow[d]\\
\left[S/J\right]^{W_P} \arrow[r] & H^*(G/P) 
\end{tikzcd}
\end{center}
commutes. 
\end{proposition}

\begin{proof}
By functoriality of the Chern class, 
\begin{center}
\begin{tikzcd}
\op{Sym}^*(\wh{\mathfrak{h}}^*)\arrow[r] \arrow[d]& H^*(\wh G/\wh B) \arrow[d]\\
\op{Sym}^*({\mathfrak{h}}^*) \arrow[r] & H^*(G/B) 
\end{tikzcd}
\end{center}
commutes. Furthermore, the kernel of the top horizontal map is sent to the kernel of the bottom horizontal map. The second diagram follows from restricting the first to the appropriate subrings. 
\end{proof}

Finally, recall from \cite{BGG}: 

\begin{proposition}\label{lin}
For any simple reflection $s_i$, 
$[X_{s_iw_0}] = -w_0 \omega_i$.
\end{proposition} 

\begin{proof}
We have $[X_{s_iw_0}] = [X_{w_0w_0s_iw_0}] = [X_{w_0s_j}]$, where $\alpha_j$ is the \emph{simple} root $-w_0\alpha_i$. Under the Borel isomorphism, $[X_{w_0s_j}]$ is identified with the BGG polynomial $P_{s_j}$, which is degree $1$ and satisfies 
$$
A_{i}P_{s_j} = \left\{\begin{array}{cc}
0 & i\ne j\\
1 & i=j
\end{array}
\right.,
$$
where $A_i$ are the divided difference operators. The only linear functionals $f\in \mathfrak{h}^*$ invariant under all $s_i$, $i\ne j$ are the multiples $f = c\omega_j$. From $(c\omega_j-(c\omega_j-c\alpha_j))/\alpha_j=1$ we learn that $c=1$, so $P_{s_j} = \omega_j = -w_0 \omega_i$.
\end{proof}

Recall the identification $H^*_T(G/B) = H^*_G(G/B\times G/B) = S\otimes_{S^W} S$. This identification once again stemming from the Chern classes of line bundles associated to characters, we have another commutative diagram 
\begin{center}
\begin{tikzcd}
\wh S \otimes_{\wh S^{\wh W}} \wh S \arrow[r] \arrow[d] & H^*_{\wh T}(\wh G/\wh B) \arrow[d] \\
S\otimes_{S^W} S \arrow[r] & H^*_T(G/B)
\end{tikzcd}
\end{center}
with horizontal maps isomorphisms and vertical maps the natural restrictions. 

For the sake of concrete calculations such as in Theorem \ref{formulaONE}(b), one wants suitable polynomial representatives for $[X_u]^T$ (and $[\wh X_{\wh u}]^{\wh T}$). W. Graham gives a procedure in \cite[Proposition 4.2]{Graham} that makes this possible, as we now describe. Ultimately, we will in fact use certain approximations (cf. Proposition \ref{approxtheory}) of $[X_u]^T$ and $[\wh X_{\wh u}]^{\wh T}$ that are inspired by the formulas for related classes in D. Anderson's note \cite{And:2007}. 

%The procedure given by \cite[Proposition 4.2]{Graham:1997} can be used to find polynomial representatives for $[X_u]^T$ (and $[\wh X_{\wh u}]^{\wh T}$), which allows for concrete calculations as in Theorem \ref{formulaONE}(b). Graham's method can take any pair of bases of $S$ over $S^W$ as its starting point, but we will make a specific choice which has nice properties. 

Graham's method can take any pair of bases of $S$ over $S^W$ as its starting point, but we will make a specific choice which has nice properties. 
Let $P_v,v\in W$, be homogeneous lifts of the BGG polynomials in $S$; that is, the image of $P_v$ in $S/J$ is identified with $[X_{w_0v}]\in H^*(G/B)$. One way to construct them is by setting $P_{w_0} = \frac{1}{|W|}\prod_{\Phi^+}\alpha$ and $P_w = A_{w^{-1}w_0}P_{w_0}$. They satisfy $\deg P_v = \ell(v)$ as well as the following properties.

\begin{lemma}\label{triangular}
$$
A_{w_0}\left(P_vP_w\right) = \left\{
\begin{array}{cc}
1, & v = w_0w\\
0, & \ell(v) + \ell(w) = \ell(w_0), v\ne ww_0\\
0, & \ell(v)+\ell(w)<\ell(w_0) \\
0, & \ell(v)+\ell(w) = \ell(w_0)+1.
\end{array}
\right.
$$
In case $\ell(v)+\ell(w)>\ell(w_0)+1$, $A_{w_0}\left(P_vP_w\right)$ belongs to $S^W$ and (if nonzero) is homogeneous of degree $\ell(v)+\ell(w)-\ell(w_0)$. 
\end{lemma}

\begin{proof}
The first and second statements are from \cite{BGG}. The third follows from degree considerations. The $W$-invariance of $A_{w_0}f$ for any $f\in S$ follows from the observation that for any $i$, $A_{s_i}A_{w_0}$ is the $0$-operator; this proves the fifth statement, and the fourth follows from noting that there are no $W$-invariant linear polynomials in the subring of $S$ generated by the roots (all $P_v$ live in this subring, and the $A_v$ act on this subring). 
\end{proof}

Following \cite{Graham}, let $(b_{u,v})$ be the symmetric, $S^W$-valued matrix with $b_{u,v} = A_{w_0}(P_vP_w)$, and let $(a_{u,v})$ be its inverse matrix. Then a representative polynomial $f\in S\otimes_\C S$ for the class of the diagonal $[G(e,e)]^G = [X_e]^T\in H^*_T(G/B)$ is given by 
$$
f = \sum a_{uv} P_u\otimes P_v.
$$
We claim that $f$ has the following special form: 
\begin{lemma}
$$f = \sum_{u\in W} P_u\otimes P_{w_0u} + \sum_{\ell(u)+\ell(v)<\ell(w_0)-1} a_{u,v}P_u\otimes P_v$$
\end{lemma}

\begin{proof}
Lemma \ref{triangular} makes it clear that, if $W$ is ordered first by increasing length of its elements, then in such a way that the distances from $e$ to $v$ and from $w_0v$ to $w_0$ are equal for all $v$, the matrix $(b_{u,v})$ is of the form 
$$
\left[\begin{array}{cccc}
0 &  \cdots & 0  & 1 \\
0 &        \cdots & 1      & * \\
\vdots & \iddots & & \vdots \\
1 & \cdots & * & *
\end{array}\right],
$$
i.e., is lower-triangular with respect to the antidiagonal. The inverse of such a matrix is upper-triangular with respect to the antidiagonal, from which we immediately deduce 
$$
a_{u,v} = \left\{
\begin{array}{cc}
1, & v = w_0u \\
0, & \ell(v)+\ell(u) = \ell(w_0), v\ne u\\
0, & \ell(v)+\ell(u) >\ell(w_0).
\end{array}
\right.
$$
It remains to show that $a_{u,v} = 0$ when $\ell(v)+\ell(u) = \ell(w_0)-1$. For this, examine the sum 
$$
\sum_{q\in W} b_{w_0u,q}a_{q,v},
$$
which must equal $0$ as $v\ne w_0u$. We may partition this sum according to the value of $n_q:=\ell(w_0u)+\ell(q)$ as 
\begin{align*}
\sum_{n_q<\ell(w_0)} b_{w_0u,q}a_{q,v} + \sum_{n_q=\ell(w_0)} b_{w_0u,q}&a_{q,v} + \sum_{n_q=\ell(w_0)+1} b_{w_0u,q}a_{q,v}\\ &+\sum_{n_q>\ell(w_0)+1} b_{w_0u,q}a_{q,v}. 
\end{align*}
Lemma \ref{triangular} shows that the first term equals $0$, the second equals $a_{u,v}$, and the third equals $0$ as well. The fourth term equals $0$ by upper-triangularity of $(a_{q,v})$, observing that 
$$
\ell(w_0u)+\ell(q)>\ell(w_0)+1 \iff \ell(v) +\ell(q)>\ell(w_0).
$$
\end{proof}

\begin{corollary}
A representative for $[X_u]^T\in H^*_T(G/B)$ is given by 
$$
[X_u]^T = A_u[X_e]^T = \sum_{v^{-1}w = u} P_v\otimes P_{w_0w} + \sum_{x,y} a_{xu,y} P_{x}\otimes P_y,
$$
the first sum over all $v,w$ satisfying $\ell(v)+\ell(u) = \ell(w)$, and the second sum over all $x,y$ satisfying $\ell(x)+\ell(u) = \ell(xu)$ and $\ell(x)+\ell(y)+\ell(u)<\ell(w_0)-1$. $($Here the divided difference operator acts on the left factor of the tensor product, as in \cite{Brion}.$)$
\end{corollary}

In practice, the calculation of Theorem \ref{formulaONE}(b) may be carried out with only the ``first-order'' approximations of $[X_u]^T, [\wh X_{\wh u}]^{\wh T}$. 

\begin{proposition}\label{approxtheory}
Define approximations 
$$
f_u = \sum_{v^{-1}w = u} P_v\otimes P_{w_0w}\in S\otimes S
$$
and 
$$
\wh f_{\wh u} = \sum_{\wh v^{-1}\wh w = \wh u} P_{\wh v}\otimes P_{\wh w_0\wh w}\in \wh S\otimes \wh S.
$$
Then identifying $S\simeq H^*_T(pt)$, we have 
$$
\int_{G/B} [X_u]^T\cdot \phi_\delta^*[\wh X_{\wh u}]^{\wh T} = \int f_u\cdot \wh f_{\wh u}|_{T},
$$
where for any pure tensor $g\otimes h\in S\otimes S$, we define 
$$
\int g\otimes h = \frac{g}{\prod_{\Phi^+}\alpha} \sum_{w\in W} (-1)^{\ell(w)}wh
$$
and extend by linearity to an operator $S\otimes S\to S$. 
\end{proposition}

\begin{proof}
Observe that if $h$ is of degree $<\ell(w_0)$, $\int g\otimes h = 0$. In the context of Theorem \ref{formulaONE} we have $\ell(w_0)-\ell(u)+\ell(\wh w_0)-\ell(\wh u) = \ell(w_0)+1$; that is, $\ell(u)+\ell(\wh u) +1 = \ell(\wh w_0)$. Consider a term missing from the approximation in the product $[X_u]^T\cdot \phi_\delta^*[\wh X_{\wh u}]^{\wh T}$, for example 
$$
\left(a_{xu,y}P_x\otimes P_y\right)\cdot \left( P_{\wh v}\otimes P_{\wh w_0\wh w}\right)|_T
$$
where $\ell(x)+\ell(u) = \ell(xu)$, $\ell(x)+\ell(y)+\ell(u)<\ell(w_0)-1$, and $\ell(\wh v)+\ell(\wh u) = \ell(\wh w)$. The degree of $P_y\cdot P_{\wh w_0\wh w}|_T$ is bounded as follows:
\begin{align*}
\ell(y)+\ell(\wh w_0)-\ell(\wh w) &<\ell(w_0)-1-\ell(x)-\ell(u)+\ell(\wh w_0)-\ell(\wh v)-\ell(\wh u)\\
&\le \ell(w_0)-1-\ell(u)+\ell(\wh w_0)-\ell(\wh u)\\
& = \ell(w_0),
\end{align*}
so this term integrates to $0$. The other types of cross-terms similarly integrate to $0$. 
\end{proof}

\subsection{A root embedding of $SL_2\to SL_3$}\label{illustrious}

Define $\iota: SL_2\to SL_3$ by $\iota: A\mapsto \left[\begin{array}{c|c} A & 0 \\\hline 0 & 1\end{array}\right]$ at the level of matrices; this is the root embedding along the simple root $\alpha_1$ for $SL_3$. For notation, let $\{\alpha_1,\alpha_2,\alpha_1+\alpha_2\}$ be the positive roots for $SL_3$ \wrt the standard Borel $\wh B$ of upper-triangular matrices. Let $\alpha$ denote the positive root for $SL_2$ \wrt the Borel $B$ of upper-triangular matrices. The only indecomposable dominant one-parameter subgroup is $\alpha^\vee: t\mapsto \left[\begin{array}{cc} t & 0 \\ 0 & t^{-1} \end{array}\right]$. It is also admissible, being orthogonal to the trivial hyperplane in $\mathfrak{h}^*$. 

\subsubsection{Change of basis} We notice that $\alpha^\vee$ is not dominant \wrt $\wh B$, so we change basis as described in Section \ref{changingtime}. Our new Borel $\wh B'$ of $SL_3$ has simple roots $\gamma_1:=\alpha_1+\alpha_2$ and $\gamma_2:=-\alpha_2$, and we have $\wh B' = s_{\alpha_2}^{-1}\wh Bs_{\alpha_2}$. Observe that $\gamma_1+\gamma_2 = \alpha_1$ is still positive; i.e., $B\subseteq \wh B'$. 

We have $P(\alpha^\vee) = B$ and $\wh P(\alpha^\vee) = \wh B'$. The usual pullback (via $\phi: SL_2/B\to SL_3/\wh B'$) in cohomology sends 
$$
\begin{array}{ccc}
\left[\wh X_e\right]\mapsto 0, & & \left[\wh X_{s_{\gamma_1}s_{\gamma_2}}\right]\mapsto [X_e],\\\\
\left[\wh{X}_{s_{\gamma_1}}\right]\mapsto 0, & & \left[\wh X_{s_{\gamma_2}s_{\gamma_1}}\right]\mapsto [X_e],\\\\
\left[\wh X_{s_{\gamma_2}}\right]\mapsto 0, & & \left[\wh X_{w_0}\right]\mapsto [X_s],
\end{array}
$$
where $w_0 = s_{\gamma_1}s_{\gamma_2}s_{\gamma_1}$ and $s = s_\alpha$ are the longest elements in the two Weyl groups. 

Therefore 
$$
\begin{array}{ccc}
\phi^* \left[\wh X_{s_{\gamma_1}s_{\gamma_2}}\right]\cdot [X_s] = [X_e]; & & \phi^* \left[\wh X_{s_{\gamma_2}s_{\gamma_1}}\right]\cdot [X_s] = [X_e]; \\ \phi^*\left[\wh X_{w_0}\right] \cdot [X_e] = [X_e].
\end{array}
$$
Checking the numerical criterion for $L$-movability, we see that 
\begin{align*}
\langle \rho + s^{-1}\rho, \alpha^\vee\rangle -\langle 2\rho, \alpha^\vee\rangle + \langle \wh{\rho} + (s_{\gamma_1}s_{\gamma_2})^{-1}\wh{\rho},\alpha^\vee\rangle &= 0-2+1 = -1\\
\langle \rho + s^{-1}\rho, \alpha^\vee\rangle -\langle 2\rho, \alpha^\vee\rangle + \langle \wh{\rho} + (s_{\gamma_2}s_{\gamma_1})^{-1}\wh{\rho},\alpha^\vee\rangle &= 0-2+1 = -1\\
\langle \rho+\rho,\alpha^\vee\rangle - \langle 2\rho, \alpha^\vee\rangle + \langle \wh{\rho}+w_0^{-1}\wh\rho,\alpha^\vee\rangle &= 2-2+0 = 0,
\end{align*}
so in the deformed cohomology, we have 
$$
\begin{array}{ccc}
\phi^\odot \left[\wh X_{s_{\gamma_1}s_{\gamma_2}}\right]\odot_0 [X_s] = 0; & & \phi^\odot \left[\wh X_{s_{\gamma_2}s_{\gamma_1}}\right]\odot_0 [X_s] = 0; \\ \phi^\odot\left[\wh X_{w_0}\right] \odot_0 [X_e] = [X_e].
\end{array}
$$
Therefore if $\mu = a\omega$ and $\wh\mu = b\omega_1+c\omega_2$ are arbitrary dominant weights, the sole inequality that $(\mu,\wh\mu)$ must satisfy for membership in $\mathcal{C}(SL_2\xrightarrow{\iota} SL_3)$ is
$$
a\le b+c,
$$
and the sole regular facet $\mathcal{F}$ is the locus $a = b+c$, with face data $(e,w_0)$. 

\subsubsection{The rays}

Notice that $\alpha^\vee$ is the only element in $\mathfrak{T}$ and the hypothesis of Corollary \ref{yippee} is satisfied; therefore we have the two rays (not on $\mathcal{F}$): $(0,\omega_1)$ and $(0,\omega_2)$ (the trivial $SL_2$ representation appears in each of the fundamental representations for $SL_3$). 

On $\mathcal{F}$, we have two type I data, corresponding to 
$
s_{\gamma_1}s_{\gamma_2}\xrightarrow{\gamma_2} w_0
$
and 
$
s_{\gamma_2}s_{\gamma_1}\xrightarrow{\gamma_1} w_0.
$
From $(u,\wh u) = (e,s_{\gamma_1}s_{\gamma_2})$, we calculate 
$$
[X_{s_\alpha u}]\cdot \phi^*\left[\wh X_{\wh u}\right] = [X_e]
$$
and 
$$
[X_{u}]\cdot \phi^*\left[\wh X_{s_{\gamma_2}\wh u}\right] = [X_e],
$$
meanwhile $s_{\gamma_1} \wh u$ is of shorter length than $\wh u$. Therefore the first type I ray has coordinates 
$$
(1,0,1)
$$
in the $\{\omega,\omega_1,\omega_2\}$ basis. 

By a similar calculation (or by symmetry), the type I ray from the datum $s_{\gamma_2}s_{\gamma_1}\xrightarrow{\gamma_1} w_0$ is 
$$
(1,1,0)
$$
in the same coordinates. 

There are no type II rays. Here $L_\delta = \{e\}$ and $\wh L_\delta\simeq \C^*$, so $\wh L_\delta' = \{e\}$ as in Remark \ref{twenty-four}. One can still check that the single ray in $\mathcal{C}(L_\delta\subset \wh L_\delta)$ maps to $0$ under $\op{Ind}$. 
Note that $(1,0,1)$ and $(1,1,0)$ generate $\mathcal{F}$. Note also that their $1$s and $0$s illustrate Lemmas \ref{ones} and \ref{zeros}.
These two, together with $(0,1,0)$ and $(0,0,1)$, indeed generate $\mathcal{C}(SL_2\xrightarrow{\iota} SL_3)$. 

\subsubsection{Illustration of Proposition \ref{sloppyeq}} On our face $\mathcal{F}$, we had $q=2$ type I rays. The kernel of the induction map has rank $c=0$. Note that $|\Delta(\wh P)| = 0$ and $|\wh \Delta| = 2$, so 
$$
0= c = q - |\wh \Delta| + |\Delta(\wh P)| = 2-2+0
$$
is satisfied. 

\subsection{Principal embeddings $SL_2\to \wh G$ for $\wh G$ simple}

Suppose $SL_2\to \wh G$ is an embedding such that $B\subseteq \wh B$ where $B$ is the standard Borel of $SL_2$ and $\wh B$ a Borel subgroup of $\wh G$. We assume that $\wh G$ is not itself $SL_2$, in which case the question is uninteresting. Assume, by conjugating $\wh B$ if necessary, that the coroot $\alpha^\vee$ for $SL_2$ is a dominant coweight of $\wh G$ \wrt $\wh B$. By a result of Dynkin \cite{DYNKIN}, we may write 
$$
\alpha^\vee = \sum_{i=1}^r d_ix_i
$$
in the Lie algebra $\wh{\mathfrak{h}}$, where the $x_i$ are dual to the simple roots $\alpha_i$ given by $\wh B$ and each $d_i$ is $0,1,$ or $2$ ($r$ is the rank of $\wh G$).

Let $\mu = m \omega$ be a dominant weight for $SL_2$ and $\wh \mu$ a dominant weight for $\wh G$. Then $(\mu, \wh \mu)$ belongs to $\mathcal{C}(SL_2\to \wh G)$ if and only if 
\begin{align}\label{probred}
-\wh\mu(\alpha^\vee) + \max_{\alpha_i(\alpha^\vee)\ne 0}d_i\wh\mu(\alpha_i^\vee) \le m\le \wh\mu(\alpha^\vee);
\end{align}
see \cite[\S 5.3]{BeS}. 

\subsubsection{Minimal inequalities in the principal case with $\wh G$ simple}
In the case each $d_i = 2$, we call the embedding ``principal'' (notably, such embeddings exist and are unique up to conjugation for any $\wh G$). Then the inequalities (\ref{probred}) become 
\begin{align}\label{defred}
-\wh\mu(\alpha^\vee) + 2\max_{1\le i\le r}\wh\mu(\alpha_i^\vee) \le m\le \wh\mu(\alpha^\vee);
\end{align}

If $\wh G$ is simple, this is not the smallest possible set of inequalities. Rather, 
\begin{proposition}
The inequalities (\ref{defred}) are satisfied if and only if $m\le \wh \mu(\alpha^\vee)$. 
\end{proposition}

\begin{proof}
It suffices to show that $-\wh\mu(\alpha^\vee)+2 \wh \mu(\alpha_i^\vee)\le 0$ for any $1\le i\le r$. Write 
$$
\alpha^\vee = \sum_{i=1}^r c_i\alpha_i^\vee
$$
for suitable $c_i$. Then the coefficients $c_i$ and $d_i$ are related by 
$$
\vec d = M \vec c,
$$
where $M$ is the Cartan matrix for $\wh G$ (or its transpose, depending on convention). Therefore $\vec c = M^{-1}\vec d$, and since each $d_j = 2$, each $c_i$ is twice the sum of the elements in a row of $M^{-1}$. The sums across rows of $M^{-1}$ are always at least $1$ if $M$ is not the Cartan matrix for $SL_2$ (see \cite[Table 2]{DYNKIN}). Therefore $c_i\ge 2$ for all $i$ and 
$$
-\wh \mu(\alpha^\vee)+2\wh \mu(\alpha_i)^\vee \le -c_i\wh \mu(\alpha_i^\vee)+2\wh\mu(\alpha_i^\vee) \le 0.
$$
\end{proof}

So the cone $\mathcal{C}(SL_2\to \wh G)$ has only one regular facet $\mathcal{F}$, with the data $(e,w_0,\alpha^\vee)$, where $w_0$ is the longest element of the Weyl group for $\wh G$ (this makes use of the special phenomenon $w_0\alpha^\vee = -\alpha^\vee$, cf. \cite[Lemma 5.3.1]{BeS}). 

\subsubsection{The rays} Again $\mathfrak{T} = \{\alpha^\vee\}$ and Corollary \ref{yippee} implies that
$\mathcal{C}(SL_2\to \wh G)$ has the $r$ rays $(0,\omega_i)$, where $\omega_i$ is a fundamental weight for $\wh G$, in addition to any rays on $\mathcal{F}$. 

As in the previous example, $\mathcal{F}$ has no type II rays because $L_\delta' = \wh L_\delta' = \{e\}$. Therefore we restrict our attention to the type I rays on $\mathcal{F}$. 

\begin{lemma}\label{peace}
If $v\xrightarrow{\alpha}w_0$ 
and $\ell(s_\beta v) = \ell(v)+1$ for some simple root $\beta$, then 
$\beta=\alpha$.
\end{lemma}

\begin{proof}
Obvious from $s_\alpha v = s_\beta v$, since in this case $s_\beta v$ is forced to be $w_0$ (there is only one element of length $\ell(w_0)$). 
\end{proof}

\begin{proposition}
There are $r$ extremal rays of $\mathcal{F}$. They are $(c_i\omega, \omega_i)$ for $i=1,\hdots, r$. 
\end{proposition}

\begin{proof}
We get a type I ray for each $v\xrightarrow{\alpha}w_0$ with $\alpha$ simple. Of course, for any $\alpha_i$, $s_{\alpha_i}w_0\xrightarrow{\alpha_i}w_0$ since $w_0$ is the longest element, so we do indeed get $r$ rays. The coordinates of ray $i$ are mostly zero by Lemma \ref{peace}, so it is of the form $(C_i\omega, \omega_i)$, where the coefficient $C_i$ is calculated via 
$$
[X_s]\cdot \phi^*\left[\wh X_{s_{\alpha_i}w_0}\right] = C_i[X_e].
$$
By Proposition \ref{lin}, $\left[\wh X_{s_{\alpha_i}w_0}\right]$ is identified (via the Borel isomorphism) with the linear polynomial $-w_0\omega_i$. Therefore its pullback is the linear polynomial $-w_0\omega_i(\alpha^\vee)\omega = \omega_i(\alpha^\vee)\omega = c_i\omega$; this gives $C_i = c_i$.
\end{proof}

\subsubsection{Illustration of Proposition \ref{sloppyeq}} The face $\mathcal{F}$ has $q=r$ type I rays. The kernel of the induction map has rank $c=0$, and $|\Delta(\wh P)| = 0$ while $|\wh \Delta| = r$, so 
$$
0= c = q - |\wh \Delta| + |\Delta(\wh P)| = r-r+0
$$
is satisfied. 

\subsection{A reductive embedding $GL_2\to Sp_4$}

Set $\omega_n  = \left[\begin{array}{cc} 0 & J_n \\ -J_n & 0 \end{array}\right]$, where $J_n$ is the $n\times n$ matrix  $\left[\begin{array}{ccc} 0 & & 1 \\ & \iddots & \\ 1 & & 0 \end{array}\right]$, and consider the associated group $Sp_{2n}$:
$$
Sp_{2n} = \{A\in SL_{2n} | A^t\omega_nA = \omega_n\}
$$

For any $n$, there exists an embedding $GL_n\hookrightarrow Sp_{2n}$ that sends an invertible matrix $A$ to 
$$
\left[
\begin{array}{c|c}
A & 0 \\\hline 
0 & J_nA^{-t}J_n
\end{array}
\right],
$$
(note $J_n^2=I$). 
We will consider the case $n=2$ for a concrete example: 
$$
\left[
\begin{array}{cc}
a & b \\
c & d
\end{array}
\right] 
\mapsto 
\left[
\begin{array}{cccc}
a & b & 0 & 0 \\
c & d & 0 & 0 \\
0 & 0 & \frac{a}{ad-bc} & \frac{-b}{ad-bc} \\
0 & 0 & \frac{-c}{ad-bc} & \frac{d}{ad-bc} 
\end{array}
\right]
$$

Restricted to the standard maximal torus of $GL_2$, the isomorphism $T\to \wh T$ sends 
$$
\op{diag}(a,d)\mapsto \op{diag}(a,d,d^{-1},a^{-1});
$$
furthermore, the standard Borel $B$ of upper-triangular matrices in $GL_2$ is sent to the standard Borel $\wh B$ of upper-triangular matrices in $Sp_4$. 

By way of notation, the positive roots for $Sp_4$ will be $\alpha_1, \alpha_2, \alpha_1+\alpha_2$, and $2\alpha_1+\alpha_2$. The single positive root for $GL_2$ will be $\alpha$, and we define a character on the center of $GL_2$ by 
$$
\Delta: \left[
\begin{array}{cc}
t & 0\\ 
0 & t
\end{array}
\right]\mapsto t^2
$$
($\Delta$ stands for ``determinant''). The natural restriction sends: 
\begin{align*}
\begin{array}{rcr}
\alpha_1\mapsto \alpha & &
\alpha_2\mapsto \Delta-\alpha\\
\alpha_1+\alpha_2\mapsto \Delta& &
2\alpha_1+\alpha_2\mapsto \Delta+\alpha
\end{array}
\end{align*}

The one-parameter subgroups of $T$ have a (rational) basis given by $\alpha^\vee$ and $z$, defined as 
\begin{align*}
\alpha^\vee: t&\mapsto 
\left[
\begin{array}{cc}
t & 0 \\
0 & t^{-1}
\end{array}
\right], &
z: t &\mapsto
\left[
\begin{array}{cc}
t & 0 \\
0 & t
\end{array}
\right]
\end{align*}
($z$ stands for ``center''). Note that $\{\frac{\alpha}{2},\frac{\Delta}{2}\}$ and $\{\alpha^\vee,z\}$ are (rational) dual bases to each other. 
The admissible, dominant, indivisible one-parameter subgroups are 
$$
\frak{S} = \left\{\frac{\alpha^\vee+z}{2},\alpha^\vee,\frac{\alpha^\vee-z}{2}\right\}.
$$

\subsubsection{Inequalities}

Let $\mu = a\omega+b \frac{\Delta}{2}$ and $\wh \mu = c\omega_1+d\omega_2$ be arbitrary weights for $T$, $\wh T$, respectively. We will list the inequalities they must satisfy so that $(\mu, \wh \mu)\in \mathcal{C}(G\hookrightarrow\wh G)$. First of all, they must be dominant weights, which means they satisfy $a\ge 0, c\ge 0, d\ge 0$. 

With $\delta = \dfrac{\alpha^\vee+z}{2}$, we have $P(\delta) = B$ and $\wh P(\delta) = \wh P_2$, the standard parabolic with single negative root $-\alpha_2$. The pullback in cohomology sends 
\begin{align*}
\begin{array}{rr}
[X_{s_1s_2s_1}^{\wh P}] \mapsto [X_s]; & 
 [X_{s_1}^{\wh P}] \mapsto 0~ \\ 
~[X_{s_2s_1}^{\wh P}] \mapsto [X_e]; &
[X_{e}^{\wh P}] \mapsto 0.
\end{array}
\end{align*}

Both nontrivial products survive in the deformed cohomology ring: 
\begin{align*}
\phi_{\delta}^\odot [X_{s_1s_2s_1}^{\wh P_2}] \odot_0 [X_e] &= [X_e]\\
\phi_{\delta}^\odot [X_{s_2s_1}^{\wh P_2}] \odot_0 [X_s] &= [X_e].
\end{align*}

These give the two inequalities 
$$
\frac{1}{2}a+\frac{1}{2} b - c- d\le 0 ~~\text{    and   }~~ -\frac{1}{2}a + \frac{1}{2}b - d \le 0.
$$

With $\delta = \alpha^\vee$, which is not dominant for $\wh B$, we find it convenient to work with $\wh B':=s_2^{-1}\wh Bs_2$ instead of $\wh B$. Let $\gamma_1 = s_2^{-1}\alpha_1 = \alpha_1+\alpha_2$ and $\gamma_2 = s_2^{-1}\alpha_2 = -\alpha_2$; these comprise our new base of the root system $\wh \Phi$. With respect to this base, we let $t_1,t_2$ denote the simple reflections. Then $P(\delta) = B$, and now $\wh P(\delta) = \wh P_{\gamma_1}$, the parabolic subgroup containing $\wh B'$ and the negative root $-\gamma_1$. 
We obtain just one inequality from this $\delta$: 

\begin{center}
\begin{tabular}{c|c}
$(w,\wh w)$ & inequality \\\hline
$(e,t_2t_1t_2)$ & $a-c-2d\le 0$
\end{tabular}
\end{center}

With $\delta = \frac{\alpha^\vee-z}{2}$, which is again not dominant for $\wh B'$, we work instead with $\wh B':=s_2^{-1}s_1^{-1}\wh B s_1s_2$. Our new base consists of $\beta_1 = -\alpha_1-\alpha_2$ and $\beta_2 = 2\alpha_1+\alpha_2$, and we will denote our new simple reflections by $r_1$ and $r_2$. We have $P(\delta) =B$ and $\wh P(\delta) = \wh P_{\beta_2}$ the parabolic subgroup containing $\wh B'$ and the negative root $-\beta_2$. 
We obtain the following inequalities:

\begin{center}
\begin{tabular}{c|c}
$(w,\wh w)$ & inequality \\\hline
$(e,r_1r_2r_1)$ & $\frac{1}{2}a-\frac{1}{2}b-c-d\le 0$ \\
$(s,r_2r_1)$ & $-\frac{1}{2}a-\frac{1}{2}b-d\le 0$
\end{tabular}
\end{center}

\subsubsection{Some rays}

Consider the face $\mathcal{F}_1 = \mathcal{F}(\frac{\alpha^\vee+z}{2},e,s_1s_2s_1)$. There is only one type I datum: $s_2s_1\xrightarrow{\alpha_1}s_1s_2s_1$. Corresponding to $(u,\wh u) = (e,s_2s_1)$, we calculate the type I ray coefficients via 
$$
\phi^*[\wh X_{s_2s_1}]\cdot [X_s] = [X_e]
$$
and 
$$
\phi^*[\wh X_{s_1s_2s_1}]\cdot [X_e] = [X_e],
$$
giving the ray $(\omega+b(\frac{\Delta}{2}), \omega_1) = (1,b,1,0)$, where $b$ is yet to be determined. 

To determine $b$, we use the (approximations of) equivariant classes $[X_e]^T$, $[\wh X_{s_2s_1s_2}]^{\wh T}$ in $H^*_T(G/B)$, $H^*_{\wh T}(\wh G/\wh B)$, respectively. Following Proposition \ref{approxtheory}, we write down 
$$
f_u = P_{s}\otimes 1+1\otimes P_s = \frac{\alpha}{2}\otimes 1 + 1\otimes \frac{\alpha}{2}
$$
and 
$$
\wh f_{\wh u} = P_{s_1}\otimes 1 + 1\otimes P_{s_1} = \left(\frac{2\alpha_1+\alpha_2}{2}\right)\otimes 1 + 1\otimes \left(\frac{2\alpha_1+\alpha_2}{2}\right),
$$
so that 
$$
\wh f_{\wh u}|_T = \left(\frac{\Delta+\alpha}{2}\right)\otimes 1 + 1\otimes \left(\frac{\Delta+\alpha}{2}\right).
$$
We calculate $\int f_u\cdot \wh f_{\wh u}|_T = \Delta+\alpha$, so $b$ is found by solving 
$$
\Delta+\alpha = \frac{\alpha}{2}+b\left(\frac{\Delta}{2}\right) +\frac{\Delta+\alpha}{2};
$$
i.e., $b = 1$. 

\begin{remark}
Observe that, in this case, we can find $b$ another way. Since the ray $(1,b,1,0)$ is supposed to lie on the face $\frac{1}{2}a+\frac{1}{2}b - c - d = 0$, we find $b$ must equal $1$. In more general situations, however, the dimension of $Z(G)$ may exceed $1$, so the single face equation can't fully determine the character $\chi$. See the next subsection for an example. 
\end{remark}

The pair $L_\delta \subset \wh L_\delta$ are isomorphic to $\C^* \subset SL_2$, and the sole fundamental weight $\omega^L$ for this $SL_2$ expressed in our notation above as a character of $\wh T/\langle \delta\rangle$ is $\omega_2-\omega_1$. As a character of $T/\langle \delta\rangle$, this restricts to $\frac{\Delta}{2}-\omega$. The rays of $\mathcal{C}(L_\delta\subset\wh L_\delta)$ are generated by $(\omega^L,\omega^L)$ and $(-\omega^L,\omega^L)$. They map to 
\begin{align*}
\left(e\left(\frac{\Delta}{2}-\omega\right),s_1s_2s_1(\omega_2-\omega_1)\right) &- (s_1s_2s_1(\omega_2-\omega_1))(\alpha_1^\vee)\left(\omega+\frac{\Delta}{2},\omega_1\right)\\
&=\left(\frac{\Delta}{2}-\omega,\omega_2-\omega_1\right)+\left(\omega+\frac{\Delta}{2},\omega_1\right)\\
&=(\Delta,\omega_2)
\end{align*}
and $(2\omega,\omega_2)$, respectively. In coordinates, these are the rays $(0,2,0,1)$ and $(2,0,0,1)$, which do indeed lie on $\mathcal{F}_1$. 

For another example, consider the face $\mathcal{F}_2 = \mathcal{F}(\frac{\alpha^\vee+z}{2},s,s_2s_1)$. The type I datum $e\xrightarrow{\alpha}s$ yields the same $(u,\wh u)$ pair and therefore the same ray as above. So instead consider the type I datum $s_1\xrightarrow{\alpha_2}s_2s_1$, giving the pair $(u,\wh u) = (s,s_1)$. The type I ray coefficients of Theorem \ref{formulaONE}(a) give the ray $(0,b,0,1)$, where $b$ is yet to be determined. 

Following Proposition \ref{approxtheory}, we have 
$$
f_s = 1\otimes 1
$$
and 
\begin{align*}
\wh f_{s_1s_2} &= P_{e}\otimes P_{s_2s_1} + P_{s_2}\otimes P_{s_1} + P_{s_1s_2}\otimes P_e\\ &= 1\otimes \frac{(2\alpha_1+\alpha_2)^2}{4} + (\alpha_1+\alpha_2)\otimes \left(\frac{2\alpha_1+\alpha_2}{2}\right) + \frac{(\alpha_1+\alpha_2)^2}{2}\otimes 1,
\end{align*}
so 
$$
\wh f_{s_1s_2}|_T = 1\otimes \frac{(\alpha+\Delta)^2}{4} + \Delta\otimes \frac{\alpha+\Delta}{2} + \frac{\Delta^2}{2}\otimes 1.
$$
We find that $\int \wh f_{s_1s_2} = 2\Delta$. The equation 
$$
2\Delta = b\left(\frac{\Delta}{2}\right) + \omega_2|_T = b\left(\frac{\Delta}{2}\right) + \Delta
$$
gives $b = 2$. 

The pair $L_\delta\subset \wh L_\delta$ remains the same as before. The rays $(\omega^L, \omega^L)$ and $(-\omega^L, \omega^L)$ are sent to 
\begin{align*}
\left(s\left(\frac{\Delta}{2}-\omega\right),s_2s_1(\omega_2-\omega_1)\right) &- s\left(\frac{\Delta}{2}-\omega\right)(\alpha^\vee)\left(\omega+\frac{\Delta}{2},\omega_1\right)-0\\
&=\left(\frac{\Delta}{2}+\omega,2\omega_1\right)-\left(\omega+\frac{\Delta}{2},\omega_1\right)\\
&=(0,\omega_1)
\end{align*}
and $(0,3\omega_1)$, respectively. Up to scaling, these both give the same ray, $(0,0,1,0)$ in coordinates. 

\subsubsection{Illustration of Proposition \ref{sloppyeq}}

On the face $\mathcal{F}_1$ above, we had $c = 0$ since $\op{Ind}$ sent a basis to a linearly independent set. Moreover, we had $q=1$ and $|\wh \Delta| - |\Delta(\wh P)| = 1$. 
On the face $\mathcal{F}_2$, we had $c=1$ since the two extremal rays mapped to scalar multiples of one another. This agrees with $q=2$ and $|\wh \Delta| - |\Delta(\wh P)| = 1$. 

\subsection{The maximal torus embedding $T\subset G$}

Let $G$ be an arbitrary semisimple group and $T$ a fixed maximal torus inside $G$. Choose a Borel subgroup $B$ such that $T\subset B\subset G$. The cone $\mathcal{C}(T\subset G)$ (sometimes called the \emph{Kostka cone}) consists of pairs $(\mu,\lambda)$ such that (after scaling) $-\mu$ is a nontrivial weight space in the $G$-representation $V_{\lambda}$. It is well-known (e.g., follows from \cite[Proposition 21.3]{Humphreys}) that this occurs if and only if $-\mu$ is contained in the convex hull of the set $W\lambda = \{w\lambda|w\in W\}$. From this one can deduce the rays of the associated cone. 

\begin{proposition}\label{easyKostka}
The extremal rays of $\mathcal{C}(T\subset G)$ are generated by the pairs of the form $(-w \omega_j, \omega_j)$ as $w$ varies in $W$ and $\omega_j$ ranges over the set of fundamental weights. 
\end{proposition}

\subsubsection{Inequalities}

The extremal rays of $\mathcal{C}(T\subset G)$ are all type I rays on some face (and type II rays on several faces). First, let us describe the faces; cf. \cite[\S 5.1]{BeS}.  

Let $\{x_i\}$ denote a set of dominant cocharacters satisfying $\alpha_j(x_i) = n_i\delta_{i,j}$ for some integers $n_i>0$. That is, $x_i$ is a positive scalar multiple of the $i^\text{th}$ fundamental coweight, and we may assume this multiple to be as small as possible. The admissible dominant indivisible one-parameter subgroups are the collection 
$
\mathfrak{S} = \{wx_i | w\in W\}.
$
The OPS $wx_i$ is dominant for $wBw^{-1}$ and not $B$ (unless $w=e$), so we change basis using $w^{-1}$ if needed. 

The map $\phi_{wx_i}$ is the inclusion of the basepoint $eP({wx_i})$
$$
\{pt\} = T/T \hookrightarrow G/P({wx_i}).
$$
The only cohomology class pulling back nontrivially is $[X_{w_0w^{-1}}]$. This gives the inequality 
$$
\mu(wx_i)+\lambda(w_0x_i)\le 0,
$$
which, finding $j$ such that $x_j = -w_0x_i$, and setting $v = ww_0$, becomes the more familiar inequality 
$
(-\mu)(vx_j)\le \lambda(x_j).
$

\subsubsection{The rays}
To avoid overwhelming notation, we will write our calculations assuming $w=e$ and then ``change bases'' back to report the extremal rays. 

Fix a face $(x_i, e, w_0)$. 
The minimal length representative for $w_0$ in $W/W_P$ is $w_0w_0^P$, where $w_0^P$ is the longest element of $W_P$. 
We claim there is only one type I ray datum on this face. It corresponds to $s_j$ where $\alpha_j = -w_0\alpha_i$. 
To see this, we have $s_kw_0w_0^P\to w_0w_0^P$ if and only if $w_0^Pw_0(\alpha_k)\prec 0$, so $w_0^P(-w_0\alpha_k)\succ 0$. Now, $-w_0\alpha_k$ is a simple root. If it belongs to $\Phi_L$, then $w_0^P$ sends it to $\Phi^-$. Thus it does not belong to $\Phi_L$. There is only one such simple root: $\alpha_i$. Thus $\alpha_k = \alpha_j$. 

Let $(\mu, \lambda)$ denote the corresponding type I ray. 
If $s_jw_0w_0^P\to s_ks_jw_0w_0^P$ and $s_ks_jw_0w_0^P$ belongs to $W^P$, then $s_ks_jw_0w_0^P = w_0w_0^P$, as there is a unique element of $W^P$ of maximal length. Thus there is only one nonzero coefficient in $\lambda$; $\lambda = \omega_j$. To determine $\mu$, we use Theorem \ref{formulaONE}(b) and Proposition \ref{approxtheory}.  For $\wh u = s_jw_0$,
$$
f_{\wh u} = 1\otimes P_{s_i} + P_{s_j}\otimes 1 = 1\otimes \omega_i + \omega_j\otimes 1,
$$
which restricts and ``integrates'' to $\omega_i+\omega_j$. Therefore $\mu = \omega_i = -w_0\omega_j$. 

Changing bases back again, this is the ray 
$
(w\omega_i, \omega_j),
$
which indeed satisfies the face equality 
$$
w\omega_i(wx_i) + \omega_j(w_0x_i) = \omega_i(x_i)+w_0\omega_j(x_i) = \omega_i(x_i)-\omega_i(x_i) = 0.
$$
Our ray can also be written $(-ww_0\omega_j, \omega_j)$, complying with Proposition \ref{easyKostka}. Moreover, as $w$ and $j$ vary, we produce all rays of $\mathcal{C}(T\subset G)$. 

The face $(x_i, e, w_0)$ has several other rays induced from the Levi pair $T/\langle x_i\rangle \subset L(x_i)/\langle x_i\rangle$. This cone's extremal rays are of the form 
$$
(-v\omega_k^L, \omega_k^L),
$$
where $v\in W_P$ and $\omega_k^L\in X^*(T/\langle x_i\rangle)$ pairs to $1$ with $\alpha_k^\vee$, $0$ with $\alpha_\ell^\vee$ for $\ell\ne k$ but $\alpha_\ell\in \Phi^P$, and pairs to $0$ with $x_i$. 
Such an extremal ray is induced to 
$$
(\nu, \wh\nu) = (-v\omega_k^L, w_0w_0^P\omega_k^L) - w_0w_0^P \omega_k^L(\alpha_j^\vee) (-w_0\omega_j, \omega_j).
$$
Let us try to simplify these formulas for $\nu, \wh \nu$. 
First, let $\alpha_m = -w_0^P\alpha_k$. We guess that $\wh \nu = -w_0\omega_m$. This can be verified by the following pairings. First, clearly $\wh \nu(-w_0\alpha_i^\vee) = \wh \nu(\alpha_j^\vee) =0$. For $\ell\ne i$, 
$$
\wh \nu(-w_0\alpha_\ell^\vee) = w_0w_0^P\omega_k^L(-w_0\alpha_\ell^\vee) - 0 = \omega_k^L(-w_0^P\alpha_\ell^\vee) = \omega_k(-w_0^P\alpha_\ell^\vee),
$$
which equals $0$ unless $\ell = m$. 

Secondly, we guess that $\nu = vw_0^P\omega_m$. Once again, consider the following pairings. The set $\{-v\alpha_\ell\} = -v\Delta(L)$ forms a base for the root system $\Phi_L$. We have 
$$
\nu(-v\alpha_\ell^\vee) = -v\omega_k^L(-v\alpha_\ell^\vee) - 0 = \omega_k(\alpha_\ell^\vee) = \delta_{k,\ell};
$$
on the other hand, 
$$
vw_0^P\omega_m(-v\alpha_\ell^\vee) = \omega_m(-w_0^P\alpha_\ell^\vee),
$$
which equals $0$ unless $\ell = k$. 
Now pair with $x_i$. By lying on the face, we know that 
$
\nu(x_i)+\wh \nu(w_0x_i)= 0
$,
which implies that $\nu(x_i) = w_0\omega_m(w_0x_i) = \omega_m(x_i)$. Moreover, 
$$
vw_0^P\omega_m(x_i) = \omega_m(x_i)
$$
since $vw_0^P\in W_P$, which stabilizes $x_i$. The simple coroots for any base of $\Phi_L$ together with $x_i$ forms a basis of $\mathfrak{h}$. Therefore $\nu = vw_0^P\omega_m$. 

To change bases back (if applicable), these type II rays are really 
$$
(wvw_0^P \omega_m, -w_0\omega_m);
$$
once again these comply with Proposition \ref{easyKostka}. 

\subsubsection{Illustration of Proposition \ref{sloppyeq}}

Fix a face $(x_i,e,w_0)$ as before. There is one type I ray on this face, so $q=1$. Furthermore, $|\Delta| - |\Delta(P)| = 1$ Therefore we expect $c = 1-1 = 0$. We can check this directly. 

Indeed, suppose $(x,y)$ maps to $0$ under $\op{Ind}$. Express $y = \sum n_k\omega_k^L$ in the basis of fundamental weights for $L/\langle x_i\rangle$. As we calculated above, the second coordinate of $\op{Ind}(x,y)$ is equal to $\sum n_k -w_0\omega_{m(k)}$. If this is $0$, each $n_k$ is forced to be $0$ by linear independence of the fundamental weights for $G$. So $y=0$. But $(x,0)\mapsto (x,0)$, so $x=0$ as well. 

\subsection{The natural embedding $Sp_{2n} \to SL_{2n}$, $n=2,3$}

It is a standard fact that, if $A$ is an invertible linear operator on a vector space $V$ of dimension $2n$ equipped with a symplectic form, and if $A$ preserves the form, then $A$ has determinant $1$. Therefore we have a natural embedding $Sp_{2n}\to SL_{2n}$ for any $n\ge 1$. 

In order to fix notation, we recall a particular description of this embedding from \cite[\S 8]{PR}. Once again set $\omega_n  = \left(\begin{array}{cc} 0 & J_n \\ -J_n & 0 \end{array}\right)$, where $J_n = \left(\begin{array}{ccc} & & 1 \\ & \iddots & \\ 1 & & \end{array}\right)$. The associated group $Sp_{2n}$ is 
$$
Sp_{2n} = \{A\in SL_{2n} | A^t\omega_nA = \omega_n\}
$$
We choose maximal torus and Borel subgroup $T\subset B\subset Sp_{2n}$ to be the subgroups of diagonal and upper-triangular matrices, respectively; i.e., $T = Sp_{2n}\cap \wh T$ and $B = Sp_{2n}\cap \wh B$, where $\wh T\subset \wh B$ are the standard maximal torus and Borel of $SL_{2n}$. Explicitly, 
$$
T = \{\op{diag}(t_1,t_2,\hdots,t_n,t_n^{-1},\hdots,t_2^{-1},t_1^{-1})\}; 
$$
furthermore, a one-parameter subgroup $t\mapsto \op{diag}(t^{a_1},\hdots,t^{a_n},t^{-a_n},\hdots,t^{-a_1})$ is dominant \wrt $B$ if and only if $a_1\ge a_2\ge\hdots \ge a_n\ge 0$. Notably, dominant one-parameter subgroups are also dominant \wrt $\wh B$, so no change of basis (as in Section \ref{changingtime}) is ever necessary. 

\subsubsection{The regular facets}

The set $\mathfrak{S}$ consists of $\delta_j$, for $j = 1,\hdots, n-2,$ or $n$, where 
$$
\delta_j: t\mapsto \op{diag}(\underbrace{t,t,\hdots, t}_{j},1,\hdots, 1, t^{-1},\hdots, t^{-1},t^{-1})\in T.
$$

Each $P(\delta_j)$ is a maximal parabolic (obtained by removing the $j^\text{th}$ simple root), whereas $\wh P(\delta_j)$ has base $\wh \Delta\setminus\{\wh\alpha_j, \wh \alpha_{2n-j}\}$ for each $j<n-1$, and $\wh P(\delta_n)$ is the maximal parabolic with associated Grassmannian $\op{Gr}(n,2n)$. 

Before listing some specific results for the cases $n=2,3$, we answer in the affirmative a question posed by the reviewer on the nature of the extremal rays of $\mathcal{C}(Sp_{2n}\to Sl_{2n})$. 

\begin{proposition}
Every extremal ray of $\mathcal{C}(Sp_{2n}\to Sl_{2n})$ lies on a regular face. 
\end{proposition}

\begin{proof}
The only extremal rays possibly not on a regular face are of the form $(0,\wh\omega_j)$. In fact, we will show these all lie on the same regular face (or do not belong to the cone, as witnessed by this same face). 

Take $\delta = \delta_n$ and 
$\wh w =$ $(\wh s_2\wh s_4\cdots \wh s_{2n-2})\cdots (\wh s_{n-2}\wh s_n\wh s_{n+2})(\wh s_{n-1}\wh s_{n+1})\wh s_n$. In more traditional notation of Schubert varieties for $\op{Gr}(n,2n)$, with $F_\bullet$ a fixed full flag, this gives 
$$
[\wh X_{\wh w}] = [\wh X_{a_\bullet}] = [\{V\in \op{Gr}(n,2n): \dim V\cap F_{a_i}\ge i~\forall i\}]
$$
where $a_{\bullet} = \{1,3,\hdots,2n-1\}$ is the $n$-element subset of $\{1,\hdots,2n\}$ given by $a_i = 2i-1$. 

By \cite[Lemma 4.19]{Coskun}, $\phi_\delta^*([\wh X_{\wh w}]) = [X_e]$. Let us verify that in fact $\phi_\delta^{\odot}[\wh X_{\wh w}] = [X_e]$. One easily checks that $\langle 2\rho, \dot\delta\rangle = \langle 2\rho, 2x_n\rangle = n^2+n$ and that $\langle \wh \rho, \dot\delta\rangle = \langle \wh \rho, 2\wh x_n\rangle = n^2$. Furthermore, with the notation of \cite{Bourbaki}, we identify $\wh w\dot\delta$ with 
$$
\wh w\dot \delta = 2(\epsilon_1+\epsilon_3+\hdots+\epsilon_{2n-1}) - \sum_{i=1}^{2n}\epsilon_i = \epsilon_1-\epsilon_2+\epsilon_3-\hdots+\epsilon_{2n-1}-\epsilon_{2n}. 
$$
It is then straightforward to compute the pairings
$$
\wh\omega_j(\wh w\dot\delta) = (\epsilon_1+\hdots+\epsilon_j)(\wh w\dot\delta) = \left\{
\begin{array}{cc}
1, & j \text{ odd}\\
0, & j \text{ even},
\end{array}
\right.
$$
which give $\langle \wh \rho, \wh w \dot\delta\rangle = n$. Therefore $\langle 2\rho, \dot\delta\rangle - \langle \wh \rho +\wh w^{-1}\wh \rho, \dot\delta\rangle = 0$ is satisfied. 

These latter pairings also establish that the inequality 
$$
\wh \omega_j(\wh w\dot\delta)\le 0
$$
holds with equality for $j$ even and fails for $j$ odd. Thus every ray of the form $(0,\wh\omega_j)$ is on a regular face. 
\end{proof}

For $n=2$, we obtain $5$ inequalities, hence $5$ faces, all from the single one-parameter subgroup $\delta_2$. We detail this below. 

For $n=3$, we obtain $24$ inequalities: $9$ coming from $\delta_1$ and $15$ from $\delta_3$. 

\subsubsection{Case $n=2$}

Below are listed the $5$ inequalities along with the Weyl group data from which they arise. Here $\mu = a_1\omega_1+a_2\omega_2$ and $\wh \mu = b_1\wh\omega_1+b_2\wh\omega_2+b_3\wh\omega_3$ are arbitrary dominant weights. The cohomology calculations were performed using {\tt Sage} \cite{sage} using a modification of the main algorithm in \cite{Kiers}. These results agree with those of \cite[\S 8.8]{PR}, although they write their inequalities in a different basis. 

$ $

\begin{center}
\begin{tabular}{c|c}
$(w,\wh w)$ & inequality \\\hline
$(s_2s_1s_2,\wh s_2)$ & $-a_1-2a_2+b_1+b_3 \le 0$ \\
$(s_1s_2,\wh s_1\wh s_2)$ & $-a_1-b_1+b_3\le 0$ \\
$(s_1s_2, \wh s_3\wh s_2)$ & $-a_1+b_1-b_3\le 0$\\
$(s_2, \wh s_3\wh s_1\wh s_2)$ & $a_1-b_1-b_3\le 0$\\
$(e,\wh s_2\wh s_3\wh s_1\wh s_2)$ & $a_1+2a_2-b_1-2b_2-b_3\le 0$
\end{tabular}
\end{center}

$ $

Let us show how Theorem \ref{hopeful} precludes $(0,0,1,0,0)$ from being an extremal ray. Our set $\mathfrak{T}$ now contains strictly more than $\mathfrak{S}$; in particular, $\delta_1\in \mathfrak{T}\setminus \mathfrak{S}$. Furthermore, the pullback	
$$	
\phi_{\delta_1}^\odot [\wh X_{\wh s_2\wh s_3}] = [X_e]	
$$	
and $\wh X_{\wh s_2\wh s_3} = \wh X_{\wh s_2\wh s_3\wh s_2} = \wh X_{\wh s_3\wh s_2\wh s_3}$ and $\wh X_{\wh s_2\wh s_3}\not\subseteq \wh X_{\wh s_3}$, so $(\wh s_2\wh s_3,\delta_1)\in S_1$. Since 	
$$	
\wh \omega_1(\wh s_2\wh s_3 \delta_1) = 1 \not \le 0,	
$$	
the candidate $(0,0,1,0,0)$ is not a ray. Notably, we expect the inequality for the data $(s_1s_2s_1,\wh s_2\wh s_3,\delta_1)$ to be redundant, and indeed it is 	
$$	
-a_1-a_2+b_1\le 0,	
$$	
which is half the sum of the inequalities $-a_1-2a_2+b_1+b_3\le 0$ and $-a_1+b_1-b_3\le 0$ from the table above. 	
A similar analysis shows $(0,0,0,0,1)$ can't be a ray and that $(0,0,0,1,0)$ must be. 

Now let us find, for example, the extremal rays on the face $\mathcal{F}$ given by the pair $(s_2s_1s_2,\wh s_2)$. From the datum $s_1s_2\xrightarrow{\alpha_2}s_2s_1s_2$, we obtain the $(u, \wh u)$ pair $(s_1s_2,\wh s_2)$. We have $s_1s_2\not\rightarrow s_2$, $s_1s_2\rightarrow s_2s_1s_2$ in $W^P$ and $\wh s_2\rightarrow \wh s_1\wh s_2$, $\wh s_2\not\rightarrow e$, $\wh s_2\rightarrow \wh s_3\wh s_2$ in $\wh W^{\wh P}$. The $a_1$ and $b_2$ coordinates are therefore $0$, and the others are calculated in cohomology: 
\begin{align*}
\phi_{\delta_2}^*\left[\wh X_{\wh s_2}\right]\cdot [X_{s_2s_1s_2}] &= 1[X_e] = a_2[X_e]\\
\phi_{\delta_2}^*\left[\wh X_{\wh s_1\wh s_2}\right]\cdot [X_{s_1s_2}] &= 1[X_e] = b_1[X_e]\\
\phi_{\delta_2}^*\left[\wh X_{\wh s_3\wh s_2}\right]\cdot [X_{s_1s_2}] &= 1[X_e] = b_3[X_e],
\end{align*}
so the datum $s_1s_2\xrightarrow{\alpha_2}s_2s_1s_2$ yields the extremal ray $(0,1,1,0,1)$. 

From the datum $e\xrightarrow{\wh\alpha_2} \wh s_2$, we obtain the $(u,\wh u)$ pair $(s_2s_1s_2,e)$. Although $s_2s_1s_2\xrightarrow{\alpha_1}s_1s_2s_1s_2$ in the Bruhat order, $s_1s_2s_1s_2 = s_2s_1s_2s_1$ is not a minimal-length representative in $W/W_P$, so the $a_1$ coordinate is $0$. Of course $s_2s_1s_2\not\rightarrow s_1s_2$, so $a_2 = 0$ as well. Neither of $\wh s_1, \wh s_3$ is a minimal-length representative in $\wh W/\wh W_{\wh P}$, so $b_1=b_3=0$. We have (from the original deformed cup product) $b_2=1$, so the extremal ray is $(0,0,0,1,0)$. 

These are the only two type I rays on $\mathcal{F}$; note that they are linearly independent according to Lemmas  \ref{ones} and \ref{zeros}. We expect (for dimension reasons) at least $2$ type II rays; let us calculate these.

Here $L_\delta$ is semisimple of type $A_1$ and $\wh L_\delta$ of type $A_1\times A_1$, and the embedding is diagonal. The extremal rays of $\mathcal{C}(L_\delta\to \wh L_\delta)$ are well-known: $(\omega, \omega^1)$, $(\omega, \omega^2)$, and $(0,\omega^1+\omega^2)$, where $\omega^i$ is the fundamental weight for the $i^\text{th}$ factor of $SL_2$ in $\wh L_\delta$. As elements of $\mathfrak{h}^*\times \wh{\mathfrak{h}}^*$ which vanish on $\dot\delta$, these are 
$$
(\omega_1-\frac{1}{2}\omega_2,\wh \omega_1-\frac{1}{2}\wh\omega_2),~~ (\omega_1-\frac{1}{2}\omega_2,\wh\omega_3-\frac{1}{2}\wh\omega_2),~~ (0,\wh\omega_1+\wh\omega_3-\wh\omega_2),
$$
respectively. 

These map to 
\begin{align*}
\left(\omega_1-\frac{1}{2}\omega_2,\frac{1}{2}\wh\omega_1+\frac{1}{2}\wh\omega_2-\frac{1}{2}\wh\omega_3\right) &- \left(\omega_1-\frac{1}{2}\omega_2\right)(\alpha_2^\vee)(0,1,1,0,1) \\
&- \left(\frac{1}{2}\wh\omega_1+\frac{1}{2}\wh\omega_2-\frac{1}{2}\wh\omega_3\right)(\wh\alpha_2^\vee)(0,0,0,1,0)\\&= (1,0,1,0,0),\\
\left(\omega_1-\frac{1}{2}\omega_2,-\frac{1}{2}\wh\omega_1+\frac{1}{2}\wh\omega_2+\frac{1}{2}\wh\omega_3\right) &- \left(\omega_1-\frac{1}{2}\omega_2\right)(\alpha_2^\vee)(0,1,1,0,1) \\
&- \left(-\frac{1}{2}\wh\omega_1+\frac{1}{2}\wh\omega_2+\frac{1}{2}\wh\omega_3\right)(\wh\alpha_2^\vee)(0,0,0,1,0)\\&= (1,0,0,0,1),\\
\left(0,\wh\omega_2\right) - \left(0\right)(\alpha_2^\vee)(0,1,1,0,1) 
&- \left(\wh\omega_2\right)(\wh\alpha_2^\vee)(0,0,0,1,0)= (0,0,0,0,0),
\end{align*}
respectively, under $\op{Ind}$. The first two of these really are extremal rays, but notice that the kernel of $\op{Ind}$ is nontrivial in this case. Note also that, by the symmetry of $\mathcal{F}$ under the Dynkin automorphism of $A_3$, the sets of type I rays and type II rays are invariant under this automorphism (switching indices $1$ and $3$).

Following is a complete list of the extremal rays of the cone $\mathcal{C}(Sp(2)\to SL(4))$ (cf. \cite[\S 8.8]{PR}):

$$
\begin{array}{cc}
(0,1,1,0,1), & (0,0,0,1,0), \\
(1,0,1,0,0), & (1,0,0,0,1), \\
(0,1,0,1,0); & 
\end{array}
$$
furthermore, these constitute the Hilbert basis of the semigroup (so the cone is ``saturated,'' see \cite{PR} for the development of this notion as well as several examples). Interestingly, all of these are type I on some facet.

The kernel of $\op{Ind}$ has rank $c=1$, and we observe that 
$$
1 = c = q-|\wh \Delta| + |\Delta(\wh P)| = 2 -3 + 2,
$$
illustrating again Proposition \ref{sloppyeq}. 

\subsubsection{Case $n=3$}

Below are the $24$ inequalities and extremal rays expressed in the fundamental weight basis: $\mu = a_1\omega_1+a_2\omega_2+a_3\omega_3$; $\wh \mu = b_1\wh\omega_1+b_2\wh\omega_2+b_3\wh\omega_3+b_4\wh\omega_4+b_5\wh\omega_5$. All calculations were done in {\tt Sage}. See \cite[\S 8.9]{PR} for the same results (but expressed in a different basis).

Inequalities coming from the one-parameter subgroup $\delta_1$: 

$ $

\begin{center}
\begin{tabular}{c|c}
$(w,\wh w)$ & inequality \\\hline
$(s_1s_2s_3s_2s_1,\wh s_4\wh s_3\wh s_2\wh s_1)$ & $-a_1-a_2-a_3+b_5 \le 0$ \\
$(s_1s_2s_3s_2s_1,\wh s_4\wh s_5\wh s_2\wh s_1)$ & $-a_1-a_2-a_3+b_3 \le 0$ \\
$(s_1s_2s_3s_2s_1,\wh s_2\wh s_3\wh s_4\wh s_5)$ & $-a_1-a_2-a_3+b_1 \le 0$\\
$(s_2s_1,\wh s_1\wh s_2\wh s_3\wh s_4\wh s_5\wh s_2\wh s_1)$ & $a_3-b_1-b_2-b_3\le 0$ \\
$(s_2s_1,\wh s_2\wh s_3\wh s_4\wh s_5\wh s_3\wh s_2\wh s_1)$ & $a_3-b_2-b_3-b_4\le 0$ \\
$(s_2s_1,\wh s_3\wh s_4\wh s_5\wh s_4\wh s_3\wh s_2\wh s_1)$ & $a_3-b_3-b_4-b_5\le 0$ \\
$(s_1, \wh s_1\wh s_2\wh s_3\wh s_4\wh s_5\wh s_3\wh s_2\wh s_1)$ & $a_2+a_3-b_1-b_2-b_3-b_4\le 0$\\
$(s_1, \wh s_2\wh s_3\wh s_4\wh s_5\wh s_4\wh s_3\wh s_2\wh s_1)$ & $a_2+a_3-b_2-b_3-b_4-b_5\le 0$\\
$(e, \wh s_1\wh s_2\wh s_3\wh s_4\wh s_5\wh s_4\wh s_3\wh s_2\wh s_1)$ & $a_1+a_2+a_3-b_1-b_2-b_3-b_4-b_5\le 0$
\end{tabular}
\end{center}

$ $

Inequalities coming from the one-parameter subgroup $\delta_3$: 

$ $ 

\begin{center}
\begin{tabular}{c|c}
$(w,\wh w)$ & inequality \\\hline
$(s_3s_2s_3s_1s_2s_3,\wh s_4\wh s_2\wh s_3)$ & $-a_1-2a_2-3a_3+b_1+b_3+b_5\le 0$\\
$(s_2s_3s_1s_2s_3,\wh s_4\wh s_1\wh s_2\wh s_3)$ & $-a_1-2a_2-a_3-b_1+b_3+b_5\le 0$\\
$(s_2s_3s_1s_2s_3,\wh s_3\wh s_4\wh s_2\wh s_3)$ & $-a_1-2a_2-a_3+b_1-b_3+b_5\le 0$\\
$(s_2s_3s_1s_2s_3,\wh s_5\wh s_4\wh s_2\wh s_3)$ & $-a_1-2a_2-a_3+b_1+b_3-b_5\le 0$\\
$(s_3s_1s_2s_3,\wh s_3\wh s_4\wh s_1\wh s_2\wh s_3)$ & $-a_1-a_3-b_1-b_3+b_5\le 0$\\
$(s_3s_1s_2s_3,\wh s_5\wh s_4\wh s_1\wh s_2\wh s_3)$ & $-a_1-a_3-b_1+b_3-b_5\le 0$\\
$(s_3s_1s_2s_3,\wh s_5\wh s_3\wh s_4\wh s_2\wh s_3)$ & $-a_1-a_3+b_1-b_3-b_5\le 0$\\
$(s_3s_2s_3,\wh s_5\wh s_3\wh s_4\wh s_1\wh s_2\wh s_3)$ & $a_1-a_3-b_1-b_3-b_5\le 0$\\
$(s_1s_2s_3,\wh s_2\wh s_3\wh s_4\wh s_1\wh s_2\wh s_3)$ & $-a_1+a_3-b_1-2b_2-b_3+b_5\le 0$\\
$(s_1s_2s_3,\wh s_5\wh s_3\wh s_4\wh s_1\wh s_2\wh s_3)$ & $-a_1+a_3-b_1-b_3-b_5\le 0$\\
$(s_1s_2s_3,\wh s_4\wh s_5\wh s_3\wh s_4\wh s_2\wh s_3)$ & $-a_1+a_3+b_1-b_3-2b_4-b_5\le 0$\\
$(s_2s_3,\wh s_5\wh s_2\wh s_3\wh s_4\wh s_1\wh s_2\wh s_3)$ & $a_1+a_3-b_1-2b_2-b_3-b_5\le 0$\\
$(s_2s_3,\wh s_4\wh s_5\wh s_3\wh s_4\wh s_1\wh s_2\wh s_3)$ & $a_1+a_3-b_1-b_3-2b_4-b_5\le 0$\\
$(s_3,\wh s_4\wh s_5\wh s_2\wh s_3\wh s_4\wh s_1\wh s_2\wh s_3)$ & $a_1+2a_2+a_3-b_1-2b_2-b_3-2b_4-b_5\le 0$\\
$(e,\wh s_3\wh s_4\wh s_5\wh s_2\wh s_3\wh s_4\wh s_1\wh s_2\wh s_3)$ & $a_1+2a_2+3a_3-b_1-2b_2-3b_3-2b_4-b_5\le 0$
\end{tabular}
\end{center}

$ $ 

The $15$ extremal rays:

$$
\begin{array}{ccc}
(1, 0, 0, 0, 0, 0, 0, 1), & (0, 0, 0, 0, 0, 0, 1, 0), & (1, 0, 0, 0, 0, 1, 0, 0), \\
 (0, 0, 0, 0, 1, 0, 0, 0), & (1, 0, 0, 1, 0, 0, 0, 0), & (0, 1, 0, 0, 0, 0, 1, 0), \\
  (0, 0, 1, 1, 0, 0, 1, 0), & (0, 0, 1, 0, 0, 1, 0, 0), & (0, 1, 0, 1, 0, 0, 0, 1), \\
   (0, 0, 1, 0, 1, 0, 0, 1), & (0, 1, 0, 0, 1, 0, 0, 0), & (0, 0, 1, 1, 0, 1, 0, 1), \\
    (0, 1, 0, 0, 0, 1, 0, 1), & (0, 1, 0, 1, 0, 1, 0, 0), & (1, 0, 1, 0, 1, 0, 1, 0)
\end{array}
$$

All extremal rays are type I on some face. The various maps $\op{Ind}$ send some extremal rays to $\vec 0$ or to non-extremal rays, such as $(1,0,1,0,1,1,0,1)$.

\begin{comment}

\section{Junk}

\begin{lemma}
Suppose $N\subset L$ is a normal subgroup. Then $N\ltimes U$ is a normal subgroup of $P$.
\end{lemma}

\begin{proof}
Recalling the presentation $P = L\ltimes U$, we must check that 
$$
(l,u)\cdot(n,u')\cdot(l,u)^{-1}
$$
can be expressed in the form $(n',u'')$ for any $l\in L, n\in N, u,u'\in U$. Indeed 
\begin{align*}
lunu'u^{-1}l^{-1}&=lnn^{-1}unu'u^{-1}l^{-1}\\
&=ln\tilde u l^{-1} \text{~~for some $\tilde u\in U$}\\
&=lnl^{-1}u'' \text{~~ for $u'' = l\tilde u l^{-1}\in U$}\\
&=n'u'',
\end{align*}
where $n' = lnl^{-1}\in N$ by normality of $N$ in $L$. 
\end{proof}
\end{comment}

\vspace{0.05in}

\noindent
Department of Mathematics, Ohio State University, 281 W Lane Ave, OH 43201\\
{{email: kiers.2@osu.edu}}

\end{document}